\documentclass[9pt]{amsart}
\usepackage{amssymb,latexsym,amsmath}
\newtheorem{theorem}{Theorem}[section]
\newtheorem{cor}{Corollary}[section]
\newtheorem{prop}{Proposition}[section]
\newtheorem{lemma}{Lemma}[section]
\newtheorem{definition}{Definition}[section]

\theoremstyle{remark}
\newtheorem{remark}{Remark}[section]
\numberwithin{equation}{section}
\def\eqn#1$$#2$${\begin{equation}\label#1#2\end{equation}}

\newcommand{\CC}{\Subset}
\def\er{\mathbb R}
\def\T{\mathcal T}
\newcommand{\ratio}{L/\nu}
\newcommand{\divo}{\textnormal{div}_{H}}

\def\eqn#1$$#2$${\begin{equation}\label#1#2\end{equation}}

\newcommand{\M}{M^*}
\newcommand{\Mq}{\M_{q_0/p}}

\newcommand{\Ds}{D_{h}^{X_s}}

\newcommand{\Dsy}{D_{h}^{Y_s}}

\newcommand{\rif}[1]{(\ref{#1})}
\newcommand{\trif}[1]{\textnormal{(\ref{#1})}}

\newcommand{\dist}{\operatorname{dist}}

\newcommand{\R}{\mathbb R}

\newcommand{\Om}{\Omega}

\newcommand{\X}{\mathfrak X}
\newcommand{\Xu}{\mathfrak X u}

\newcommand{\e}{\mathrm{e}}
\newcommand{\dx}{\, dx}

\renewcommand{\epsilon}{\varepsilon}

\def\mvint_#1{\mathchoice
          {\mathop{\vrule width 6pt height 3 pt depth -2.5pt
                  \kern -8pt \intop}\nolimits_{\hspace{-1ex}#1}}%
          {\mathop{\vrule width 5pt height 3 pt depth -2.6pt
                  \kern -6pt \intop}\nolimits_{\hspace{-1ex}#1}}%
          {\mathop{\vrule width 5pt height 3 pt depth -2.6pt
                  \kern -6pt \intop}\nolimits_{\hspace{-1ex}#1}}%
          {\mathop{\vrule width 5pt height 3 pt depth -2.6pt
                  \kern -6pt \intop}\nolimits_{\hspace{-1ex}#1}}}

\newcommand{\DhZ}{D_{h}^Z}
\newcommand{\DhT}{D_{h}^T}
\newcommand{\DhX}{D_{h}^{X_s}}
\newcommand{\DhY}{D_{h}^{Y_s}}

\newcommand{\deltaX}{\big(\mu^2 +|\Xu|^2\big)}
\newcommand{\deltaXi}{\big(\mu^2 +|X_s u|^2\big)}
\newcommand{\deltaXhi}{\big(\mu^2 +|\Xu(x \e^{hX_s})|^2\big)}

\newcommand{\deltaXt}{\big(\mu^2+|\Xu(x)|^2 + |\Xu(x\e^{hT})|^2\big)}

\newcommand{\weight}{\big(\mu^2+|\Xu|^2\big)^\frac{p-2}{2}}

\newcommand{\weightT}{\big(\mu^2 +|\Xu(x)|^2 + |\Xu(x\e^{hT})|^2\big)^\frac{p-2}{2}}
\newcommand{\weightX}{\big(\mu^2 +|\Xu(x)|^2 + |\Xu(x\e^{hX_s})|^2\big)^\frac{p-2}{2}}

\newcommand{\supp}{\operatorname{supp}}

 \newcommand{\intav}{-\hskip -1.1em\int} \newbox\tratto
\setbox\tratto=\hbox{{\count0=0\dimen0=-,9pt\dimen1=1,1pt
\loop\ifnum\count0<11 \advance \count0 by1 \vrule width.51pt
height\dimen1 depth\dimen0\kern-0.17pt \advance\dimen0
by-0.05pt\advance\dimen1 by0.1pt\repeat
\loop\ifnum\count0<21\advance \count0 by1 \vrule width.6pt
height\dimen1 depth\dimen0\kern-0.2pt \advance\dimen0
by-0.1pt\advance\dimen1by0.05pt\repeat}}

\newcommand{\medint}{\displaystyle\copy\tratto\kern-10.4pt\int\limits}

\newcommand{\XXX}{\mathfrak{X}}
\allowdisplaybreaks \DeclareMathOperator{\loc}{loc}

\begin{document}
\title[Heisenberg Gradient regularity]{Gradient regularity for elliptic
equations in the Heisenberg Group}

 \author{Giuseppe Mingione}
\address{Dipartimento di Matematica, Universit\`{a}
di Parma, Viale G.~P.~Usberti 53/A, I-43100 Parma, Italy}
 \email{giuseppe.mingione@unipr.it}
  \urladdr{http://www.unipr.it/\~{}mingiu36}

  \author{Anna Zatorska-Goldstein}
\address{Instytut Matematyki Stosowanej i Mechaniki,
Uniwersytet Warszawski, Banacha 2, 02-097 Warszawa, Poland}
 \email{azator@mimuw.edu.pl}
   \urladdr{http://www.mimuw.edu.pl/~azator/}

\author{Xiao Zhong}
\address{Department of Mathematics and Statistics,
P.O.Box 35 (MaD) FI-40014, University of Jyv\"askyl\"a, Finland}
 \email{zhong@maths.jyu.fi }
   \urladdr{http://www.math.jyu.fi/~zhong/}
\date{\today}

 \begin{abstract} We give dimension-free regularity conditions for a class of
 possibly degenerate sub-elliptic equations in the Heisenberg group
 exhibiting super-quadratic growth in the horizontal gradient; this solves an issue raised in \cite{MM},
 where only dimension dependent bounds for the growth exponent are given.
 We also obtain
explicit a priori local regularity
 estimates, and cover the case of the horizontal $p$-Laplacean operator, extending some regularity proven in \cite{DM2}.
In turn, the a priori estimates found are shown to imply the suitable local
Calder\'on-Zygmund theory for the related class of non-homogeneous, possibly degenerate equations involving discontinuous
 coefficients. These last results extend to the sub-elliptic setting a few classical non-linear Euclidean results \cite{Ip, DMa}, and to the non-linear case estimates of the same nature that were available in the sub-elliptic setting only for solutions to linear equations.
 \end{abstract}
  \maketitle
  \setcounter{tocdepth}{1}
 \tableofcontents
 \section{Introduction}
The regularity in question concerns sub-elliptic equations of the
type\begin{equation}\label{due} \divo a(\XXX u)= \sum_{i=1}^{2n} X_i
a_i\!\left(\mathfrak{X}u\right) =0,
\end{equation}
which are defined in a bounded, open sub-domain $\Omega$ of the
Heisenberg group $\mathbb{H}^n$, $n\geq 1$. The vector field $a
=(a_i)\colon \mathbb R^{2n} \mapsto \mathbb R^{2n}$ is assumed to be
of class $C^1$ and satisfying the following growth and ellipticity
conditions:
\begin{equation}\label{growth}
|Da(z)|(\mu^2+|z|^2)^{\frac{1}{2}} + |a(z)| \leq
L(\mu^2+|z|^2)^{\frac{p-1}{2}},
\end{equation}
and
\begin{equation}\label{ell}
\nu (\mu^2+|z|^2)^{\frac{p-2}{2}}|\lambda|^2 \leq
\sum_{i,j=1}^{2n}D_{z_j}a_i(z)\lambda_i\lambda_j  ,
\end{equation}
for every $z,\lambda \in \mathbb R^{2n}$, where
$$0 < \nu \leq 1 \leq L\,,\qquad \mu \in [0,1]\,,\qquad p \geq 2\;.$$ At certain stages we shall assume
the (sub-elliptic) non-degeneracy condition
\begin{equation}\label{nondeg}
\mu >0\;.
\end{equation}
Assumptions \rif{growth}-\rif{ell} are standard when considering
quasi-linear equations, and their consideration traces back to the
classical Euclidean work of Ladyzhenskaya \& Uraltseva \cite{LU}.
Such assumptions are clearly tailored on the basic model equation
\eqn{subnondeg}
$$
\divo  \left(\left(\mu^2+ \left| \mathfrak{X}u
\right|^2\right)^{\frac{p-2}{2}} \XXX u\right) =0\;, $$ whose
left-hand side operator reduces to the Kohn-Laplacean for $p=2$,
while taking $\mu=0$ we have the also familiar horizontal
$p$-Laplacean operator on the left-hand side: \eqn{subdeg}
$$
\divo \left(|\mathfrak{X}u|^{p-2} \XXX u\right) =0\;.
$$ In order to
preliminarily fix some notation, let us recall that we are denoting
points $x \in \mathbb{H}^n\equiv \er^{2n+1}$ by mean of the usual
exponential coordinates \eqn{iidd}
$$x=(x_1,x_2,\ldots,x_n, x_{n+1}, \ldots, x_{2n}, t)\;,$$ while throughout the paper we are
denoting \eqn{can1}
$$X_i \equiv
X_i(x)=\partial_{x_i}-\frac{x_{n+i}}{2}\partial_t,\qquad   X_{n+i}
\equiv X_{n+i}(x)=\partial_{x_{n+i}}+\frac{x_{i}}{2}\partial_t,$$
and \eqn{can2}
$$\  T \equiv T(x)=\partial_t, \qquad \mathfrak{X}u=(X_1u,
X_2u, \ldots, X_{2n}u).$$ The functional ambient of the problem
\rif{due} is the sub-elliptic Sobolev space $HW^{1,p}(\Omega)$ (see
Section \ref{subsob} below), that is, solutions $u$ are assumed to
belong to $L^p(\Omega)$ and to satisfy \eqn{initial}
$$
\XXX u \in L^{p}(\Omega,\er^{2n})\;,
$$
while nothing is assumed about $Tu$. We recall that if $F\equiv
(F_i):\Omega \to \er^{2n}$ is an $L^1$ vector field in the following
we shall denote the horizontal divergence operator by
$$
\divo F \equiv \sum_{i=1}^{2n} X_i F_i\;,
$$
which is obviously defined in the distributional sense. We refer to
Section 2 for more on the Heisenberg groups $\mathbb{H}^n$,
$n=1,2,3\ldots, $ and for the related notation adopted in this
paper.

\subsection{Gradient regularity} The study of regularity properties of weak solutions to
\rif{due} started with the classical paper of H\"ormander \cite{H},
which dealt with general vector fields and linear equations, and was
later followed by other remarkable contributions devoted to the
linear case, as for instance \cite{F2, F, Kohn}. Capogna was the
first to obtain H\"older continuity theorems for the gradient of
solutions to quasi-linear sub-elliptic equations in divergence form:
initially in the Heisenberg group \cite{Ca1}, and then in more
general Carnot groups \cite{Ca2}; see also his thesis \cite{Ca0}. The operators considered in
\cite{Ca1, Ca2} have quadratic growth, that is, they satisfy
\rif{growth}-\rif{ell} for $p=2$, so that equations as those in
\rif{subnondeg}-\rif{subdeg} are not covered by his theory unless a
priori regularity assumptions are made on the gradient. The case
$p>2$ is another story; indeed while H\"older continuity of $u$ has been
obtained in \cite{CDG, Lu}, when considering the gradient of solutions
only partial regularity results are available, that
is, the regularity of the gradient outside a closed, negligible
subset of the domain $\Omega$; this fact has first been established
by Capogna \& Garofalo in \cite{CG}; another proof is given by
F\"oglein \cite{Fog}. When turning to everywhere
continuity of $Du$, the regularity results obtained prescribe that
the exponent $p$ should not be ``too far from $2$", roughly meaning that
the non-linearity of \rif{due} is in some sense not too strong. In
this respect, Domokos \cite{D}, extending earlier, pioneering
results of Marchi \cite{Ma}, proved that $
 Tu \in L^p_{\loc}(\Omega)$ if $p<4$,
which proved to be an up-to-now unavoidable upper bound on $p$, coming in a
particularly natural way from the analysis of \rif{due}. Proving
that $Tu \in L^p_{\loc}(\Omega)$ is of course the first fundamental
step towards the regularization of solutions $u$ to \rif{due}, since
for them the initial regularity information is just \rif{initial}.
As for the higher regularity of $Du$ or $\Xu$, a few H\"older
regularity results are available in \cite{CM, DM, DM2, MM}; a common
feature of such papers is to prove regularity results for solutions
assuming not only that $p<4$, but also an additional dimensional
bound of the type \eqn{dimbound}
$$2\leq p < 2+o_n\,$$ where $o_n>0$ denotes
a rather awkward, and only in principle explicitly computable quantity, such that
$o_n\searrow 0$ when $n \nearrow \infty$. An unpleasant feature of
an assumption such as \rif{dimbound} is that for a fixed $p$ in the
range $[2,4)$ only low dimensional Heisenberg groups can be dealt
with. For instance, considering the full range $[2,4)$, the
regularity results available in \cite{MM} only apply to
$\mathbb{H}^1$ and $\mathbb{H}^2$; we note that the paper \cite{MM},
where up to now the best bounds of the type \rif{dimbound} have been
found, only regards the non-degenerate case $\mu>0$. Indeed, we
explicitly remark that only few regularity results are available in
the (sub-elliptic) degenerate case $\mu=0$, and therefore for
solutions to \rif{subdeg}. See \cite{DM}; moreover, in the degenerate case the quantity
$o_n$ in \rif{dimbound} is not explicitly computable.

In this respect, the aim of the present paper is now twofold: first
we are giving {\em the first dimension-free pointwise regularity
results for gradients of solutions}, therefore completely avoiding the use of any dimensional assumptions of the type \rif{dimbound}. Second,
and probably more interestingly, up to a certain extent we shall
also treat the degenerate case $\mu=0$, thereby {\em covering the
sub-elliptic $p$-Laplacean equation} \rif{subdeg}. For instance, we
shall prove the local Lipschitz continuity of solutions with respect
to the intrinsic Carnot-Carath\`eodory metric.

The first result we are presenting regards the non-degenerate case
$\mu >0$.
\begin{theorem} [The non-degenerate case]\label{main} Let $u \in HW^{1,p}(\Omega)$ be a weak solution to the equation \trif{due}
under the assumptions \trif{growth}-\trif{nondeg} with $2 \leq p <
4$. Then
 the Euclidean gradient $Du$ is locally H\"older continuous in $\Omega$.
\end{theorem}
See Section 2.4 below for the definition of the Horizontal Sobolev
space $HW^{1,p}$. The previous result solves an issue raised in \cite{MM}, where the authors
were able to obtain the same degree of regularity only under an additional assumption of the type \rif{dimbound}. As later described in
Section 1.3, we shall adopt here
different technical tricks from the ones used in \cite{MM}; these will allow us to develop more efficient bootstrap procedures.

Theorem \ref{main} comes along with explicit a
priori estimates:
\begin{theorem}[Non-degenerate estimates] \label{main2} Let $u \in HW^{1,p}(\Omega)$ be a weak solution to the equation \trif{due}
under the assumptions \trif{growth}-\trif{nondeg} with $2 \leq p <
4$. There exists a constant $c$, depending on $n,p$ and $\ratio$,
but otherwise independent of $\mu$, of the solution $u$, and of the
vector field $a(\cdot)$, such that the following inequalities hold
for any CC-ball $B_R \subset \Omega$: \eqn{apapinf}
$$
\sup_{B_{R/2}}|\XXX u|\leq c \left(\intav_{B_{R}}(\mu+|\XXX
u|)^{p}\, dx\right)^{1/p}\;,
$$
and \eqn{apatinf}
$$
\sup_{B_{R/2}}R|T u|\leq c\mu^{\frac{Q(2-p)}{4}}
\left(\intav_{B_{R}}(\mu+|\Xu|)^{p}\,
dx\right)^{\frac{1}{p}+\frac{Q(p-2)}{4p}}\;.
$$
Finally, for every $1 < q < \infty$ there exists a constant
$\tilde{c}$ depending only on $n,p,\ratio, q$ such that
\eqn{apatpre}
$$
\left(\intav_{B_{R/2}}|T u|^{q}\, dx\right)^{1/q}\leq
\frac{\tilde{c}}{R} \left(\intav_{B_{R}}(\mu+|\XXX u|)^{p}\,
dx\right)^{1/p}\;.
$$
\end{theorem}
For the definition of CC-balls and more notation see Section
\ref{CCb} below. See also \rif{average} for more notation. Next we
turn to the degenerate case $\mu=0$, where the chief model example
is \rif{subdeg}.
\begin{theorem}[The degenerate case] \label{main3}
 Let $u \in HW^{1,p}(\Omega)$ be a weak solution to the equation \trif{due}
 under the assumptions \trif{growth}-\trif{ell} with $\mu=0$, where $
2\leq  p < 4.$ Then  \eqn{betterint}
 $$ \mathfrak{X}u \in
L^{\infty}_{\textnormal{loc}}(\Omega,\er^{2n}), \qquad
\mbox{and}\qquad Tu \in L^{q}_{\textnormal{loc}}(\Omega) \qquad
\mbox{for every} \ \ q<\infty\;.$$ Moreover there exists a constant
$c$, depending on $n,p,\ratio$, but otherwise independent
of the solution $u$, and of the vector field $a(\cdot)$, such that
the following inequality holds for any CC-ball $B_R \subset \Omega$:
\eqn{apapinfdeg}
$$
\sup_{B_{R/2}}|\XXX u|\leq c \left(\intav_{B_{R}}|\XXX u|^{p}\,
dx\right)^{1/p}\;.
$$
Finally, for every $q< \infty$ there exists a constant $\tilde{c}$
depending only on $n,p,\ratio,q$ such that \eqn{apat00}
$$
\left(\intav_{B_{R/2}}|T u|^{q}\, dx\right)^{1/q}\leq
\frac{\tilde{c}}{R} \left(\intav_{B_{R}}|\XXX u|^{p}\,
dx\right)^{1/p}\;.
$$
\end{theorem}
The previous theorem partially extends some regularity results proven in \cite{DM2}, where the authors work under an assumption of the type \rif{dimbound}, this time
$o_n$ being a small, unspecified quantity coming from the application of abstract Cordes type condition methods.
In turn, the boundedness of the horizontal gradient naturally yields
a priori Lipschitz bounds:
\begin{cor}[CC-Lipschitz regularity]\label{cor1} Let $u \in HW^{1,p}(\Omega)$ be a weak solution to the equation \trif{due}
 under the assumptions \trif{growth}-\trif{ell} with $
2\leq  p < 4.$ Then $u$ is locally Lipschitz continuous in $\Omega$
with respect to the CC-metric in $\mathbb H^n$. Moreover there
exists a constant $c$, depending only on $n,p,\ratio$, but otherwise
independent of $\mu$, of the solution $u$, and of the vector field
$a(\cdot)$, such that \eqn{Lipest}
$$
 |u(x)-u(y)|\leq c \left(\intav_{B_{R}}(\mu+|\XXX
u|)^{p}\, dx\right)^{1/p} d_{cc}(x, y)\;.
$$
holds whenever $B_R \subset \Omega$, and $x,y \in B_{R/2}$.
\end{cor}
See \rif{distanza} below for the definition of the intrinsic
distance $d_{cc}(\cdot,\cdot)$. Another consequence of the Theorem
 \ref{main3} and of the standard, Euclidean Sobolev-Morrey embedding theorem, is now the
following:
\begin{cor}[Almost Euclidean-Lipschitz regularity]\label{cor2} Let $u \in HW^{1,p}(\Omega)$ be a weak solution to the equation \trif{due}
 under the assumptions \trif{growth}-\trif{ell} with $
2\leq  p < 4.$ Then $u \in C^{0,\alpha}_{\loc}(\Omega)$ for every
$\alpha < 1$.
\end{cor}
Needless to say, in the last result the H\"older continuity is referred to the standard Euclidean metric. We finally mention that the previous theorems are stated for $2 \leq p <4$ for completeness, since in the automatically non-degenerate case $p=2$ they are essentially due to Capogna \cite{Ca1}.
\subsection{Calder\'on-Zygmund type estimates} The estimate
\rif{apapinfdeg} found in Theorem \ref{main3} opens the way to a
non-linear version of the estimates of Calder\'on-Zygmund type in
the the Heisenberg group, up to now developed only in the case of
linear sub-elliptic equations \cite{brabratams, brabrajde}. Here we
shall deal with non-linear equations. Let us recall that in the
Euclidean setting this is a classical result dating back to T.
Iwaniec \cite{Ip} in the scalar case, and later extended to systems
of $p$-Laplacean type in \cite{DMa} by DiBenedetto \& Manfredi; see
also \cite{Caffpe} for a different approach. The equations
considered by such authors are modeled by \eqn{eup}
$$
\textnormal{div}\ (|Du|^{p-2}Du)= \textnormal{div}\ (|F|^{p-2}F)\;,
$$
in open subsets of $\er^n$, and the result asserts that $F \in
L^{q}_{\loc}$ implies $Du \in L^{q}_{\loc}$ for any $q>p$.
 More
recently Calder\'on-Zygmund type estimates valid for solutions to
general non-linear elliptic systems have been proposed in \cite{KM},
and, following the techniques of this last paper, in the Heisenberg
group case in \cite{GZ} for certain non-linear problems with quadratic
growth, that is, when $p=2$. An extension for linear equations in CR
manifolds has been obtained in \cite{shawang}. In the following we shall give higher
integrability results for problems with possibly super-quadratic growth
$p\geq 2$. The equations we are considering are the natural
horizontal version of \rif{eup}, involving possibly
discontinuous coefficients of VMO type; specifically
\begin{equation}\label{duecz} \divo \left[b(x)
a\!\left(\mathfrak{X}u\right)\right] =\divo (|F|^{p-2}F)\,,
\end{equation}
with \eqn{vmo}
$$
b(\cdot) \in \textnormal{VMO}_{\loc}(\Omega)\qquad \mbox{and}\qquad
\nu \leq b(x)\leq L\;.$$ See Section \ref{vmosec} for the precise
definition of the space VMO$_{\loc}(\Omega)$. The prototype of
\rif{duecz} is clearly the non-homogeneous $p$-Laplacean equation
with VMO-coefficients, that is
\begin{equation}\label{unoczx}
\divo  \left( b(x)\left| \mathfrak{X}u \right|^{p-2} \X u\right)
=\divo\left( |F|^{p-2} F \right)\;,
\end{equation}
where $\nu\leq b(x)\leq L$ satisfies \rif{vmo}, and $F \in
L^p(\Omega,\er^{2n})$. The main result is the following:
\begin{theorem}[of Calder\'on-Zygmund type]\label{phcz} Let $u \in HW^{1,p}(\Omega)$ be a weak solution to the equation \trif{duecz}
under the assumptions \trif{growth}-\trif{ell} with $2\leq p <4$,
and \trif{vmo}. Then
$$F \in L^{q}_{\loc}(\Omega,\er^{2n}) \qquad \mbox{implies that}
\qquad \XXX u\in L^{q}_{\loc}(\Omega,\er^{2n})\,,$$ whenever
$p<q<\infty$. Moreover there
 exists a constant
$c$, depending only on $n,p,\ratio, q$, and the function $b(\cdot)$,
such that the following reverse-H\"older type inequality holds for
any CC-ball $B_R \Subset \Omega$: \eqn{apapd}
$$
\left(\intav_{B_{R/2}}|\XXX u|^{q}\, dx\right)^{1/q}\leq c
\left(\intav_{B_{R}}(\mu+|\XXX u|)^{p}\,
dx\right)^{1/p}+c\left(\intav_{B_{R}}|F|^q\, dx\right)^{1/q}\;.
$$
\end{theorem}
For an alternative statement concerning the dependence of the
constant in \rif{apapd} see also Remark \ref{alt} below, while for a
more precise dependence on the various constants see Remark
\ref{alt2}. Let us recall that in the Euclidean case there is a wide
literature on Calder\'on-Zygmund type estimates for linear problems
with VMO-coefficients starting from the Euclidean work of Chiarenza
\& Frasca \& Longo \cite{Chiarenza}, dealing with linear problems. A
non-linear approach has been proposed in \cite{KZ}. As for the
sub-elliptic setting, the theory is confined to the linear case
\cite{brabratams}, where the case of H\"ormander vector fields are
considered. In this paper we give the first results for non-linear
problems with VMO coefficients, allowing also for BMO coefficients
with small BMO semi-norm, see Remark \ref{bmo} below. Anyway we
remark that the integrability results obtained here are new
already in the case $b(x) \equiv 1$ - that is, when no coefficients
are involved. Moreover, we remark that the result of Theorem
\ref{phcz} extends to a family of more general equations with
continuous coefficients; the corresponding statements are presented
at the end of the paper.

\subsection{Technical approach, and novelties.} The approach proposed in this paper strongly
differs from those proposed in earlier ones. Indeed, a common
strategy for attacking the regularity problem in the sub-elliptic
setting, going back to H\"ormander \cite{H} and then followed in
subsequent works \cite{F, F2, Ca1, Ca2}, is to first obtain
separately a certain maximal regularity for the vertical part of the
gradient $Tu$, and then, using such an additional information,
obtaining regularity results for the horizontal part $\Xu$. Such an
approach is for instance followed also in the non-linear setting in
\cite{Ca1, Ca2}, where it turns out to be successful since $p=2$. We
take different path, hereby proposing a double-bootstrap method: we
shall obtain regularity for $Tu$ using the one obtained for $\Xu$,
and vice-versa. More precisely we shall prove that \eqn{mixed}
$$ Tu \in L^{q_k} \Longrightarrow \Xu \in L^{p_k} \quad
\mbox{and}\quad Xu \in L^{p_k} \Longrightarrow Tu \in L^{q_{k+1}}
$$
where $\{p_k\}$ and $\{q_k\}$ are two sequences diverging to infinity; in some
sense we repeat H\"ormander's original strategy breaking it in a
countable number of pieces. As a first consequence we obtain that
\eqn{mixed2}
$$\Xu, Tu \in L^{q} \qquad \mbox{for every} \ \ q < \infty\;,$$ while we remark that
all the foregoing inclusions are meant to be local since no boundary
information is a priori given on solutions. The use of such a mixed
iteration is a direct consequence of the non-linearity of equation
\rif{due}, since $Tu$ cannot be realized as a solution of a similar
equation, and a deeper interaction between the horizontal and the
vertical parts of the gradient must be exploited. The implementation
of \rif{mixed} requires a rather delicate interaction between:
suitable Caccioppoli type estimates - also called energy estimates -
for the horizontal and vertical gradients, see Section 5;
interpolation inequalities of Gagliardo-Nirenberg type in the
Heisenberg group, see Section 4; integration-by-parts methods, see
Section 7; a certain kind of non-standard energy estimates of mixed
type, see Section 6. A careful combination of such ingredients will
lead to \rif{mixed}. Once the integrability information in
\rif{mixed2} is gained, a suitable variant of Moser's iteration
technique will lead to $\Xu \in L^{\infty}$, see Section 8. Finally,
in the non-degenerate case $\mu>0$ this will lead to $Tu \in
L^{\infty}$ via the results in \cite{MM}, and eventually to the
local H\"older continuity of the Euclidean gradient, which is a
standard implication after the work in \cite{Ca0, Ca1, MM}.

An important background of our technique is the observation of the
natural analogy between sub-elliptic equations of the type
\rif{due}, and the more classical Euclidean non-uniformly elliptic
equations, or ``equations with non-standard growth conditions", or
with ``$(p,q)$-growth conditions", as very often called in the
setting of the Calculus of Variations \cite{ELM2, ELM3}. In fact, our
techniques are inspired by those developed for such situations, see
for instance \cite{BFZ}, although the implementation in the
Heisenberg group requires a completely different technical approach.
Problems with non-standard growth indeed involve equations featuring
ellipticity properties which appear to be weaker in certain special
spatial directions: this immediately reminds of the situation of
horizontal quasi-linear equations in the Heisenberg group as
\rif{due}, where the vertical derivative $Tu$ does not appear
directly in the operator. It rather appears only in an intrinsic
way, via the horizontal vector fields $\XXX u$ and after
commutation, see \rif{comm} below, and therefore the vertical
direction is clearly playing a very special role. Such a lack of
``vertical ellipticity" is in fact the basic source of problems in
the theory of elliptic equations in the Heisenberg group.

As mentioned above, a key ingredient for the subsequent results are
the explicit a priori estimates \rif{apapinf} and \rif{apapinfdeg}.
Indeed, these will allow for  a suitable application of recent
non-linear techniques for obtaining higher integrability estimates
for non-homogeneous equations \cite{Caffpe, KM}. Here, due to the
presence of the VMO coefficients, we shall use these in combination
with various maximal operators, and higher integrability estimates
in the spirit of Gehring's lemma. Observe that, due to the non-linearity of the problems
we are considering, the standard approaches based on harmonic analysis tools
such as, singular integrals, commutators, and so forth, are not available in the present setting.

Finally, let us summarize the content of the paper. In Section 2 we shall collect
preliminaries concerning the sub-elliptic setting, while in Section
3 we shall re-visit and re-state in a suitable way a few known
regularity results for elliptic equations in the Heisenberg group.
Sections 4-7 are devoted to the implementation of \rif{mixed}, in
the way described a few lines above. Here we shall else re-visit
some arguments from \cite{MM}, and we shall use the a priori
boundedness of the solution already obtained in \cite{CDG}. In
Section 8 we prove $L^{\infty}$-estimates for the gradient and
therefore Theorems \ref{main}, \ref{main2}. Section 9 is devoted to
the degenerate case: we prove Theorem \ref{main3}, by combining
Theorem \ref{main2} with a standard approximation method, and then
we obtain Corollaries \ref{cor1}-\ref{cor2}. The proof of Theorem
\ref{phcz} is in Section 10, while in Section 11 we
give a few possible generalizations of Theorem \ref{phcz}.

{\em Acknowledgments.} G.~M.~is supported by MUR via the national
project ``Calcolo delle Variazioni", and by GNAMPA via the project
``Singularities and regularity in non-linear potential theory". Part
of this work was done while G.~M.~ was visiting the Universities of
Warsaw, Helsinki and Erlangen-N\"urnberg, in May, June and August 2007, respectively. He wishes to
thank all the members of the institutions for the nice hospitality. A.ZG. would like to thank her
hosts at the Helsinki University of Technology and at the University of Jyv\"askyl\"a.
X.~Z. is supported by the Academy of Finland, project 207288. The authors wish to thank Anna F\"oglein for remarks on a first version of the manuscript.

\section{Notation, preliminaries}
\subsection{ Notations, conventions} In this paper we shall adopt the usual,
but somehow arguable convention to denote by $c$ a general constant,
that may vary from line to line; peculiar dependence on parameters
will be properly emphasized in parentheses when needed. More
precisely we shall usually denote $c \equiv c(\alpha,
\beta,\gamma,\ldots)$, meaning that that $c$ is actually an
increasing (or decreasing) function of $\alpha,
\beta,\gamma,\ldots$; in general $c\nearrow \infty$ when either one
of the parameters goes to infinity or to zero. For this reason, when dealing with a constant potentially depending on several parameters,
in the case when one of the parameters remains bounded, the constant is in fact independent on the parameter in question. Specific
occurrences will be clarified by the context. Moreover, special
occurences will be denoted by $c_*, c_1, c_2$ or the like. In this
paper all the constant named by $c_*, c_1, c_2$ and so on will be
assumed without loss of generality to be larger than $1$. The scalar
product between elements $z_1,z_2$ of $\er^{2n}$ will be denoted by
$\langle z_1,z_2 \rangle $; very often, when no ambiguities will
arise, we shall simply denote $\langle z_1,z_2 \rangle \equiv
z_1z_2$. Finally $\{e_1,\ldots,e_{2n+1}\}$ denotes the standard
basis of $\er^{2n+1}$.

In the following, several of the integral estimates for solutions to
\rif{due} will involve constants depending on the ellipticity/growth
parameters $\mu$ and $L$, displayed in \rif{growth}-\rif{ell}.
Without loss of generality, eventually replacing the vector field
$a(\cdot)$ by $a(\cdot)/\nu$ we may assume that $\nu=1$. Therefore,
scaling back, we see that all the constants depending on $\nu,L$
will actually depend on the unique quantity $\ratio$, and as such
they will be denoted for the rest of the paper.

\subsection{Heisenberg groups.} We identify the Heisenberg group $\mathbb{H}^n$
with $\mathbb{R}^{2n+1}$, $n\geq 1$, via the exponential coordinates
in \rif{iidd}, see also \rif{eexp} below. The group multiplication
is given by
\begin{eqnarray*} && (x_1 , ... , x_{2n}, t ) \cdot (y_1 , ... ,
y_{2n} , s ) \\ && \hspace{2cm}= ( x_1 + y_1 , ... ,x_{2n} + y_{2n}
, t + s + \frac{1}{2} \sum_{i=1}^n (x_i y_{n+i} - x_{n+i} y_i ))
\;,\end{eqnarray*} and makes $\mathbb{H}^n$ a non-commutative group.
 For $1\le i\le n$ the canonical left invariant
 vector fields are those in \rif{can1}-\rif{can2}. The only non-trivial commutator is
\eqn{comm}
 $$T=\partial_t=[X_i,
X_{n+i}]\equiv X_iX_{n+i}-X_{n+i}X_i, \qquad \mbox{for every}\
i=1,\ldots,n\;.$$ The vector fields $X_1, X_2, \ldots, X_{2n}$ are
called horizontal vector fields, while $T$ is the vertical vector
field. The horizontal gradient of a function
$u\colon\mathbb{H}^n\mapsto\mathbb{R}$ is the vector $\XXX u$
defined in \rif{can2}. The vector fields $\{X_i\}_i$ enjoy the
remarkable property of being opposite to their formal adjoint, that
is \eqn{adad}
$$
X_i^* = -X_i, \qquad \mbox{for every}\ \ i =1,\ldots,2n\,.
$$
The second horizontal derivatives are given by the $2n\times 2n$
matrix $\mathfrak{X}\mathfrak{X}u=\mathfrak{X}^2u$ with entries
$\left(\mathfrak{X}(\mathfrak{X}u) \right)_{i,j}=\left(\XXX \XXX
u\right)_{i,j}= X_i(X_j(u)).$ Note that such a matrix is not
symmetric due to the non-commutativity of the horizontal vector
fields $X_{i}$. We shall
denote the standard Euclidean gradient of a function $u$ as $Du
= (D_1u,\ldots, D_{2n+1}u).$ For notational convenience, when
referring to the coordinates and vector fields in
\rif{iidd}-\rif{can1}
 we shall also denote $Y_s =X_{s+n}$ and $y_s=x_{s+n}$, for $s\in \{1,\ldots,n\}.$

The Heisenberg Lie algebra $\mathfrak{h}^n$ is a step 2 nilpotent
Lie algebra. This means that $\mathfrak{h}^n$ admits a decomposition
as a direct sum of vector spaces
$\mathfrak{h}^n=\mathfrak{h}_0\oplus\mathfrak{h}_1$ such that
$[\mathfrak{h}_0, \mathfrak{h}_0]=\mathfrak{h}_{1}.$ The horizontal
part $\mathfrak{h}_0$ is generated by $\{X_1,\ldots, X_n,
Y_1,\ldots, Y_n\}$ and the vertical part $\mathfrak{h}_1$ by $T$.
  Note that $\mathfrak{h}^n$
 is generated as a Lie algebra by $\mathfrak{h}_0$. \par
 The   exponential mapping $\text{exp}\colon \mathfrak{h}^n\mapsto \mathbb{H}^n$ is a global
diffeomorphism. A point $x\in\mathbb{H}^n$ has exponential
coordinates $(x_1, \ldots, x_n, y_1,\ldots, y_n, t)$ if \eqn{eexp}
$$x=\text{exp}\left(\left(\sum_{j=1}^n x_iX_{i}+y_i Y_i\right)+ t T
\right).$$ The identification between $\mathbb{H}^n$,
$\mathfrak{h}^n$, and $\mathbb{R}^{2n+1}$ is precisely the use of
exponential coordinates in $\mathbb{H}^n$, and it is already used in
\rif{iidd}; in the following we shall denote exp$(Z)\equiv \e^Z$.
\par
The horizontal tangent space at a point $x\in \mathbb{H}^n$ is the
$2n$-dimensional subspace
$$T_{\mathrm{h}}(x)=\text{linear span}\{X_{1}(x),\ldots, X_{n}(x), Y_1(x),\ldots, Y_n(x)\}.$$
A piecewise smooth curve $t\mapsto\gamma(t)$ is horizontal if
$\gamma'(t)\in T_h(\gamma(t))$ whenever $\gamma'(t)$ exists. Given
two points $x  , y \in \mathbb{H}^n$ denote by $\Gamma(x,y) = \{
\textrm{horizontal curves joining} \ x \ \textrm{and} \  y\}.$
Chow's accessibility theorem \cite{chow} implies that $\Gamma(x,y)
\neq \emptyset$.\par For convenience, we fix an ambient Riemannian
metric in $\mathbb{H}^n$ so
 that the set $\mathfrak{h}_0=\{X_1,\ldots, X_n, Y_1,\ldots, Y_n\}$
 is a left invariant  orthonormal frame and the Riemannian volume
 element and group Haar measure agree, and are equal to
 the Lebesgue measure in $\mathbb{R}^{2n+1}$. The Carnot-Carath\`eodory metric (CC-distance) is then defined by
\eqn{distanza}
$$d_{cc}(x, y) = \inf\{\textrm{length}(\gamma): \gamma \in \Gamma(x,y)\}.$$
It depends only on the restriction of the ambient Riemannian metric
to the horizontal distribution generated by the horizontal tangent
space. In the following, with $A,B \subset \mathbb H^n$ being
non-empty subsets, we denote $\dist(A,B):= \inf \{d_{cc}(x, y):  x
\in A, \ y \in B\}$, the Carnot-Carath\`eodory distance between
sets. For more on CC-distances and general properties of metrics
related to vector fields we refer to the classical paper \cite{NSW}.

\subsection{CC-balls, and the homogeneous dimension $Q$.}\label{CCb} The Carnot gauge is $|x|_{cc} = d_{cc} (x,0)$. A few explicit
formulas are available \cite{belaiche}, but it is probably more
convenient to work with an equivalent gauge \cite{belaiche}, smooth
away from the origin, called the Heisenberg gauge: \eqn{ee1}
$$|x|_{\mathbb{H}^n} :=
\left(\left( \sum_{j=1}^nx_i^2+y_i^2 \right)^{2}+
t^2\right)^{1/4}\approx |x|_{cc} \;.
$$
In this paper all the balls, centered at $x_0 \in \mathbb H^n$ and
with radius $R$, will be defined with respect to the CC-distance,
that is $B(x_0,R) =\{y \in \mathbb{H}^n \colon d_{CC}(x_0,y)<R\}.$
In view of \rif{ee1} they are equivalent to the gauge balls
obviously defined by $\{y \in \mathbb{H}^n \colon |y^{-1}\cdot
x_0|_{\mathbb{H}^n}<R\}.$ The non-isotropic dilations are the group
homorphisms given by \eqn{dilation}
$$\delta_R\left(x_1,\ldots,x_n,y_1, \ldots y_n, t
\right)=\left(R x_1,\ldots, R x_n, R y_1, \ldots R y_n, R^2 t
\right),
$$
where $R>0$. The point is that we get the ball centered at the
origin of radius $R>0$ by applying the non-isotropic dilation
$\delta_R$ to the unit ball centered at the origin, that is \eqn{basicrel}
$$B(0,
R)=\delta_R B(0, 1)\;.$$ The equivalence \rif{ee1} and the natural
scaling in \rif{dilation} leads to define the number $Q=2n+2$ as the
homogeneous dimension of $\mathbb H^n$. In particular, we have $
|B(x_0,R)| \approx R^Q, $ where $|B_R|$ denotes the Lebesgue measure
of the ball $B(x_0,R)$. From such an estimate the doubling property
of the CC-balls $B_R$ easily follows; specifically, for any
$B(x_0,R)\subset \mathbb H^n$, there holds \eqn{doubling}
$$
|B(x_0,2R)|\leq C_d |B(x_0,R)|\;.
$$
In the following, when clear, or not essential to the context, we
will omit the center of the ball $B_R = B(x_0,R)$ and, if not
otherwise stated, when considering several balls simultaneously,
they  will be concentric. Finally, again when no ambiguity will
arise, we shall also denote $\lambda B \equiv B(x_0,\lambda R)$, if
$B \equiv B(x_0,R)$, and, when the center of the ball will not be
important, we shall use the short-hand notation $B(x_0,R)\equiv
B_R$. Moreover, when some constant will depend on the homogeneous
dimension $Q$, such a dependence will be very often indicated as on
the number $n$.

Let $B_R \subset \er^n$ be a ball, and $f\colon B_R \to \er^{k}$ be
an integrable map; we define the average of $f$ over the ball $B_R$
as \eqn{average}
$$
(f)_R\equiv (f)_{B_R}:= \intav_{B_R} f(x)\, dx=
\frac{1}{|B_R|}\int_{B_R} f(x)\, dx\approx R^{-Q}\int_{B_R} f(x)\,
dx\;.
$$
The following Krylov-Safonov type covering lemma may be inferred from
\cite{Kinsha, GZ}.
\begin{lemma} \label{metrico} Let $B_R \subset \mathbb H^n$ be a ball with radius
$R$, and let $\delta \in (0,1)$. Assume that $E, G \subset B_R$ are
measurable sets such that $|E|\leq \delta |B_R|$. Assume also that
for any ball $B(x_0,\varrho)$ centered in $B_R$, with $\varrho \leq
2R$, and such that $|E \cap B(x_0,5\varrho)|>\delta |B_R \cap
B(x_0,\varrho)|$, there holds $E \cap B(x_0,5\varrho) \subset G$.
Then it follows that $|E|\leq \delta |G|$.
\end{lemma}
\subsection{Horizontal Sobolev spaces and weak solutions.} \label{subsob} The horizontal
Sobolev space $HW^{1,p}(\Omega)$ consists of those functions $u \in
L^p(\Omega)$ whose horizontal distributional derivatives are in turn
in $L^p(\Omega)$, that is $\mathfrak{X}u\in L^p(\Omega,\er^{2n})$.
 $HW^{1,p}(\Omega)$ is a Banach space when equipped with the norm
 defined by
$
 \|u\|_{HW^{1,p}(\Omega)}:=\|u\|_{L^p(\Omega)}+
 \|\mathfrak{X}u\|_{L^p(\Omega,\er^{2n})},$
 for $p\ge 1$.
 The closure of $C_0^{\infty}(\Omega)$ in $HW^{1,p}(\Omega)$ is denoted by
 $HW_0^{1,p}(\Omega)$, while the local variant  $HW^{1,p}_{\loc}(\Omega)$ is obviously defined by saying that
$u \in HW^{1,p}_{\loc}(\Omega)$ if and only if $u \in
HW^{1,p}(\Omega')$, for every open subset $\Omega' \Subset \Omega$.
Now, keeping \rif{adad} in mind, a weak solution to the equation
\rif{duecz} with $F \in L^{p}(\Omega,\er^{2n})$ is a function $u \in
HW^{1,p}(\Omega)$ such that \eqn{wf}
$$
\int_{\Omega} b(x)\sum_{i=1}^{2n} a_i(\mathfrak{X}u)X_i \varphi \,
dx =\int_{\Omega} \sum_{i=1}^{2n} |F|^{p-2}F_iX_i \varphi \, dx,
\qquad\text{ for all } \ \varphi \in HW^{1,p}_0(\Omega)\;.
$$
Therefore, when considering equation \rif{due}, this means to
require that \eqn{wf0}
$$
\int_{\Omega} \sum_{i=1}^{2n} a_i(\mathfrak{X}u)X_i \varphi \, dx
=0, \qquad\text{ for all } \ \varphi \in HW^{1,p}_0(\Omega)\;.
$$
A crucial result concerning horizontal Sobolev spaces is the
following Heisenberg group version of the Sobolev embedding theorem.
\begin{theorem}\label{subsobt} Let $w \in HW^{1,q}_0(B)$ with $1< q < Q$, where $B\subset
\mathbb H^n$ is a CC-ball. Then there exists a constant $c\equiv
c(n,q)$ such that \eqn{subsobin}
$$
\left(\mvint_{B} |w|^{\frac{Qq}{Q-q}}\dx\right)^{\frac{Q-q}{Qq}}\leq
c |B|^{\frac{1}{Q}}\left(\mvint_{B} |\X
w|^{q}\dx\right)^{\frac{1}{q}}\;.
$$
\end{theorem}
A proof of the previous result can be found for instance in
\cite{CDG, Lu}, where the statement is given in the case of balls
with a suitably small radius $r \leq R_0$. The general case stated
above easily follows by a standard scaling argument, using the
dilation operator in \rif{dilation} and \rif{basicrel}. See also the
proof of Proposition \ref{horiz} below, end of Step 2.
\subsection{Vanishing mean oscillations.}\label{vmosec}
Let $b \colon \Omega \to \er$ be a measurable function, and $\Omega'
\Subset \Omega$; we define \eqn{defivmo}
$$
[b]_{R_0} \equiv [b]_{R_0,\Omega'}:=\sup_{B_R\subset \Omega', R \leq
R_0}\intav_{B_R} |b(x)-(b)_{B_R}|\, dx\;,
$$
where $R_0 >0$, $B_R$ is any CC-ball with radius $R$, and,
accordingly to \rif{average} \eqn{mediavmo}
$$
(b)_R\equiv (b)_{B_R}:= \intav_{B_R} b(x)\, dx\;.
$$
The function $b$ is said to have (locally) vanishing mean
oscillation, that is, to be a VMO-function iff, for every choice of
the subset $\Omega' \Subset \Omega$ it holds that \eqn{vmocond}
$$
\lim_{R\searrow 0} [b]_{R,\Omega'}=0\;.
$$
\subsection{Difference quotients.} Here we recall a few basic properties of the
difference quotient operators in the Heisenberg group.
\begin{definition}\label{diffquot}
Let $Z$ be a vector field in $\mathbb{H}^n$. The difference quotient
of the function $w$ at the point $x$ is
$$D_{h}^Zw(x)=\frac{w(x \e^{h Z})-w( x)}{h}\;,\qquad h\not=0\,.$$
\end{definition}
The latter definition will be always used whenever the function $w$
in question is defined both at $x\e^{hZ}$ and at $x$.
The following lemma collects a few standard properties of difference
quotients that can be for instance inferred from \cite{H, Ca1, D, G,
MM}.
\begin{lemma}\label{hormander}
Let $\Omega'\Subset \Omega$ be an open subset. Let $Z$ being a left-invariant vector field, and $w \in L_{\loc}^p (\Omega )$ for
$p>1$. If there exist two positive constants $\sigma <
\dist(\Omega',\partial \Omega)$ and $C$ such that
$$\sup_{0 < |h| < \sigma} \int_{\Omega'} | D^Z_h w|^p \, dx \leq C^p$$
 then $Zw \in L^p (\Omega')$ and $\| Zw\|_{L^p (\Omega')} \leq C$.
Conversely, if $Zw \in L^p (\Omega')$ then for some $\sigma >0$
$$\sup_{0 < |h| < \sigma} \int_{\Omega'} | D^Z_h w
|^p \, dx \leq  c(p)\|Zw \|_{L^p (\Omega)}^p
 .$$ Moreover $ D^{Z}_hw \to Zw$ strongly in $L^p(\Omega')$.
\end{lemma}
Finally a trivial lemma, which is basically a consequence of the
Campbell-Hausdorff formula; the proof is left to the reader.
\begin{lemma}\label{tritra} Let $\varphi \in HW^{1,t}(\Omega)$, and $X,Z$ be
smooth left-invariant vector fields such that $[X,Z]\varphi \in
L^{t}_{\loc}(\Omega)$, with $t\geq 1$. If
$\tilde{\varphi}:=\varphi(x\e^{Z})$ then $X\tilde{\varphi} \in
L^t_{\loc}(\Omega)$ and \eqn{compx}
$$
X[\varphi(\cdot
\e^Z)](x)=X\tilde{\varphi}(x)=X\varphi(x\e^Z)+[X,Z]\varphi(x\e^Z)
$$
holds provided $x, x\e^Z \in \Omega$. As a consequence we have, for
$h\not=0$ \eqn{eq:hom lemma 1}
$$
X (D_{h}^{Z}\varphi)(x)= D_{h}^{Z}(X\varphi)(x)+
[X,Z]\varphi(x\e^{hZ})\;.
$$
\end{lemma}
Before going on, first two algebraic lemmata; see \cite{ham}, for
instance.
\begin{lemma} \label{al1}
Let $1<p<\infty$. There exists a constant $c= c(n,p)>1$, independent
of $\mu \in [0,1]$, such that, for any $z_1, z_2 \in \mathbb R^{2n}$
\begin{eqnarray}
c^{-1}\Bigl( \mu^2+|z_1|^2+|z_2|^2 \Bigr)^{\frac{p-2}{2}} & \leq&
\int_0^1 (\mu^2 + |z_2 + \tau z_1|^2)^{\frac{p-2}{2}} \ d\tau
\nonumber \\
&\leq&  c\Bigl( \mu^2+|z_1|^2+|z_2|^2
\Bigr)^{\frac{p-2}{2}}.\label{all1}
\end{eqnarray}
\end{lemma}
\begin{lemma} \label{al2}
Let $1<p<\infty$. There exists a constant $c\equiv c(n,p)>1$,
independent of $\mu \in [0,1]$, such that, for any $z_1, z_2 \in
\mathbb R^{2n}$
\begin{eqnarray*}
c^{-1}\Bigl( \mu^2+|z_1|^2+|z_2|^2 \Bigr)^{\frac{p-2}{2}}|z_2-z_1|^2
&   \leq&\left| (\mu^2+|z_2|^2)^{\frac{p-2}{4}}z_2 -(\mu^2+|z_1|^2)^{\frac{p-2}{4}}z_1 \right |^2\\
&\leq&  c\Bigl( \mu^2+|z_1|^2+|z_2|^2
\Bigr)^{\frac{p-2}{2}}|z_2-z_1|^2.
\end{eqnarray*}
\end{lemma}
Finally a few general properties related to growth/ellipticity
conditions \rif{growth}-\rif{ell}.
\begin{lemma} \label{prop B}
The following equality holds:
\begin{equation}\label{Btilde}
    \left(\DhZ a_i(\Xu)\right)(x)
    =
    \sum_{j=1}^{2n}  a_{i,j}^Z (x) \DhZ X_j u(x),
\end{equation}
where \eqn{atz}
$$
     a_{i,j}^Z (x) = \int_0^1 D_{z_j} a_i \big(\Xu(x) + \tau h \DhZ \Xu(x)\big)\,
     d\tau\,,
$$
and $i,j \in\{1,\ldots, 2n\}$. Moreover there exists a constant $c
\equiv c(n,p)\geq 1$ such that
\begin{equation} \label{upperb}
| a_{i,j}^Z (x) | \leq c \big(\mu^2+|\Xu(x)|^2 + |\Xu(x\e^{hZ})|^2
\big)^\frac{p-2}{2}
\end{equation} and\begin{equation} \label{lowerb}
c^{-1} \big(\mu^2+|\Xu(x)|^2 + |\Xu(x\e^{hZ})|^2 \big)^\frac{p-2}{2}
|\lambda|^2 \leq \sum_{i,j=1}^{2n} a_{i,j}^Z (x) \lambda_i\lambda_j,
\end{equation}
hold for every $\lambda \in \er^{2n}$, whenever $x, x\e^{hZ}  \in
\Omega$.
\end{lemma}
\begin{proof} The proof of \rif{Btilde} follows directly from the
definition of $a_{i,j}^Z (x)$, while that of
\rif{lowerb}-\rif{upperb} follows from \rif{growth}-\rif{ell} and
Lemma \ref{al1}.
\end{proof}
\begin{lemma} \label{lemma:diff quot eq}
Let $u \in HW^{1,p}(\Omega)$ be a weak solution to the equation
\trif{due}
 under the assumptions \trif{growth}-\trif{ell} with $
2\leq  p < 4.$ Then for any $\varphi \in C_c^\infty(\Omega)$,
left-invariant vector field $Z$ and $h>0$ such that
$|\e^{hZ}|_{cc}<\mathrm{dist}(\mathrm{supp}\,\varphi,\partial\Omega
)$ we have
\begin{equation}\label{eq: diff quot eq}
    \int_\Omega \sum_{i=1}^{2n}\Big( \DhZ a_i (\Xu)(x) X_i \varphi(x)
    + a_i(\Xu)(x\e^{hZ})[Z,X_i]\varphi(x)\Big)\, dx=0\,.
\end{equation}
\end{lemma}
\begin{proof} With $\tilde{\varphi}(x):= \varphi(x\e^{-hZ})$, using \rif{compx} we have
that $
    X_i \tilde{\varphi}(x)=X_i
    \varphi(x\e^{-hZ})+h[Z,X_i ]\varphi(x\e^{-hZ})
$. Testing \rif{wf0} with $\tilde{\varphi}$ and changing variable $x
\mapsto x\e^{hZ}$, we obtain
\begin{equation*}
    \int_\Omega \sum_{i=1}^{2n}a_i (\Xu(x\e^{hZ}))\left(X_i
    \varphi(x)+h[Z,X_i]\varphi(x)\right)\,dx=0\,.
\end{equation*}
Now we subtract \rif{wf0} from the last identity and divide the
resulting equation by $h$. This finally gives \rif{eq: diff quot
eq}.
\end{proof}
Finally, a standard property of weak derivatives in the Euclidean case, that holds in the present setting too.
We give a sketchy proof for the sake of completeness.
\begin{lemma}\label{eles}
Let $v,w \in L^{1}_{\loc}(\Omega)$ such that $vw,vX_sw, wX_s v \in L^1_{\loc}(\Omega)$ for some $s
\in \{1,\ldots,2n\}$. Then $X_s(vw) \in L^{1}_{\loc}(\Omega)$ and
$X_s(vw)=vX_s w+wX_s v$.
\end{lemma}
\begin{proof} We first assume that both the functions are locally essentially bounded. Then we mollify them using standard mollifiers $\varphi_{\varepsilon}$, obtaining
$v_{\varepsilon}=v*\varphi_{\varepsilon},w_{\varepsilon}=w*\varphi_{\varepsilon}$, so that $v_{\varepsilon} \to v$ and $w_{\varepsilon}\to w$ almost everywhere and
$X_s v_{\varepsilon} \to X_sv$ and $X_s w_{\varepsilon} \to X_sw$ locally in $L^1(\Omega)$; see the formulas in the proof of \cite[Theorem 11.9]{HajKos} for details.
Therefore, using that
$v_{\varepsilon},w_{\varepsilon}$ are locally uniformly bounded we get that $v_{\varepsilon}X_sw_{\varepsilon} \to vX_sw $ and
$w_{\varepsilon}X_sv_{\varepsilon} \to wX_sv $ locally in $L^1(\Omega)$; at this point using the
definition of distributional derivative in the $X_s$-direction the assertion of the lemma follows in this first case.
In a second case we consider the situation  when only one function is bounded, say $v$. We can apply the result of the first case to $v$ and
to the truncated function
$w_k:=\max\{\min\{w,k\},-k\}$, for $k \in \mathbb N$, and the assertion follows using Lebesgue's dominated convergence when letting $k \nearrow \infty$,
and the fact that
$vX_s w, wX_s v$ are supposed to be locally in $L^1(\Omega)$. Finally, the general case follows by the second one
applying the same truncation argument of the second case to one of the two functions.
\end{proof}
\subsection{\bf Maximal Operators.}\label{maximal} Here we present a miscellanea
of various maximal operators and related inequalities. Let $B_0 \subset \er^n$
be a CC-ball. We shall consider, in the
following, the Restricted Maximal Function Operator relative to
$B_0$. This is defined as \eqn{ma1}
$$\M_{B_0}(f)(x):= \sup_{B
\subseteq B_0,\ x \in B} \intav_{B} |f(y)|\ dy\;, $$ whenever $f \in
L^1(B_0)$, where $B$ denotes any CC-ball contained in $B_0$, not
necessarily with the same center, as long as it contains the point
$x$. More generally, if $s\geq 1$ we define \eqn{ma2}
$$\M_{s,B_0}(f)(x):= \sup_{B
\subseteq B_0,\ x \in B} \left(\intav_{B} |f(y)|^s\ dy\right)^{1/s}
$$ whenever $f \in L^s(B_0)$; of course $\M_{1,B_0}\equiv \M_{B_0}$.
Another type of restricted - but ``centered" - maximal operator is
given by \eqn{rem}
$$M_{R}(f)(x):= \sup_{B(x,r)
,r \leq R} \intav_{B(x,r)} |f(y)|\ dy\;. $$

 We recall the
following weak type $(1,1)$ estimate for $\M_{B_0}$: \eqn{weakes}
$$|\{x \in B_0 \ : \ \M_{B_0}(f)(x)\geq \lambda \}| \leq \frac{c_W}{\lambda^\gamma}
\int_{B_0}|f(y)|^{\gamma}\ dy, \quad \mbox{for every }\  \lambda
>0 \ \mbox{and} \  \gamma \geq 1\;,$$ which is valid for any $f \in
L^1(B_0)$; the constant $c_W$ depends only on the homogenous
dimension $Q$ via the doubling constant $C_d$ in \rif{doubling}, and
therefore ultimately on $n$; for this and related issues we refer to
\cite{Steinbig}. A standard consequence of \rif{weakes} is then
\eqn{weakes2}
$$\int_{B_0}|\M_{B_0}(f)|^{\gamma}\ dx \leq \frac{c(Q,\gamma)}{\gamma-1}
\int_{B_0}|f|^{\gamma}\ dx\;, \qquad \mbox{for every}\ \
\gamma>1\;.$$ A straightforward consequence of \rif{weakes2} is the
following similar estimate for $\M_{s,B_0}$: \eqn{weakes3}
$$\int_{B_0}|\M_{s,B_0}(f)|^{\gamma}\ dx \leq \frac{c(Q,\gamma)}{s(\gamma-s)}
\int_{B_0}|f|^{\gamma}\ dx\;, \qquad \mbox{for every}\ \
\gamma>s\;.$$ Finally, we report an inequality due to Hajlasz \&
Strzelecki \cite{HajSt}, see also \cite{HajKos}, Section 3, for
related results.
\begin{prop}\label{HSp} Let $f \in HW^{1,1}(\Omega)$ and $R>0$. Then
there exists an absolute constant $c\equiv c(n)$ such that
$$
|f(x)-f(y)|\leq c[M_{R}(|\X f|)(x)+M_{R}(|\X f|)(y)]d_{cc}(x,y)
$$
whenever $d_{cc}(x,y)\leq R/2\leq \dist(\Omega', \partial \Omega)/2$
and $x, y \in \Omega' \Subset \Omega$.
\end{prop}
\section{Basic regularity}
In this section we summarize and revisit a few regularity
results known for solutions to \rif{due}, in order to get statements
in a form tailored to our later needs.
\subsection{Basic regularity results} The following is a basic result of Capogna \& Danielli \& Garofalo \cite{CDG}, and Lu \cite{Lu}.
\begin{theorem} \label{CDG}
 Let $u \in HW^{1,p}(\Omega)$ be a weak solution to the equation \trif{due} under the assumptions
  \trif{growth}--\trif{ell} with $p>1$. Then there exists a positive H\"older exponent
  $\alpha \equiv \alpha (n,p,L/\nu)$ such that  $ u \in
  C^{0,\alpha}_{\textnormal{loc}}(\Omega).$ In particular, $u$ is
 a locally bounded function, and for every open subset $\Omega' \Subset
 \Omega$ there exists a constant $c$, depending only on $n,p,L/\nu$,
 and $\dist(\Omega', \partial \Omega)$
 but otherwise independent of $\mu \in [0,1]$, of the solutions $u$ and on the vector field $a(\cdot)$, such that
\eqn{inest}
 $$
\|u\|_{L^{\infty}(\Omega')} \leq c
\left(\|u\|_{L^{p}(\Omega)}+\mu\right)\;.
 $$
\end{theorem}
Just let us observe that the validity of \rif{inest} directly
follows from the weak Harnack inequality of Theorem 3.2 in
\cite{CDG}, via a standard covering argument. Now another basic
result, due to Domokos \cite{D}, see also \cite{Ma}.
\begin{theorem} \label{domokos}
 Let $u \in HW^{1,p}(\Omega)$ be a weak solution to the equation \trif{due} under the assumptions
  \trif{growth}-\trif{nondeg}, with $2\leq p <4$.
Then we have $ Tu \in L^{p}_{\textnormal{loc}}(\Omega)\;.$ Moreover,
for every couple of open subsets $\Omega' \Subset \Omega'' \subset
\Omega$ there exists a constant $c$ depending only on
$\dist(\Omega',
\partial \Omega'')$, $n,p,L/\nu$, but otherwise independent of $\mu
\in (0,1]$, of the solutions $u$ and on the vector field $a(\cdot)$,
such that \eqn{apriorid}
$$
\int_{\Omega'} |T u|^p\, dx \leq c \int_{\Omega''} \left( \mu+ |\XXX
u|\right)^p\, dx\;.
$$
In the previous estimate $c \nearrow \infty$ when $p\nearrow 4$.
\end{theorem}
\begin{proof} The proof of the fact that $ Tu \in L^{p}_{\textnormal{loc}}(\Omega) $ is contained in
Theorem 1.2 from \cite{D}. In order to get estimate \rif{apriorid}
we first use the estimate contained in Theorem 1.2 from \cite{D},
that gives
$$
\int_{B_{\gamma R}} |T u|^p\, dx \leq c \int_{B_R} \left(|\XXX
u|^p+|u|^p +\mu^p \right)\, dx\;,
$$
whenever $B_{R} \Subset \Omega$ and where $\gamma \in (0,1)$; the
constant $c$ here depends on $n,p,\ratio,\gamma$ and $R$. Then we
observe that if $u$ weakly solves \rif{due} then so does $u - (u)_{B_R}$
and therefore, applying the previous estimate to this new
function we get \eqn{apriorid3}
$$
\int_{B_{\gamma R}} |T u|^p\, dx \leq c \int_{B_R} \left(|\XXX u|^p
+|u- (u)_{B_R}|^p+\mu^p \right)\, dx\;.
$$
Now, in order to get rid of the integrals involving $u$ in the
previous estimate, we use Jerison's Poincar\'e inequality
\cite{jerison}, that is $ \|u-(u)_{B_R}\|_{L^{p}(B_R)}\leq
c(n,p)R\|\XXX u\|_{L^{p}(B_R)}.$ Now \rif{apriorid} follows by
joining the previous inequality to \rif{apriorid3} and finally using
a standard covering argument.
 Note that the constant $c$ in \rif{apriorid}
critically depends on $\dist(\Omega',
\partial \Omega)$ in the sense that $c\nearrow \infty$ when
$\dist(\Omega',
\partial \Omega) \searrow 0$. The constant $c$ remains bounded when
$\mu \searrow 0$ as a careful inspection of the proof of Theorem 1.2
from \cite{D} reveals.\end{proof} The proof of the following result
can be found in \cite{MM}, Theorem 8.
\begin{theorem} \label{verticaluno}
 Let $u \in HW^{1,p}(\Omega)$ be a weak solution to the equation \trif{due} under the assumptions
  \trif{growth}-\trif{nondeg}, with $2\leq p <4$.
  Assume also that $\mathfrak{X}u \in L^q_{\loc}(\Omega, \er^{2n})$, where $q \geq p$
  satisfies
\eqn{chichi}
$$
p<2+\frac{q}{n+1}\;.
$$
Then we have $ Tu \in L^{\infty}_{\textnormal{loc}}(\Omega).$
Moreover,
 let
$B_{r}= B(x_0, r)\Subset \Omega$, then we have
\begin{equation}
\|Tu\|_{L^{\infty}(B_{\rho})} \le
\left(\frac{c}{r-\rho}\right)^{\frac{\chi}{\chi-1}}
\left(\frac{\|\mu+|\mathfrak{X}u|\|_{L^{q}(B_r)}}{\mu}\right)^{\frac{(p-2)\chi}{2(\chi-1)}}
\|Tu\|_{L^{\frac{2q}{q-p+2}}(B_{r})}\label{apriorivert00},
\end{equation}
 for every $B_\rho =B(x_0,\rho) \subset  B_r$, where \eqn{chichi0}
$$
\chi=\frac{Q}{Q-2}\frac{q-p+2}{q}>1\;.
$$
The constant $c$ only depends on $n,p,\ratio$, being otherwise
independent of the particular solution $u$, the constant $\mu$, and
the vector field $a(\cdot)$, and  $q$.
\end{theorem}
We just remark that conditions \rif{chichi} and \rif{chichi0} are actually equivalent.
\subsection{Difference quotients results.} Before going
on let us clarify a few conventions we shall adopt for the rest of
the paper when dealing with difference quotients as defined in Lemma
\ref{hormander}; such conventions should be kept in mind in the
following especially when reading the proofs of Lemma \ref{diffquo}
and Proposition \ref{horiz} below. By the writing ``$h \to 0$" we
shall implicitly mean ``$h_k \to 0$", since we shall actually have
$h \equiv h_k$ where $\{h_k\}_k$ is a positive decreasing sequence
such that $h_k \to 0$; we shall also eventually, and actually very
often, pass to non-relabeled sub-sequences that will still be
denoted by $\{h_k\}_k$. This will be useful since when letting $h
\to 0$ we shall need to use certain real analysis convergence
results, that are valid up to the passage to sub-sequences. With
such a definition/use of $D_h^{Z}\equiv D_{h_k}^{Z}$, all the
standard properties of difference quotients remain valid, and the
final results are the same, since the point in the use of difference
quotients is approximating real derivatives with discrete finite
difference operators. Finally in the following we shall state
convergence results such as ``$G(x \e^{hZ}) \to G(x)$ in
$L^t_{\loc}(\Omega)$" as $h \to 0$, for some $G \in L^{t}(\Omega)$,
and a smooth vector field $Z$. This must be interpreted as follows:
it is clear that it makes sense to consider $G(x\e^{hZ})$ only
provided $x\e^{hZ} \in \Omega$; on the other hand, for each open subset
$\Omega'' \Subset \Omega$ there exists a number $h_0>0$, depending
on $\Omega''$ and $Z$, such that $x\e^{hZ} \in \Omega$ provided $x
\in \Omega''$ and $|h| \leq h_0$. Therefore by the previous
convergence statement on $G(x \e^{hZ})$ we actually mean $G(x
\e^{hZ}) \to G(x)$ in $L^t(\Omega'')$, where $0 \swarrow |h| \leq
h_0$, for every possible choice of the open subset $\Omega'' \Subset
\Omega$.
%Finally, in the following we shall very often deal with functions
%such as $g(x):=G(xe^Z)$; in this case we have $X g(x)\equiv
%X(G(\cdot e^z))(x)$ being $X$ any other invariant vector field;
%using a clear abuse on notation we shall also denote
%$Xg(x)=X(G(xe^Z))$. See also Lemma \ref{tritra}.

The next lemma summarizes and exploits various difference quotient
arguments and results scattered in \cite{D} and \cite{MM}.
\begin{lemma}\label{diffquo}
 Let $u \in HW^{1,p}(\Omega)$ be a weak solution to the equation \trif{due} under the assumptions
  \trif{growth}-\trif{nondeg}, with $2\leq p <4$. Then we have
\eqn{convbase}
$$
D_{h}^{e_i} \XXX u \to D_i \XXX u \qquad \mbox{in} \qquad
L^{2}_{\loc}(\Omega, \er^{2n})\qquad \mbox{for every} \
i=1,\ldots,2n+1\;,
$$
and therefore \eqn{dese2}
$$
|\XXX \mathfrak{X} u|^2 + |T \mathfrak{X} u|^2 \in
L^1_{\loc}(\Omega)\;.
$$
Moreover \eqn{dese}
$$
(\mu^2+|\mathfrak{X}u|^2)^{\frac{p-2}{2}} \left[|\XXX \mathfrak{X}
u|^2 + |T \mathfrak{X} u|^2\right] \in L^1_{\loc}(\Omega)\;,
$$
and for every choice of open subset $\Omega'  \Subset \Omega''
\Subset \Omega$ there exists a constant $c$ depending only on
$n,p,\ratio$ and $\dist(\Omega', \partial \Omega'')$ such that
\eqn{secT}
$$
\int_{\Omega'} (\mu^2+|\mathfrak{X}u|^2)^{\frac{p-2}{2}} \left[|\XXX
\mathfrak{X} u|^2 + |T \mathfrak{X} u|^2\right] \, dx \leq c
\int_{\Omega''} \left(|\XXX u|^p +|Tu|^p+\mu^p\right)\, dx \;.
$$
In the last inequality the constant $c$ is in particular independent
of $\mu \in (0,1]$, of the solution $u$, and of the vector field
$a(\cdot)$. Finally, we have \eqn{convsxt}
$$a(\XXX u) \in W^{1,\frac{p}{p-1}}_{\loc}(\Omega, \er^{2n})\;.$$
\end{lemma}
\begin{proof} We have to go back to the difference quotient
arguments of \cite{D} and \cite{MM} where the inclusions in
\rif{dese} are proved; in particular we refer to Section 3 of
\cite{MM}. Then, due to the non-degeneracy condition $\mu>0$, we
have that $\XXX u \in W^{1,2}_{\loc}(\Omega,\er^n)$, and this fact
immediately implies \rif{convbase} and \rif{dese2} via Lemma
\ref{hormander}. In order to establish the remaining implications we
shall argue first to get differentiation assertions with respect to
the horizontal directions $X_i$, $i=1,\ldots,2n$; then, in view of
\cite[Theorem 7]{MM} the same arguments will apply when taking
difference quotients with respect to the vertical direction $T$, that
is $D_{h}^{T}$. By the proof of Theorem 1.3 in \cite{D} we see that
the quantity
$(\mu^2+|\mathfrak{X}u(x)|^2+|\mathfrak{X}u(x\e^{hX_i})|^2)^{\frac{p-2}{2}}
|D_{h}^{X_i}\mathfrak{X} u(x)|^2$ remains locally bounded in
$L^1(\Omega, \er^{2n})$, or more precisely, it stays bounded in
$L^1(\Omega', \er^{2n})$ for every $\Omega' \Subset \Omega$, as long
as $h$ is suitably small, depending on $\Omega'$ - see the
``conventions" immediately before the Lemma. Therefore, we also see
that the quantity $D_{h}^{X_i}[(\mu^2+|\XXX
u|^2)^{\frac{p-2}{4}}\XXX u]$ remains locally bounded in
$L^2(\Omega, \er^{2n})$ since an application of Lemma \ref{al2}
gives
\begin{eqnarray*}
&&\int_{\Omega'} \ \left|D_{h}^{X_i}\left(
(\mu^2+|\mathfrak{X}u|^2)^{\frac{p-2}{4}}
 \mathfrak{X}u\right)\right|^2 \, dx
\\ && \qquad \leq c(n,p)\int_{\Omega'} \
(\mu^2+|\mathfrak{X}u(x)|^2+|\mathfrak{X}u(x\e^{hX_i})|^2)^{\frac{p-2}{2}}
 |D_{h}^{X_i}\mathfrak{X}u|^2 \, dx\;.
\end{eqnarray*}
Therefore by Lemma \ref{hormander} we have that
$X_i\left((\mu^2+|\mathfrak{X}u|^2)^{\frac{p-2}{4}}
 \mathfrak{X}u\right) \in L^{2}_{\loc}(\Omega, \er^{2n})$ and  \eqn{comvl1}
$$
D_{h}^{X_i}\left( (\mu^2+|\mathfrak{X}u|^2)^{\frac{p-2}{4}}
 \mathfrak{X}u\right) \to X_i\left(
(\mu^2+|\mathfrak{X}u|^2)^{\frac{p-2}{4}}
 \mathfrak{X}u\right) \quad  \ \mbox{in}\ \ L^2_{\loc}(\Omega,
 \er^{2n})\;.
$$
Moreover, as $\XXX \XXX u, T \XXX u \in L^2_{\loc}(\Omega)$, we may
assume that
$$
(\mu^2+|\mathfrak{X}u(x)|^2+|\mathfrak{X}u(x\e^{hX_i})|^2)^{\frac{p-2}{2}}
|D_{h}^{X_i}\mathfrak{X} u(x)|^2 \to
(\mu^2+2|\mathfrak{X}u(x)|^2)^{\frac{p-2}{2}} |X_i\mathfrak{X}
u(x)|^2,
$$
and $$ D_{h}^{X_i}\mathfrak{X} u(x) \to X_i\mathfrak{X} u(x)
$$
almost everywhere.
 In turn this last fact together with another application of Lemma
\ref{al2}, and the use of \rif{comvl1} allow to apply a well-known
variant of Lebesgue's dominated convergence theorem, finally
yielding
\begin{eqnarray}
\nonumber
&&(\mu^2+|\mathfrak{X}u(x)|^2+|\mathfrak{X}u(x\e^{hX_i})|^2)^{\frac{p-2}{2}}
|D_{h}^{X_i}\mathfrak{X} u(x)|^2 \\ && \qquad \qquad \qquad
\qquad\to (\mu^2+2|\mathfrak{X}u(x)|^2)^{\frac{p-2}{2}} |X_i \XXX
u(x)|^2 \quad \mbox{in} \quad L^1_{\loc}(\Omega)\;.\label{interme}
\end{eqnarray}
Now, according to the notation used Lemma \ref{prop B}, we write \eqn{convl2}
$$
D_{h}^{X_i}(a_i(\mathfrak{X}u))(x) = \int_0^1  D a\left(
\mathfrak{X}u(x)+\tau hD_{h}^{X_i}\mathfrak{X}u(x)\right)
 \,d\tau D_{h}^{X_i} \XXX u (x)\;,
$$
so that $D_{h}^{X_i} a(\XXX u)\to Da(\XXX u)X_i \XXX u$ almost
everywhere. Using \rif{convl2} and again Lemma \ref{prop B}, we have
\eqn{berntrick}
$$
\nonumber |D_{h}^{X_i}(a_i(\mathfrak{X}u))(x)|
\leq c(n,p,L)(\mu^2+|\mathfrak{X}u(x)|^2+|\mathfrak{X}u(x\e^{hX_i})|^2)^{\frac{p-2}{2}}
|D_{h}^{X_i}\mathfrak{X} u(x)|\;.
$$
Therefore, using Lemma \ref{supertrivial} below with $\varepsilon
=1$, we have
\begin{eqnarray*}
|D_{h}^{X_i}(a_i(\mathfrak{X}u))(x)|^{\frac{p}{p-1}} &  \leq &
c(\mu^2+|\mathfrak{X}u(x)|^2+|\mathfrak{X}u(x\e^{hX_i})|^2)^{\frac{p-2}{2}}
|D_{h}^{X_i}\mathfrak{X} u(x)|^2\\ && \qquad \qquad + c
(\mu^2+|\mathfrak{X}u(x)|^2+|\mathfrak{X}u(x\e^{hX_i})|^2)^{\frac{p}{2}}\;.
\end{eqnarray*}
Therefore $D_{h}^{X_i} (a(\XXX u)) \to Da(\XXX u)X_i \XXX u $ in
$L^{\frac{p}{p-1}}_{\loc}(\Omega, \er^{2n})$ follows applying Lemma
\ref{hormander} by \rif{interme} and again the well-known variant of
Lebesgue's dominated convergence theorem, and in a similar way
\rif{convsxt} also follows. Finally, as already mentioned above, the
differentiability results involving $T \XXX u$ follow exactly as
those involving $\XXX \XXX u$; see for instance \cite[Theorem
7]{MM}. In particular the local estimate thereby included implies the one
in \rif{secT} via a standard covering argument. The peculiar
dependence of the constant $c$ comes from a straightforward analysis
of the proofs in \cite{D, MM}.
\end{proof}
\begin{lemma}\label{supertrivial} For every $a,b\geq 0$, $p \geq
2$, and $\varepsilon>0$ we have $ (a^{p-2}b)^{\frac{p}{p-1}}\leq
\varepsilon a^{p-2}b^2 + c(p,\varepsilon) a^p. $
\end{lemma}
\begin{proof} When $p\not =2$ - otherwise the statement is trivial - just write
$$
(a^{p-2}b)^{\frac{p}{p-1}}=
a^{\frac{p(p-2)}{2(p-1)}}a^{\frac{p(p-2)}{2(p-1)}}b^{\frac{p}{p-1}}
$$
and then apply the standard Young's inequality with conjugate
exponents $2(p-1)/p$ and $2(p-1)/(p-2)$.
\end{proof}
\subsection{Higher integrability in Gehring's style} Let us first report a few trivial consequences of assumptions
\rif{growth}-\rif{ell}, see also \cite{Min}, Section 2.2. Since
$p\geq 2$, assumption \rif{ell} implies, for any $z_1,z_2 \in
\er^{2n}$ \eqn{mon2}
$$   c^{-1}
|z_2-z_1|^p \leq  \langle a(z_2)-a(z_1),z_2-z_1\rangle\;.$$ Finally,
inequality \rif{growth}, together with  a standard use of Young's
inequality, yield for every $z \in \er^{2n}$ \eqn{veramon}
$$
c^{-1}(\mu^2+|z|^2)^{\frac{p-2}{2}}|z|^2-c\mu^p \leq \langle
a(z),z\rangle, \qquad \qquad c \equiv c(n,p,\ratio)\geq 1\;.
$$
Then a standard consequence of \rif{growth} and \rif{veramon}
follows in the next
\begin{lemma}\label{cotril} Let $v \in u + HW^{1,p}_0(B_R)$ be the unique solution
to the following Dirichlet problem: \eqn{Dir0}
$$
\left\{
    \begin{array}{cc}
    \textnormal{div} \ a(\XXX v)=0& \qquad \mbox{in }B_R\\
        v= u&\qquad \mbox{on }\partial B_R\,,
\end{array}\right.
$$
where the vector field $a\colon \er^{2n} \to \er^{2n}$ satisfies
\trif{growth}-\trif{ell} for $p>1$, and $B_R \Subset \Omega$ is a
CC-ball. Then there exists a constant $c$ depending only on
$n,p,\ratio$, such that \eqn{cotri}
$$
\int_{B_R} |\XXX v|^p\, dx \leq c \int_{B_R} (\mu+|\XXX u|)^p\,
dx\;.
$$
\end{lemma}
For a related proof using quasiminima see \cite[Chapter 6]{G},
dealing with related, completely standard, Euclidean cases.

Next, a higher integrability result for solutions to \rif{duecz},
together with a first form of inequality \rif{apapd}. Note that here
no upper bound on $p$ is required.
\begin{theorem}\label{gehx} \label{gehring}
 Let $u \in HW^{1,p}(\Omega)$ be a weak solution to the equation \trif{duecz}
 under the assumptions \trif{growth}-\trif{ell}, with $p\geq 2$, and
 $F \in L^{q}_{\loc}(\Omega, \er^{2n})$ for some $q>p$. Then there exists $\tilde{q} >p$, depending only on $n,p, \ratio$, such that $\XXX u \in
 L^{\tilde{q}}_{\loc}(\Omega,\er^{2n})$. Moreover, there exists a constant $c$ depending only on $n,p,\ratio$ such
 that for every CC-ball $B_{2R}\Subset
 \Omega$ the following reverse type inequality:
 \eqn{gerh}
$$
\left(\intav_{B_{R}}|\XXX u|^{q_0}\, dx\right)^{1/q_0}\leq c
\left(\intav_{B_{2R}}(\mu+|\XXX u|)^p\,
dx\right)^{1/p}+c\left(\intav_{B_{2R}}|F|^{q_0}\,
dx\right)^{1/q_0}\;,
$$
holds whenever $p \leq q_0 \leq \tilde{q}$.
\end{theorem}
\begin{proof} The proof more or less works as in the standard
Euclidean setting, and we shall only give a sketch of it; see
\cite[Chapter 6]{G} for the Euclidean case or directly \cite{zat}.
Let $B_R \Subset \Omega$ be a CC-ball, and let us fix a cut-off
function $\eta \in C^{\infty}_0(B_R)$ such that $0\leq \eta \leq 1$,
$\eta \equiv 1 $ in $B_{R/2}$, and $|\XXX \eta |\leq c/R$. The
existence of such a function is as in \cite{CDG}, and in the
specific setting of the Heisenberg group it easily follows from
\rif{ee1} and the definition of $CC$-balls; see Section \ref{CCb}.
Testing \trif{wf} by $\varphi = \eta^p(u-(u)_{B_{R}})$, and using
\rif{growth} and \rif{veramon} in a standard way together with
Young's inequality, we get
$$
\intav_{B_{R/2}}  |\XXX u|^p\, dx \leq
cR^{-p}\intav_{B_{R}}|u-(u)_{B_R}|^p\, dx +
c\intav_{B_{R}}(\mu^p+|F|^p)\, dx\;,
$$
with $c \equiv c (n, p, \ratio).$ See again \cite[Chapter 6]{G}. The
intermediate integral in the last inequality can be estimated by
using the Sobolev-Poincar\'e inequality in the Heisenberg group
\cite{jerison,Lu}, that is
$$
\intav_{B_{R}}|u-(u)_{B_R}|^p\, dx\leq cR^{p}
\left(\intav_{B_{R}}|\XXX u|^{p\sigma}\, dx\right)^{1/\sigma}\;,
$$
for some $\sigma\equiv \sigma(n,p) \in (0,1)$. Therefore, combining
the last two inequalities we get
$$
\intav_{B_{R/2}}  |\XXX u|^p\, dx \leq c\left(\intav_{B_{R}}|\XXX
u|^{p\sigma}\, dx \right)^{1/\sigma}+ c\intav_{B_{R}}(\mu^p+|F|^p)\,
dx\;.
$$
This is a reverse-H\"older inequality with increasing support, in
turn allowing to apply Gehring's lemma in the sub-elliptic setting -
see for instance \cite{zat}. This finally yields the full statement
and \rif{gerh}, after a few elementary manipulations.
\end{proof}

\section{Interpolation and basic integrability}
\subsection{Interpolation inequalities} The following inequality is an end point instance of the general
Gagliardo-Nirenberg inequality in the Euclidean spaces ${\mathbb
R}^n$. For all $f\in C^\infty_0({\mathbb R}^n)$, it holds that
\begin{equation}\label{EUGN} \int_{{\mathbb R}^n}\vert\nabla f\vert^{\gamma+2}\dx
\le c(n,\gamma)\vert\vert f\vert\vert^2_{L^\infty({\mathbb
R}^n)}\int_{{\mathbb R}^n}\vert\nabla f\vert^{\gamma-2}
\vert\nabla^2 f\vert^2\dx, \quad \gamma\ge 0.
\end{equation}
The proof of the above inequality is elementary; indeed, it follows
from integration by parts. In the rest of the section we shall give
the analog of inequality \rif{EUGN} in the Heisenberg group; again,
the proof involves only integration by parts. Actually, we shall
first give a version of \rif{EUGN} for solutions to \rif{due}, that
is the thing we are mainly interested in for the subsequent
developments, and then, as a corollary of the proof given, a more
general Heisenberg group version of \rif{EUGN} will follow in
Theorem \ref{gngeneral} below.

First a few technical preliminaries. Consider the following
truncation operators: \eqn{TT}
$$
\T_{\beta,k}(t):= \left\{
\begin{array}{ccc}
(\mu^2+t)^{\beta} & \mbox{if} & t \in [0,k) \\
\\
(\mu^2+k)^{\beta} & \mbox{if} & t \in [k, \infty)\;.
\end{array}
\right. \qquad \mbox{for} \ \ t,\beta,k \geq 0, \ \ \mu >0\;.
$$
To make the notation easier we shall also denote here $\T_{\beta} \equiv
\T_{\beta,k}$, with the understanding that $k$ is temporarily fixed.

\begin{lemma} For every choice of $\varepsilon \in (0,1)$,
$\alpha,k\geq 0$, and $b\in \er$ it holds that \eqn{YT}
$$
2\T_{p/2+\alpha,k}(t^2)b \leq \varepsilon \T_{p/2+\alpha+1,k} (t^2)
+ \varepsilon^{-1} \T_{p/2+\alpha-1,k}(t^2)b^2\,.
$$
\end{lemma}
\begin{proof} First the case $t^2 < k$. Using the standard quadratic Young's
inequality we have \begin{eqnarray}\nonumber \T_{p/2+\alpha,k}(t^2)b
& =
&\sqrt{\varepsilon}(\mu^2+t^2)^{p/4+\alpha/2+1/2}(1/\sqrt{\varepsilon})(\mu^2+t^2)^{p/4+\alpha/2-1/2}b\\
&\leq& (\varepsilon/2)(\mu^2+t^2)^{p/2+\alpha+1} +
(\varepsilon^{-1}/2)(\mu^2+t^2)^{p/2+\alpha-1}b^2\nonumber \\
&=& (\varepsilon/2) \T_{p/2+\alpha+1,k} (t^2) + (\varepsilon^{-1}/2)
\T_{p/2+\alpha-1,k}(t^2)b^2\;,\label{cT}
\end{eqnarray}
and \rif{YT} follows in this case. When $t^2\geq k$ we write the
previous chain of inequalities substituting $\mu^2+t^2$ by $\mu^2+k$
everywhere in \rif{cT} and \rif{YT} follows in this case too.
\end{proof}
\begin{lemma}\label{G:Nin}
 Let $u \in HW^{1,p}(\Omega)$ be a weak solution to the equation \trif{due}
 under the assumptions \trif{growth}-\trif{nondeg}, with $
2\leq  p < 4.$  Then for all $\sigma \geq 0$ and $\eta \in
C^{\infty}_c(\Omega)$, we have
\begin{eqnarray}
&&\int_\Om \eta^2 \deltaX^\frac{p+2+\sigma}{2} \dx  \leq  c\int_\Om
\big(  \eta^2\mu^2 + |\X \eta|^2 u^2 \big)
    \deltaX^\frac{p+\sigma}{2} \dx \nonumber\\
&& \qquad \qquad \qquad \qquad  + c\|u\|^2_{L^{\infty}(\supp \eta)}
\int_\Om \eta^2
    \sum_{s=1}^{2n} \deltaXi^\frac{p-2+\sigma}{2} |X_s X_su|^2
    \dx,\label{HGN}
\end{eqnarray}
where $c\equiv c(n,p,\sigma)>0$.
\end{lemma}
\begin{proof} For ease of notation in the following we let $\alpha := \sigma/2$. First let us observe that the very definition in
\rif{TT} implies that the map $t \to \T_{p/2+\alpha}(t^2)t$ is
globally Lipschitz continuous and therefore the chain rule in the
Heisenberg group - see \cite{CDG} - and the fact that $\Xu \in
W^{1,2}_{\loc}(\Omega,\er^{2n})$ as given by Lemma \ref{diffquo},
imply that \eqn{trusob}
$$
\eta^2\T_{p/2+\alpha}((X_s u)^2)X_su \in W^{1,2}_{\loc}(\Omega)\,,
$$
holds for every $s \in \{1,\ldots,2n\}$. Now, inclusion
\rif{trusob} allows for the following integration by parts:
\begin{eqnarray}
\nonumber && P_0 := \int_{\Omega}   \eta^2
\T_{p/2+\alpha}((X_su)^2)(X_su)^2\, dx
=\int_{\Omega}   \eta^2 \T_{p/2+\alpha}((X_s u)^2)X_su X_s u\, dx\\
\nonumber && = - \int_{\Omega}    u \eta^2 \T_{p/2+\alpha}((X_s
u)^2) X_s X_s u\, dx - 2 \int_{\Omega}    u \eta^2
\T_{p/2+\alpha}'((X_s
u)^2)(X_su)^2 X_s X_s u\, dx  \\
&&\qquad \qquad   - 2 \int_{\Omega}  u \eta X_s \eta
\T_{p/2+\alpha}((X_s u)^2) X_s u\, dx =: P_1 + P_2
+P_3\;.\label{PP00}
\end{eqnarray}
Of course we used \rif{adad}. Let us now estimate the three
integrals defined in \rif{PP00}, that is $P_1,P_2$ and $P_3$. With
$\varepsilon \in (0,1)$, by means of \rif{YT} we have
\begin{eqnarray*}
 |P_1| & \leq &  \varepsilon \int_{\Omega}  \eta^2  \T_{p/2+\alpha+1}((X_s u)^2)\, dx
 \\ && \qquad \qquad \qquad \qquad + c\|u\|_{L^{\infty}(\supp \eta)}^2\int_{\Omega} \eta^2 \T_{p/2+\alpha-1}((X_s u)^2)| X_s
X_s u|^2\, dx \\ &\leq & \varepsilon P_0 + \int_{\Omega} \eta^2\mu^2
\T_{p/2+\alpha}((X_s u)^2)\, dx \\ && \qquad \qquad \qquad \qquad +
c\|u\|_{L^{\infty}(\supp \eta)}^2\int_{\Omega} \eta^2
\T_{p/2+\alpha-1}((X_s u)^2)| X_s X_s u|^2\, dx  \;,
\end{eqnarray*}
as, obviously,  $\T_{p/2+\alpha+1}((X_s u)^2)\leq
\T_{p/2+\alpha}(X_s u)^2(\mu^2+(X_s u)^2)$. In the previous
inequality we have $c \equiv c (\varepsilon)$. The estimate of $P_2$
requires slightly more care; by Young's inequality and the
definition in \rif{TT}, we have
\begin{eqnarray*}
 |P_2| & \leq &  (p+2\alpha) \|u\|_{L^{\infty}(\supp \eta)}\int_{\{(X_s u)^2\leq k\}}  \eta^2  (\mu^2+(X_su)^2)^{\frac{p+2\alpha}{2}}|X_s X_s  u|\, dx
  \\ &\leq & \varepsilon \int_{\{(X_s u)^2\leq k\}}  \eta^2  (\mu^2+(X_su)^2)^{\frac{p+2+2\alpha}{2}}\, dx
  \\ && \qquad \qquad + c \|u\|_{L^{\infty}(\supp \eta)}^2\int_{\{(X_s u)^2\leq k\}}  \eta^2  (\mu^2+(X_su)^2)^{\frac{p-2+2\alpha}{2}}|X_s X_s  u|^2\,
  dx\\ &\leq & \varepsilon \int_{\Omega}  \eta^2  \T_{p/2+\alpha}((X_s u)^2)(\mu^2+(X_su)^2)\, dx
  \\ && \qquad \qquad + c\|u\|_{L^{\infty}(\supp \eta)}^2\int_{\{(X_s u)^2\leq k\}}  \eta^2  (\mu^2+(X_su)^2)^{\frac{p-2+2\alpha}{2}}|X_s X_s  u|^2\, dx
  \\ &\leq &  \varepsilon P_0 + \int_{\Omega}  \eta^2\mu^2
\T_{p/2+\alpha}((X_s u)^2)\, dx
  \\ && \qquad \qquad + c\|u\|_{L^{\infty}(\supp \eta)}^2\int_{\Omega}  \eta^2  \T_{p/2+\alpha-1}((X_s u)^2)|X_s X_s  u|^2\,
  dx\;,
\end{eqnarray*}
where again $c \equiv c (p,\varepsilon,\sigma)$. Finally, the
estimation of $P_3$; again using standard Young's inequality
\begin{eqnarray*}
 |P_3| & \leq &  \int_{\Omega}  \eta |\XXX \eta||u| \T_{p/2+\alpha}((X_s u)^2)|X_su|\, dx
 \\  &\leq & \varepsilon P_0 +
  c(\epsilon) \int_{\Omega}   |\XXX \eta|^2u^2 \T_{p/2+\alpha}((X_s u)^2)\,
 dx  \;.
\end{eqnarray*}
Connecting the inequalities found for $P_1,P_2,P_3$ to \rif{PP00} we
have
\begin{eqnarray*}
P_0 & \leq & 3\varepsilon P_0 +  c \int_{\Omega}  \big(  \eta^2\mu^2
+ |\X \eta|^2 u^2 \big) \T_{p/2+\alpha}((X_s u)^2)\,
 dx\\
 && \qquad \qquad + c\|u\|_{L^{\infty}(\supp \eta)}^2\int_{\Omega}  \eta^2  \T_{p/2+\alpha-1}((X_s u)^2)|X_s X_s  u|^2\,
 dx\;,
\end{eqnarray*}
where $c$ depends on $n,p,\sigma$ and $\epsilon$. Observing that all
the quantities involved in the previous inequality are finite as
$\XXX u \in W^{1,2}_{\loc}(\Omega,\er^{2n})$, taking
$\varepsilon=1/6$, recalling that $\alpha = \sigma/2$, an easy
manipulation now yields
\begin{eqnarray}
\nonumber && \int_{\Omega}   \eta^2
\T_{p/2+\sigma/2,k}((X_su)^2)(\mu^2+(X_su)^2)\, dx \\ \nonumber &&
\qquad \qquad \leq
  c\int_{\Omega}
\big(  \eta^2\mu^2 + |\X \eta|^2 u^2 \big)\T_{p/2+\sigma/2,k}((X_s
u)^2)\,
 dx\\
 && \hspace{3cm}+ c\|u\|_{L^{\infty}(\supp \eta)}^2\int_{\Omega}  \eta^2  \T_{p/2+\sigma/2-1,k}((X_s u)^2)|X_s X_s  u|^2\,
 dx\;,\label{finTx}
\end{eqnarray}
for any $s \in \{1,\ldots,2n\}$, where $c$ depends only on $n,p$ and
$\sigma$. At this point \rif{HGN} follows summing up inequalities
\rif{finTx} for $s \in \{1,\ldots,2n\}$ and eventually letting $k
\nearrow \infty$, using the monotone convergence theorem.
\end{proof}
\begin{remark} In the previous proof we never used that
$u$ is a solution of \rif{due} but only that $\Xu$ locally belongs to $
HW^{1,2}(\Omega,\er^{2n})$, and that $u$ is locally bounded.
Therefore neither the ellipticity ratio $\ratio$, nor the degeneracy
parameter $\mu$, appear in \rif{HGN}.
\end{remark}
We conclude with a more general statement extending the Euclidean
one in \rif{EUGN}, which is at this stage an obvious consequence of
the proof of Lemma \ref{G:Nin}, and of the previous remark.
\begin{theorem}\label{gngeneral}
Let $\sigma$ be a non-negative number and $p \geq 2$. Then for all
$u\in C^\infty(\Omega)$ and $\eta\in C^\infty_c(\Omega)$, we have
$$
\int_\Om \eta^2 |\Xu|^{p+2+\sigma} \dx \leq c \int_\Om |\X \eta|^2
u^2
   |\Xu|^{p+\sigma} \dx + c \int_\Om \eta^2 u^2
    \sum_{s=1}^{2n} |X_su|^{p-2+\sigma} |X_s X_su|^2 \dx,
$$
where $c\equiv c(n,p,\sigma)>0$.
\end{theorem}
\subsection{Basic higher integrability} As an immediate corollary of Lemma \ref{G:N} applied with $\sigma=0$,
and of Lemma \ref{diffquo}, we gain a first higher integrability
property of solutions to \rif{due}:
\begin{lemma}\label{G:N}
 Let $u \in HW^{1,p}(\Omega)$ be a weak solution to the equation \trif{due}
 under the assumptions \trif{growth}-\trif{nondeg}, with $
2\leq  p < 4.$  Then \eqn{basichi1}
$$ \Xu \in
L^{p+2}_{\loc}(\Omega,\er^{2n})\,.
$$
Moreover, for every couple of open subsets $\Omega' \Subset \Omega''
\Subset \Omega$ there exists a constant $c$ depending only on
$n,p,\ratio$, $\dist(\Omega',
\partial \Omega'')$, and $\|u\|_{L^\infty(\Omega'')}$, but independent of
$\mu$, of the solution $u$, and of the vector field $a(\cdot)$, such
that \eqn{estint100}
$$
\int_{\Omega'} |\Xu|^{p+2}\dx \leq c \int_{\Omega''}
\big(|\Xu|^{p}+|Tu|^{p}+\mu^p\big)\dx\;.
$$
\end{lemma}
Observe that \rif{estint100} immediately follows by \rif{HGN} with
$\sigma =0$, and by \rif{secT} via a standard covering argument -
note that the choice of $\eta$, $\Omega'$ and $\Omega''$ in
\rif{HGN} and \rif{estint100} is arbitrary.
\section{Caccioppoli type inequalities}
In this section we shall derive a few preliminary energy estimates,
or so called Caccioppoli type inequalities, for the horizontal and
vertical gradients $\Xu$ and $Tu$ respectively. We shall modify some
of the arguments introduced in \cite{MM} in order to find new types of Caccioppoli inequalities -
that is, energy estimates.
In turn these will be at the core of the main iteration in Section 7.

\subsection{Smooth truncation operators}  We shall start defining certain ``smooth truncation
operators" which are already used, in a slightly different from, in
\cite{MM}. We define
\begin{equation}\label{gF}
g_{\alpha,k}(t)=\frac{k(\mu^2+t)^\alpha}{k+(\mu^2+t)^\alpha}\hspace{15mm}
t,\alpha\geq 0, \ \ \mu >0 \qquad k \in \mathbb N.\end{equation} We
have that
\begin{equation}\label{boundg}
0\leq g_{\alpha,k}(t)\leq \min\{k, (\mu^2+t)^\alpha\}, \quad
\mbox{and} \quad
0\leq g_{\alpha,k}(t)\leq g_{\alpha,k+1}(t)
\end{equation}
hold for every $k \in \mathbb N$, and moreover
\begin{equation}\label{conv1}
\lim_{k \to \infty} g_{\alpha,k}(t)=(\mu^2+t)^\alpha.\end{equation}
A few elementary computations, actually a variant of the ones
already presented in \cite{MM}, Section 5.2, give that
\begin{equation}\label{g2}
g'_{\alpha,k}(t)(\mu^2+t)\leq \alpha g_{\alpha,k}(t),
\qquad
|g''_{\alpha,k}(t)|(\mu^2+t)\leq 3(\alpha+1) g'_{\alpha,k}(t)\,.
\end{equation}
We shall also deal with the following family of functions:
\begin{equation}\label{WF}
W_{\alpha,k}(t):=2g'_{\alpha,k}(t)t+g_{\alpha,k}(t),\qquad
t,\alpha\geq 0 \qquad k \in \mathbb N.\end{equation} Using the first inequality in \rif{g2}
and then the first in \rif{boundg}, together with the fact that
$g'_{\alpha,k}(t)\geq 0$, we find
 \begin{equation}\label{boundW}
 g_{\alpha,k}(t)\leq W_{\alpha,k}(t)\leq (2\alpha+1)g_{\alpha,k}(t)\leq (2\alpha+1)k
.\end{equation} Moreover, taking the second estimate in \rif{g2} into account, and then
again the first estimate in \rif{g2}, we also find
\begin{equation}\label{W3}
|W'_{\alpha,k}(t)|t\leq|W'_{\alpha,k}(t)|(\mu^2+t)\leq 3(\alpha+1)
W_{\alpha,k}(t).
\end{equation}
Using that $g'_{\alpha,k}(t)\leq g'_{\alpha,k+1}(t)$ for every
$k,\alpha$ and $t$, taking the second inequality in \rif{boundg} into account we have
\begin{equation}\label{incW}
 W_{\alpha,k}(t)\leq W_{\alpha,k+1}(t)\qquad \qquad \text{for all} \ k \in \mathbb N.\end{equation}
Finally, by \rif{conv1} it follows that \begin{eqnarray}\nonumber
(\mu^2+t)^{\alpha} \leq
 \lim_{k \to \infty}  W_{\alpha,k}(t) & = & (\mu^2+t)^{\alpha-1}[2\alpha t +(\mu^2+t)] \\
 & \leq & 3(\alpha+1)(\mu^2+t)^\alpha.\label{limW}
\end{eqnarray}
\subsection{The horizontal Caccioppoli inequality} Here we prove a
suitable energy estimate involving powers of the natural quantity
$(\mu^2+|\Xu|^2)^{1/2}$, that is ``the weight"  of the equation
\rif{subnondeg}.
\begin{lemma}\label{XX sigma lem}
 Let $u \in HW^{1,p}(\Omega)$ be a weak solution to the equation \trif{due}
 under the assumptions \trif{growth}-\trif{nondeg}, with $
2\leq  p < 4.$ Let $\sigma \geq 2$ and assume that \eqn{startint}
$$\Xu
\in L^{p+\sigma}_{\loc}(\Omega,\er^{2n}), \qquad \mbox{and} \qquad
\vert\Xu\vert^{p-2+\sigma}|Tu|^2\in L^1_{\loc}(\Omega)\;.$$ Then for
all $\eta\in C^\infty_c(\Omega)$, we have
\begin{align} \label{XX sigma est}
\int_{\Omega} \eta^2 &\weight
\sum_{s=1}^{2n} \deltaXi^\frac{\sigma}{2} |\X X_s u|^2 \dx \nonumber\\
&\leq c (\sigma +1) \int_\Om (|\X \eta|^2 + \eta |T\eta|)
\deltaX^\frac{p+\sigma}{2}\dx
%\sum_{i=1}^{2n} \deltaXi^\frac{\sigma}{2} \dx
\nonumber\\
&\qquad \qquad+ c (\sigma +1)^3 \int_\Om  \eta^2
\deltaX^{\frac{p-2+\sigma}{2}} |Tu|^2 \dx,
%\sum_{i=1}^{2n} \deltaXi^\frac{\sigma}{2} \dx
%\nonumber \\
%&\qquad +
%C (\sigma + 1)^3
%\int_\Om  \eta^2
%\deltaX^\frac{p}{2} |Tu|^2
%\sum_{i=1}^{2n} \deltaXi^\frac{\sigma-2}{2} \dx.
\end{align}
and moreover
\begin{align} \label{XX sigma estpre}
\int_{\Omega} \eta^2 &\weight
\sum_{s=1}^{2n} \deltaXi^\frac{\sigma}{2} |\X X_s u|^2 \dx \nonumber\\
&\leq c (\sigma +1) \int_\Om (|\X \eta|^2 + \eta |T\eta|)
\sum_{s=1}^{2n} (\mu^2+|X_su|^2)^{\frac{p+\sigma}{2}}\dx
%\sum_{i=1}^{2n} \deltaXi^\frac{\sigma}{2} \dx
\nonumber\\
&\qquad \qquad+ c (\sigma +1)^3 \int_\Om  \eta^2 \sum_{s=1}^{2n}
(\mu^2+|X_su|^2)^{\frac{p-2+\sigma}{2}} |Tu|^2 \dx\;.
%\sum_{i=1}^{2n} \deltaXi^\frac{\sigma}{2} \dx
%\nonumber \\
%&\qquad +
%C (\sigma + 1)^3
%\int_\Om  \eta^2
%\deltaX^\frac{p}{2} |Tu|^2
%\sum_{i=1}^{2n} \deltaXi^\frac{\sigma-2}{2} \dx.
\end{align}
Both in \trif{XX sigma est} and in \trif{XX sigma estpre} we have
$c\equiv c(n,p,\ratio)>1$, and in particular the constant $c$ does
not depend on $\mu,u$, and on the vector field $a(\cdot)$.
\end{lemma}
\begin{proof} With the definition in \rif{gF}, in the following we shall abbreviate $g(\cdot)\equiv g_{\sigma/2,k}$,
for a fixed $k \in \mathbb N$, while, according to \rif{WF}, we
shall denote $W(\cdot):=2g'(\cdot)t+g(\cdot)$. For the rest of the
proof all the constants denoted by $c$ or the like will depend only
on $n,p, \ratio$, and will be independent of $\mu, u$, $k$ and
$\sigma$. Any dependence on $\sigma$ in the following inequalities
will be explicitly displayed. We start by applying Lemma
\ref{lemma:diff quot eq} with the choice $Z=X_s$ for $s \in
\{1,\ldots,n\}$; for every $\varphi \in C^{\infty}_c(\Omega)$, and
$h\not=0$ accordingly small, we arrive at
\begin{equation}\label{XX:1}
    \int_\Omega \langle  \Ds a (\Xu),\X \varphi\rangle \dx
        = - \int_\Omega a_{n+s}(\Xu)(x\e^{hX_s}) T\varphi \dx .
\end{equation}
We test \eqref{XX:1} with $\varphi \equiv \phi_1:=\eta^2 g(|\Ds
u|^2) \Ds u$, for $s \in \{1,\ldots,n\}$. By a simple density
argument this is an admissible test function in \rif{XX:1}, since
$g$ is bounded, and moreover $Tu \in L^{p}_{\loc}(\Omega)$. We
obtain, for every $i \in \{1,\ldots,2n\}$
$$
X_i \phi_1 = 2\eta X_i \eta \, g(|\Ds u|^2) \Ds u + \eta^2 W(|\Ds
u|^2) X_i \Ds u
$$
and
$$
T \phi_1 = 2\eta T\eta \, g(|\Ds u|^2) \Ds u + \eta^2 W(|\Ds u|^2) T
\Ds u\;.
$$
Inserting the last two equalities into \eqref{XX:1} yields
\begin{align} \label{XX:2}
\int_\Omega \eta^2 \sum_{i=1}^{2n} \DhX a_i(\Xu)
    & X_i \Ds u  W(|\DhX u|^2)\dx \nonumber \\
&= - 2 \int_\Omega \eta
    \sum_{i=1}^{2n}
    \DhX a_i (\Xu)\,X_i \eta
    g(|\DhX u|^2) \Ds u
    \dx\nonumber\\
&\quad -  2 \int_\Omega \eta T \eta
    a_{n+s} (\Xu)(x \e^{hX_s})
    g(|\DhX u|^2) \Ds u
    \dx\nonumber\\
&\quad - \int_\Omega \eta^2
    a_{n+s} (\Xu)(x \e^{hX_s})
    W(|\DhX u|^2) T \DhX u \dx\;.
\end{align}
As we are dealing with difference quotients in the horizontal
directions, the operators $\X$ and $\Ds$ do not commute. Therefore
we need to use identity \eqref{eq:hom lemma 1}; this gives, for
every $j \in \{1,\ldots,2n\}$
\begin{equation*}
    (\Ds X_j u)(x)
         =
         X_j(\Ds u)(x)+[X_s,X_j]u(x\e^{hX_s})\, .
\end{equation*}
Now use Lemma \ref{prop B} with $Z \equiv X_s$, and adopting the
related notation in \rif{atz}, we have
\begin{align} \label{XX:3}
\DhX a_i(\Xu)(x) &=
\sum_{j=1}^{2n} a_{i,j}^{X_s}(x) \DhX X_j u(x) \nonumber\\
&= \sum_{j=1}^{2n} a_{i,j}^{X_s}(x)
\big[ X_j \DhX u(x) + [X_s,X_j] u(x\e^{hX_s}) \big] \nonumber\\
&= \sum_{j=1}^{2n} a_{i,j}^{X_s}(x) X_j \DhX u(x) + a_{i,n+s}^{X_s}(x) Tu(x\e^{hX_s})\,.
\end{align}
From now on in every occurence of the symbol $\sum$ the indexes
$i,j$ will run from $1$ to $2n$. Joining \eqref{XX:2} and
\eqref{XX:3} we obtain
\begin{align}\label{XX test}
\int_\Omega \eta^2
    \sum_{i,j} &
    a_{i,j}^{X_s} (x) \,
    X_j \DhX u \,
    X_i \DhX u \,
    W(|\DhX u|^2)
    \dx \nonumber \\
&=  -
    \int_\Omega \eta^2
    \sum_{i}
    a_{i,n+s}^{X_s}(x) \,X_i \DhX u \,
    Tu(x\e^{hX_s}) \,
    W(|\DhX u|^2)
    \dx   \nonumber  \\
&\quad  - 2 \int_\Omega \eta
    \sum_{i,j}
    a_{i,j}^{X_s} (x) \,
    X_i \eta \,
    X_j \DhX u\,
    g(|\DhX u|^2)\DhX u\,
    \dx \nonumber\\
&\quad - 2 \int_\Omega
    \eta
    \sum_{i}
    a_{i,n+s}^{X_s}(x) \,X_i \eta \,
    Tu(x\e^{hX_s}) \,
    g(|\DhX u|^2) \,
    \Ds u
    \dx \nonumber\\
&\quad -  2 \int_\Omega \eta T \eta
    a_{n+s} (\Xu)(x \e^{hX_s}) \,
    g(|\DhX u|^2) \,
    \DhX u
    \dx\nonumber\\
&\quad - \int_\Omega \eta^2
    a_{n+s} (\Xu)(x \e^{hX_s})
    W(|\DhX u|^2) T \DhX u \dx\,.
\end{align}
A completely similar equation, with $Y_s=X_{n+s}$ replacing $X_s$
everywhere in \rif{XX test}, can be obtained by testing \eqref{XX:1}
with $\varphi \equiv \phi_2 :=\eta^2 g(|\Dsy u|^2) \Dsy u$. We
finally sum up the resulting two equalities over $s=1,2, \ldots, n$,
thereby obtaining
\begin{align}\label{XX pre}
\int_\Omega \eta^2
    \sum_{s=1}^{2n}
    \sum_{i,j}
    a_{i,j}^{X_s} (x) \,
&   X_j \DhX u \,
    X_i \DhX u \,
    W(|\DhX u|^2)
    \dx \nonumber \\
&=  - 2 \int_\Omega \eta
    \sum_{s=1}^{2n} \sum_{i,j}
a_{i,j}^{X_s} (x) \, X_i \eta\,X_j \DhX u
    g(|\DhX u|^2) \,
    \DhX u
\dx \nonumber \\
&\quad -   \int_\Omega
    \eta^2 \,
    \sum_{s=1}^{n} \sum_{i}
    \Big(
    a_{i,n+s}^{X_s}(x)
    Tu(x\e^{hX_s}) \,
    W(|\DhX u|^2) \,
    X_i \DhX u \nonumber \\
&\hspace{3cm} -
    a_{i,s}^{Y_s}(x)
    Tu(x\e^{hY_s}) \,
    W(|\DhY u|^2) \,
    X_i \DhY u
    \Big) \dx \nonumber \\
&\quad -   2 \int_\Omega
    \eta \,
    \sum_{s=1}^{n} \sum_{i}
    X_i \eta
    \Big(
    a_{i,n+s}^{X_s}(x)
    Tu(x\e^{hX_s}) \,
    g(|\DhX u|^2) \,
    \DhX u \nonumber \\
&\hspace{3cm}-
    a_{i,s}^{Y_s}(x)
    Tu(x\e^{hY_s}) \,
    g(|\DhY u|^2) \,
    \DhY u
    \Big) \dx \nonumber \\
&\quad -   2 \int_\Omega
    \eta T \eta \,
    \sum_{s=1}^{n}
    \Big(
    a_{n+s}(\Xu)(x \e^{hX_s})
    g(|\DhX u|^2) \,
    \DhX u \nonumber \\
&\hspace{3cm}-
    a_s(\Xu)(x \e^{hY_s})
    g(|\DhY u|^2) \,
    \DhY u
    \Big) \dx \nonumber \\
&\quad -
    \int_\Omega
    \eta^2 \,
    \sum_{s=1}^{n}
    \Big(
    a_{n+s}(\Xu)(x \e^{hX_s})
    W(|\DhX u|^2) \,
    T \DhX u \nonumber \\
&\hspace{3cm} -
%\qquad \qquad \qquad +
    a_s(\Xu)(x \e^{hY_s})
    W(|\DhY u|^2) \,
    T \DhY u
    \Big) \dx \nonumber \\
&=: I_1 + I_2 + I_3 + I_4 +I_5\;.
\end{align}
We now proceed estimating the various terms spreading-up from
\eqref{XX pre}. To estimate the left hand side from below we use
\eqref{lowerb} obtaining
\begin{multline} \label{XX lhs}
\text{l.h.s. of} \  (\ref{XX pre}) \\
\geq c^{-1} \int_\Omega \eta^2 \sum_{s=1}^{2n}
    \weightX
    W(|\DhX u|^2) |\X \DhX u|^2 \dx\,,
\end{multline}
with $c \equiv c(n,p,\ratio)\geq 1$. In order to estimate the
integrals $I_1,\ldots,I_4$ we use \eqref{Btilde}, \eqref{upperb} and
Young's inequality, obtaining for $\varepsilon \in (0,1)$ that
\begin{eqnarray*}\nonumber
 && |I_1|\\
 && \leq c \int_\Omega \eta \, |\X \eta| \,
    \sum_{s=1}^{2n}
    \weightX \,
    g(|\DhX u|^2) \,
    |\DhX u| \,
    |\X \DhX u| \dx \nonumber \\
&&\leq \epsilon \int_\Omega
    \eta^2 \,
    \sum_{s=1}^{2n}
    \weightX \,
    W(|\DhX u|^2) \,
    |\X \DhX u|^2
    \dx \nonumber \\
&&\quad + c(\epsilon) \int_\Omega
    |\X \eta|^2 \,
    \sum_{s=1}^{2n}
    \weightX\,
    W(|\DhX u|^2)\,
    |\DhX u|^2 \dx\;,\label{XX rhs1}
\end{eqnarray*}
and, in a similar way
\begin{eqnarray*}\nonumber
&&|I_2|\\& &\leq
    c \int_\Om
    \eta^2 \sum_{s=1}^{2n}
    \weightX \,
    W(|\DhX u|^2) \\ && \hspace{8cm}\cdot
    |Tu(x\e^{hX_s})|\,
    |\X \DhX u| \dx \nonumber \\
&&\leq
    \epsilon
    \int_\Om \eta^2
    \sum_{s=1}^{2n}
    \weightX \,
    W(|\DhX u|^2) \,
    |\X \DhX u|^2 \dx  \nonumber \\
&& \qquad + c(\epsilon)
    \int_\Om
    \eta^2 \,
    \sum_{s=1}^{2n}
    \weightX \\&& \hspace{7cm}\cdot
    W(|\DhX u|^2) \,
    |Tu(x\e^{hX_s})|^2 \dx\label{XX rhs2}\;,
\end{eqnarray*}
\begin{eqnarray*} \nonumber && |I_3| \\ & &\leq
    c \int_\Om
    \eta |\X \eta| \,
    \sum_{s=1}^{2n}
    \weightX \,
    |Tu(x\e^{hX_s})| \\ && \hspace{8cm}\cdot
    g(|\DhX u |^2) \,
    |\DhX u| \dx \nonumber \\
&&\leq c \int_\Om
    |\X \eta|^2
    \sum_{s=1}^{2n}
    \weightX \,
    W(|\DhX u|^2)\,
    |\DhX u|^2 \dx \nonumber \\
&& \qquad + c \int_\Om
    \eta^2
    \sum_{s=1}^{2n}
    \weightX \,
    W(|\DhX u|^2)\,
    |Tu(x\e^{hX_s})|^2 \dx\,,\label{XX rhs3}
\end{eqnarray*}
and finally
$$
|I_4| \leq c \int_\Om
    \eta |T\eta| \sum_{s=1}^{2n}
    \big( \mu^2 + |\Xu(x\e^{hX_s})|^2 \big)^\frac{p-1}{2} \,
    W(|\DhX u |^2) |\DhX u| \dx\;.
$$
The estimation of the last integral $I_5$ in \eqref{XX pre} needs
slightly more care, and will be done later. We have that $\Xu \in
L^p(\Omega,\er^{2n})$ and, by Theorem \ref{domokos} we also have $Tu
\in L^p_{\loc}(\Omega)$, while Lemma \ref{diffquo} gives $\X \Xu \in
L^2_{\loc}(\Om, \R^{2n\times 2n})$, therefore, using also
\rif{convbase}, up to passing to non-relabeled sub-sequences, we may
assume for every $s=1,\ldots,2n$ that
\begin{align}\label{ccoo}
\Xu(x \e^{hX_s}) &\to \Xu(x) &\quad
\text{in $L^p_{\loc}(\Om, \R^{2n})$ and a.e.} \nonumber \\
Tu(x \e^{hX_s}) &\to Tu(x) &\quad
\text{in $L^p_{\loc}(\Om)$ and a.e.}  \\
\X \DhX u(x) &\to \X X_s u(x) &\quad \text{in $L^2_{\loc}(\Om,
\R^{2n})$ and a.e.}\label{ccoo2}
\end{align}
See also Section 3.2. The convergence statement in \rif{ccoo2} needs perhaps an explanation; for $i = 1,\ldots,2n$, write
$X_i  \DhX u(x)= \DhX X_iu(x)+ [X_i,X_s]u(x\e^{hX_s})$, according to \rif{eq:hom lemma 1}. Then, using the result of Lemma
\ref{hormander}, the fact that $p\geq 2$,
and the convergence in \rif{ccoo}, we have that
$X_i \DhX u(x) \to X_sX_iu(x)+[X_i,X_s]u(x)=X_iX_su(x)$ locally in $L^2(\Omega)$, and, up to a sub-sequence, almost everywhere. Therefore \rif{ccoo2} is completely proved.

Now we want to pass to the limit with $h\to 0$
in \eqref{XX pre} taking into account the estimates for the
integrals $I_1,\ldots,I_4$. Absorbing the terms with $\epsilon$ in
the l.h.s., applying Fatou's lemma for the resulting l.h.s., and
Lebesgue's dominated convergence theorem for the r.h.s. - keep in
mind that $W(\cdot)$ is bounded by \rif{boundW} - we obtain
\begin{align} \label{pre XX}
\int_\Omega \eta^2 \sum_{s=1}^{2n}
    & \weight
    W(|X_s u|^2) |\X X_s u|^2
    \dx \nonumber\\
&\leq c \int_\Omega
    \Big( |\X \eta|^2 + \eta |T\eta| \Big) \,
    \deltaX^\frac{p}{2}\,
    \sum_{s=1}^{2n}
    W(|X_s u|^2)
    \dx \nonumber \\
&\qquad + c \int_\Om
    \eta^2 \,
    \weight\,|Tu|^2
    \sum_{s=1}^{2n}
    W(|X_s u|^2)
    \dx \nonumber\\
&\qquad + \limsup_{h \to 0} \left|
   \int_\Omega
    \eta^2 \,
    \sum_{s=1}^{n}
    a_{n+s}(\Xu)(x \e^{hX_s})
    W(|\DhX u|^2) \,
    T \DhX u
    \dx \right| \nonumber \\
&\qquad + \limsup_{h \to 0}\left|
  \int_\Omega
    \eta^2 \,
    \sum_{s=1}^{n}
    a_s(\Xu)(x \e^{hY_s})
    W(|\DhY u|^2) \,
    T \DhY u
    \dx \right|.
\end{align}
Now we compute and estimate the last two limits, that actually
exist, in the previous inequality; we shall concentrate on the
second-last one, similar arguments working for the last one. By
Lemma \ref{diffquo} we know that $\X Tu \in L^2_{\loc}(\Omega,
\er^{2n})$. Therefore, for every $s \in\{1,\ldots,2n\}$ we have that
\eqn{stella}
$$
\DhX Tu \to X_s Tu \quad \text{in} \quad L^2_{\loc}(\Omega) \quad
\text{as} \quad h \to 0\;.
$$
Using Young's inequality we can bound the term under the integral
sign as follows:
\begin{align}
\big| a_{n+s}(\Xu)(x \e^{hX_s}) W(|\DhX u|^2) & \,
T \DhX u \big| \nonumber \\
&\leq c(\sigma,k) \deltaXhi^\frac{p-1}{2}\,
|\DhX Tu|\nonumber \\
&\leq c(\sigma,k) \big[ \deltaXhi^\frac{2p-2}{2} + |\DhX Tu|^2
\big]\,,\label{convle}
\end{align}
where we used \rif{boundW} and that $\alpha =\sigma/2$. Since
$\sigma \geq 2$ then \rif{startint} implies that $\Xu \in
L_{\loc}^{p+2}(\Om, \R^{2n})$ and moreover $p<4$ implies that we can
use the fact that $ 2p -2 < p+2$. Therefore $\Xu \in
L_{\loc}^{2p-2}(\Om, \R^{2n})$ and hence
$$
\Xu(x\e^{hX_s}) \to \Xu(x) \quad \text{in
$L^{2p-2}_{\loc}(\Omega,\er^{2n})$ and a.e. as $h \to 0$}.
$$
Thus, thanks to \rif{stella}-\rif{convle}, we can let $h \to 0$
using a well-known variant of Lebesgue's dominated convergence
theorem; therefore we obtain
\begin{multline} \label{I9 lim}
\lim_{h \to 0}
   \int_\Omega
    \eta^2 \,
    \sum_{s=1}^{n}
    a_{n+s}(\Xu)(x \e^{hX_s})
    W(|\DhX u|^2) \,
    T \DhX u
    \dx \\
=
   \int_\Omega
    \eta^2 \,
    \sum_{s=1}^{n}
    a_{n+s}(\Xu)
    W(|X_s u|^2) \,
    X_s T u
    \dx\;.
\end{multline}
In a completely similar manner, we also have
\begin{multline} \label{I10 lim}
\lim_{h \to 0}
   \int_\Omega
    \eta^2 \,
   \sum_{s=1}^{n}
    a_s(\Xu)(x \e^{hY_s})
    W(|\DhY u|^2) \,
    T \DhY u
    \dx \\
=
   \int_\Omega
    \eta^2
    \sum_{s=1}^{n}
    a_s(\Xu)
    W(|Y_s u|^2) \,
    Y_s T u
    \dx\,.
\end{multline}
Connecting \rif{I9 lim} and \rif{I10 lim} to \rif{pre XX} we get
\begin{align} \label{pre XXpp}
\int_\Omega \eta^2 \sum_{s=1}^{2n}
    & \weight
    W(|X_s u|^2) |\X X_s u|^2
    \dx \nonumber\\
&\leq c \int_\Omega
    \Big( |\X \eta|^2 + \eta |T\eta| \Big) \,
    \deltaX^\frac{p}{2}\,
    \sum_{s=1}^{2n}
    W(|X_s u|^2)
    \dx \nonumber \\
&\qquad + c \int_\Om
    \eta^2 \,
    \weight
    \sum_{s=1}^{2n}
    W(|X_s u|^2)
    |Tu|^2 \dx \nonumber\\
&\qquad +\left|   \int_\Omega
    \eta^2 \,
    \sum_{s=1}^{n}
    a_{n+s}(\Xu)
    W(|X_s u|^2) \,
    X_s T u
    \dx\right| \nonumber \\
&\qquad +  \left|\int_\Omega
    \eta^2
    \sum_{s=1}^{n}
    a_s(\Xu)
    W(|Y_s u|^2) \,
    Y_s T u
    \dx\right|\;,
\end{align}
with $c \equiv c(n,p,\ratio)$. We continue estimating the last two
integrals; we shall estimate the first one, the estimation of the latter
being completely analogous.
%First let us observe that \eqn{intbp}
%$$
%\eta^2 a_{n+s}(\Xu)
%    W(|X_s u|^2),\eta^2 a_{s}(\Xu)
%    W(|Y_s u|^2) \in HW^{1,\frac{p}{p-1}}_{\loc}(\Omega)\;,
%$$
%for every $s \in\{1,\ldots,n\}$.
We integrate by
parts as follows:
\begin{align}
\int_\Omega
    \eta^2  \,
   \sum_{s=1}^{n}
    a_{n+s}(\Xu)
    W(|X_s u|^2) &\,
    X_s T u
    \dx
= - 2 \int_\Om \eta Tu  \sum_{s=1}^{n} X_s \eta
a_{n+s}(\Xu) W(|X_s u|^2) \dx  \nonumber \\
\nonumber & \quad - \int_\Om \eta^2 Tu
 \sum_{s=1}^{n} \sum_{\alpha=1}^{2n} D_{z_\alpha} a_{n+s}(\Xu) X_s X_\alpha u W(|X_su|^2) \dx \\
& \quad - 2 \int_\Om \eta^2 Tu
 \sum_{s=1}^{n} a_{n+s}(\Xu) W'(|X_su|^2) X_s u X_s X_s u \dx \nonumber \\
&=: A + B + C. \label{grdwrit}
\end{align}
The previous integration by parts needs of course to be justified; we postpone its verification to the very end of the proof.
The estimates for $A$, $B$, $C$ follow again by \eqref{upperb},
\eqref{lowerb} and Young's inequality; indeed, as for $A$ we have
\begin{align*}
|A| &\leq 2 \int_\Om \eta |\X \eta| \deltaX^\frac{p-1}{2} |Tu|
\sum_{s=1}^{n} W(|X_su|^2) \dx \\
&\leq c \int_\Om |\X \eta|^2 \deltaX^\frac{p}{2}
\sum_{s=1}^{n} W(|X_su|^2) \dx \\
&\hspace{3cm} + c \int_\Om \eta^2 \weight  \sum_{s=1}^{n}
W(|X_su|^2) |Tu|^2 \dx\,.
\end{align*}
Using that $X_sX_\alpha = X_\alpha X_s+ [X_s,X_\alpha]$, we have,
with $\varepsilon \in (0,1)$
\begin{align*}
|B| &\leq \left| \int_\Om \eta^2 Tu \sum_{s=1}^{n}
\sum_{\alpha=1}^{2n} D_{z_\alpha} a_{n+s}(\Xu) X_\alpha X_s  u
W(|X_su|^2) \dx
\right| \\
&\hspace{3cm} + \left| \int_\Om \eta^2 |Tu|^2 \sum_{s=1}^{n}
D_{z_{n+s}} a_{n+s}(\Xu) W(|X_su|^2) \dx
\right| \\
&\leq c\int_\Om \eta^2 \weight \sum_{s=1}^{n} W(|X_su|^2) \big(|Tu|
|\X X_s u| + |Tu|^2 \big)
\dx \\
&\leq \epsilon \int_\Om \eta^2 \weight
\sum_{s=1}^{n} W(|X_su|^2) |\X X_s u|^2 \dx \\
&\hspace{3cm} + c(\epsilon) \int_\Om \eta^2 \weight \sum_{s=1}^{n}
W(|X_su|^2) |Tu|^2 \dx\;.
\end{align*}
Finally, using \rif{W3} we have
\begin{align*}
|C| &\leq c \int_\Om \eta^2 \deltaX^\frac{p-1}{2} |Tu|
\sum_{s=1}^{n} W'(|X_su|^2) |X_s u| \, |X_s X_s u| \dx \\
&\leq \frac{\epsilon}{c(\sigma+1)} \int_\Om \eta^2 \weight
\sum_{s=1}^{n} W'(|X_su|^2) |X_s u|^2 \, |X_s X_s u|^2 \dx \\
&\hspace{3cm} + \frac{c(\sigma+1)}{\epsilon} \int_\Om \eta^2
\deltaX^\frac{p}{2} |Tu|^2
\sum_{s=1}^{n} W'(|X_su|^2) \dx \\
&\leq c\epsilon \int_\Om \eta^2 \weight
\sum_{s=1}^{n} W(|X_su|^2) |\X X_s u|^2 \dx \\
&\hspace{3cm} + \frac{c(\sigma+1)^2}{\epsilon} \int_\Om \eta^2
\deltaX^\frac{p}{2}  \sum_{s=1}^{n} \frac{W(|X_su|^2)}{\deltaXi}
|Tu|^2\dx\;.
\end{align*}
Joining together the estimates for $A,B,C$, we obtain
\begin{align} \label{I9}
\nonumber \Big| \int_\Omega
    \eta^2 \,
    \sum_{s=1}^{n}
    a_{n+s}(\Xu) &
    W(|X_s u|^2) \,
    X_s T u
    \dx \Big|\\ \nonumber &
\leq c \epsilon \int_\Om \eta^2 \weight
\sum_{s=1}^{n} W(|X_s u|^2) |\X X_s u|^2 \dx \nonumber\\
& + c \int_\Om |\X \eta|^2 \deltaX^\frac{p}{2}
\sum_{s=1}^{n} W(|X_su|^2) \dx \nonumber \\
& + c(\epsilon) \int_\Om \eta^2 \weight
\sum_{s=1}^{n} W(|X_su|^2) |Tu|^2\dx \nonumber \\
& + c(\epsilon)(\sigma+1)^2 \int_\Om \eta^2 \deltaX^\frac{p}{2}
\sum_{s=1}^{n} \frac{W(|X_su|^2)}{\deltaXi} |Tu|^2 \dx\,,
\end{align}
where $c \equiv c(n,p,\ratio)$. A completely analogous estimate,
replacing on the right hand side of \rif{I9} $X_s$ by $Y_s$, holds
also for the term
$$
 \int_\Omega
    \eta^2
    \sum_{s=1}^{n}
    a_s(\Xu)
    W(|Y_s u|^2) \,
    Y_s T u
    \dx\,,
$$
appearing in \rif{pre XXpp}. Therefore using \eqref{I9}, and its
$Y_s$-analog, to estimate \eqref{pre XXpp}, absorbing terms with
$\epsilon$ on the left hand side, we finally obtain
\begin{align*}
\int_\Omega \eta^2
    \weight
    &\sum_{s=1}^{2n}
    W(|X_s u|^2) |\X X_s u|^2
    \dx \\
&\leq c \int_\Omega
    \Big( |\X \eta|^2 + \eta |T\eta| \Big) \,
    \deltaX^\frac{p}{2}\,
    \sum_{s=1}^{2n}
    W(|X_s u|^2)
    \dx \\
& \qquad + c \int_\Om
    \eta^2 \,
    \weight
   \sum_{s=1}^{2n}
    W(|X_s u|^2)|Tu|^2
    \dx \\
& \qquad + c(\sigma+1)^2 \int_\Om
    \eta^2 \deltaX^\frac{p}{2}
    \sum_{s=1}^{2n} \frac{W(|X_s u|^2)}{\deltaXi} |Tu|^2\dx\,,
\end{align*}
where $c$ only depends on $n,p,\ratio$, but is otherwise independent
of $\mu, \sigma, k$, of the solution $u$, and of the vector field
$a(\cdot)$. Letting $k \nearrow \infty$ in the previous inequality,
using \rif{incW}-\rif{limW} to apply the monotone convergence
theorem, and finally using the elementary inequalities
$$
(\mu^2+|\Xu|^2)^{\frac{p}{2}}\sum_{s=1}^{2n}(\mu^2+|X_su|^2)^{\frac{\sigma}{2}}\leq
c(n,p) \sum_{s=1}^{2n}(\mu^2+|X_su|^2)^{\frac{p+\sigma}{2}}\;,
$$
$$
(\mu^2+|\Xu|^2)^{\frac{p-2}{2}}\sum_{s=1}^{2n}(\mu^2+|X_su|^2)^{\frac{\sigma}{2}}\leq
c(n,p) \sum_{s=1}^{2n}(\mu^2+|X_su|^2)^{\frac{p-2+\sigma}{2}}\;,
$$
and, since $\sigma \geq 2$ by assumption,
$$
(\mu^2+|\Xu|^2)^{\frac{p}{2}}\sum_{s=1}^{2n}(\mu^2+|X_su|^2)^{\frac{\sigma-2}{2}}\leq
c(n,p) \sum_{s=1}^{2n}(\mu^2+|X_su|^2)^{\frac{p-2+\sigma}{2}}\;,
$$
 we get \rif{XX sigma estpre}, from which also \rif{XX sigma est}
immediately follows. It remains to give the

{\em Justification of }\rif{grdwrit}. Fix $s \in \{1,\ldots,n\}$; assume that
\eqn{elle1}
$$
    X_s\left(\eta^2
    a_{n+s}(\Xu)
    W(|X_s u|^2)
     T u  \right)\in L^1_{\loc}(\Omega)
    $$
    and that the identity
\begin{eqnarray}
 \nonumber X_s\left(\eta^2
    a_{n+s}(\Xu)
    W(|X_s u|^2)
     T u  \right)&=&  (X_s\eta^2) a_{n+s}(\Xu)
    W(|X_s u|^2)
     T u\\  \nonumber
     && +\eta^2 \sum_{j=1}^{2n}D_{z_j}a_{n+s}(\Xu)X_sX_ju W(|X_s u|^2)
     T u\\ \nonumber
     && +2\eta^2 a_{n+s}(\Xu) W'(|X_s u|^2)X_su X_sX_su
     T u\\ \nonumber
     && +\eta^2 a_{n+s}(\Xu) W(|X_s u|^2)X_sTu\\
     &=:&B_1+B_2+B_3+B_4\;,\label{elle2}
\end{eqnarray}
holds in the distributional sense, with $B_1,\ldots, B_4 \in L^1_{\loc}(\Omega)$. Then,
since $\eta$ has compact support in $\Omega$, we have that
$$
\int_\Om X_s\left(\eta^2
    a_{n+s}(\Xu)
    W(|X_s u|^2)
     T u  \right)\, dx =0\;,
$$
from which \rif{grdwrit} follows via \rif{elle2}. In turn it remains to establish the validity of
\rif{elle1}-\rif{elle2}. We shall repeatedly use Lemma \ref{eles}; we start observing
that by \rif{growth} and $Tu \in L^p_{\loc}(\Omega)$, Young's inequality gives that $a_{n+s}(\Xu)
    W(|X_s u|^2)
     T u\in L^1_{\loc}(\Omega)$. We are of course using that $W(\cdot)$ is bounded. The same argument gives that
    $B_1 \in L^1_{\loc}(\Omega)$. Next we have
    $$
    |B_2|\leq c(k,\sigma)\left[(\mu^2+|\mathfrak{X}u|^2)^{\frac{p-2}{2}} |\XXX \mathfrak{X}
u|^2 + \mu^p+|\Xu|^p+|Tu|^p\right]\;,
    $$
   and observe that the right hand side belongs to $L^1_{\loc}(\Omega)$ by \rif{dese}, therefore $B_2 \in L^1_{\loc}(\Omega)$.
Then, by \rif{growth}, \rif{W3} and Young's inequality we have
\begin{eqnarray*}
|B_3| &\leq & \frac{c(k,\sigma)|X_su|}{\mu^2+|X_su|^2} (\mu^2+|\mathfrak{X}u|^2)^{\frac{p-1}{2}}|\X \X u||Tu|\\
&\leq & c(k,\sigma, \mu)\left[(\mu^2+|\mathfrak{X}u|^2)^{\frac{p-2}{2}}|\X \X u|^2+(\mu^2+|\mathfrak{X}u|^2)^{\frac{p}{2}}|Tu|^2\right]
\end{eqnarray*}
 and observe that all the quantities in the right
 hand side belong to $L^1_{\loc}(\Omega)$ by \rif{dese}
 and \rif{startint}, since here we are assuming $\sigma \geq 2$. We again conclude that $B_3 \in L^1_{\loc}(\Omega)$.
 Finally, again by \rif{growth} we have that
  $$
    |B_4|\leq c(k,\sigma)(\mu^2+|\mathfrak{X}u|^2)^{\frac{p-2}{2}} \left[|\XXX T
u|^2 + (\mu^2+|\mathfrak{X} u|^2)\right]\;,
    $$
and again, $B_4 \in L^1_{\loc}(\Omega)$ follows from \rif{dese}. At this stage we can apply Lemma \ref{eles} to the product
$a_{n+s}(\Xu)
    W(|X_s u|^2)
     T u\in L^1_{\loc}(\Omega)$ concluding that \rif{elle1}-\rif{elle2} hold.
\end{proof}
\subsection{The vertical Caccioppoli inequality}
We now state the energy estimate involving $Tu$. Its proof is
considerably simpler and it is close to similar estimates in the
Euclidean case, since the operators $T$ and $\XXX $ commute. We
report the full proof for the sake of completeness.
\begin{lemma} \label{XT sigma lem}
 Let $u \in HW^{1,p}(\Omega)$ be a weak solution to the equation \trif{due}
 under the assumptions \trif{growth}-\trif{nondeg}, with $
2\leq  p < 4.$ Let $\sigma \geq 0$ and assume that \eqn{intass1}
$$\Xu
\in L^{p+2+\sigma}_{\loc}(\Omega,\er^{2n})\,, \qquad \mbox{and} \qquad
Tu \in L^\frac{p+2+\sigma}{2}_{\loc}(\Omega)\;.$$ Then we have
\begin{equation} \label{XT sigma}
\int_{\Omega} \weight |Tu|^{\frac{\sigma}{2}} |\X Tu|^2 \eta^2 \,dx
\leq c \int_{\Omega}\weight |Tu|^{\frac{\sigma+4}{2}} |\X \eta|^2
\,dx,
\end{equation}
for all $\eta\in C^\infty_c(\Omega)$, where the constant $c\equiv
c(n,p,\ratio)$, is independent of $\mu$, of the solution $u$, and of
the vector field $a(\cdot)$.
\end{lemma}
\begin{proof} We again start by applying Lemma
\ref{lemma:diff quot eq}, this time with the choice $Z=T$; for every $\varphi \in C^{\infty}_c(\Omega)$, and
$h\not=0$ accordingly small, we arrive at
\begin{equation}\label{XT:1}
    \int_\Omega \langle \DhT a (\Xu), \X \varphi\rangle \,  dx=0\, .
\end{equation}
Observe that we have used that $[T,X_i]=0$ for every $i=1,\ldots,2n$.
As a test function in \rif{XT:1} we choose $\varphi=\eta^2 |\DhT
u|^\frac{\sigma}{2} \DhT u$. Note that this is an admissible test
function in view of the fact that $u$ is locally bounded, see
Theorem \ref{CDG}. Since $[T,X_s]=0$ for any $s=1,\ldots,2n$, we
have $
    \X(\DhT u)=\DhT(\Xu)
$ by Lemma \ref{tritra}. Inserting $\varphi$ into \eqref{XT:1} we
find
\begin{eqnarray}
   \nonumber &&(1+\sigma/2)\int_\Omega \eta^2 \sum_{i=1}^{2n}\DhT a_i(\Xu)  X_i \DhT u |\DhT u|^\frac{\sigma}{2} \dx
    \\&& \qquad \qquad = -2 \int_\Omega \eta \sum_{i=1}^{2n} \DhT a_i(\Xu)X_i \eta |\DhT u|^{\frac{\sigma}{2}} \DhT u \dx
    .\label{XT:2}
\end{eqnarray}
Using \eqref{Btilde} and \eqref{lowerb} with $Z\equiv X$, we can
estimate the l.h.s. of \eqref{XT:2} from below
\begin{equation*}
    \text{l.h.s. of \eqref{XT:2}} \geq
    c^{-1} \int_\Omega \weightT |\DhT u|^\frac{\sigma}{2} |\XXX \DhT u|^2 \eta^2
    \dx\,,
\end{equation*}
where $c \equiv c(n,p,\ratio)\geq 1$. For the r.h.s of \eqref{XT:2}
we use again \eqref{Btilde} together with \eqref{upperb} and Young's
inequality obtaining, with $\varepsilon \in (0,1)$
\begin{multline*}
    |\text{r.h.s of \eqref{XT:2}}| \leq
    \epsilon \int_\Omega \weightT |\DhT u|^\frac{\sigma}{2} |\XXX  \DhT u|^2 \eta^2 \dx
     \\
     +c(\epsilon) \int_\Omega \weightT |\DhT u|^{\frac{\sigma}{2}+2} |\X \eta|^2 \dx .
\end{multline*}
Combining these estimates and choosing $\varepsilon$ suitably small
as usual, we arrive at the following Caccioppoli-type estimate:
\begin{multline}\label{XT:cacc}
    I_h:=\int_\Omega \weightT |\DhT u|^\frac{\sigma}{2} |\XXX \DhT u|^2 \eta^2 \dx
    \\
    \leq
    \tilde{c} \int_\Omega \weightT
    |\DhT u|^{\frac{\sigma}{2}+2} |\X \eta|^2 \dx=:II_h
\end{multline}
which is obviously valid for any $h>0$ such that
$\sqrt{h}=|\e^{hT}|_{cc}<\mathrm{dist}(\supp \eta,\partial\Omega)$;
here $  \tilde{c}$ depends on $n,p,\ratio$. Using Young's inequality
to estimate the r.h.s of \eqref{XT:cacc} we finally obtain
 \eqn{pre XT}
 $$II_h
\leq
    c \int_{\supp \eta}
    \deltaXt^\frac{p+2+\sigma}{2}  \dx
    + c \int_{\supp \eta}
    |\DhT u|^\frac{p+2+\sigma}{2}  \dx\;,
$$
with $c \equiv c(\|\XXX \eta\|_{L^{\infty}})$. Since both $Tu$ and
$\Xu$ exist and satisfy \rif{intass1}, by Lemma \ref{hormander},
\rif{pre XT}, and a well-known variant of Lebesgue's dominated
convergence theorem, we obtain that \eqn{conv12}
$$
\lim_{h \to 0} II_h =   \tilde{c}\int_{\Omega}\weight
|Tu|^{\frac{\sigma+4}{2}} |\X \eta|^2 \,dx\;.
$$
On the other hand, by Lemma \ref{diffquo} and using and Fatou's
lemma we have that \eqn{conv2}
$$
\int_{\Omega} \weight |Tu|^{\frac{\sigma}{2}} |\X Tu|^2 \eta^2
\,dx\leq  \liminf_{h \to 0} I_h\,.
$$
The proof of \rif{XT sigma} now follows combining
\rif{conv12}-\rif{conv2} with \rif{XT:cacc}.
\end{proof}
\section{Intermediate integrability}
The aim of this section is to improve the already found higher
integrability result in \rif{basichi1}. Indeed the main result here
is
\begin{lemma} \label{HiX00}
 Let $u \in HW^{1,p}(\Omega)$ be a weak solution to the equation \trif{due}
 under the assumptions \trif{growth}-\trif{nondeg}, with $
2\leq  p < 4.$ Then \eqn{hhcc00}
$$
\Xu \in L_{\loc}^{p+4}(\Omega, \er^{2n}).
$$
Moreover, for every couple of open subsets $\Omega' \Subset \Omega''
\Subset \Omega$ there exists a constant $c$ depending only on
$n,p,\ratio$, $\dist(\Omega',
\partial \Omega'')$, and $\|u\|_{L^\infty(\Omega'')}$, but independent
of $\mu$, of the solution $u$, and of the vector field $a(\cdot)$,
such that \eqn{estint1base}
$$
\int_{\Omega'} |\Xu|^{p+4}\dx \leq c \int_{\Omega''}
\big(|\Xu|^{p}+|Tu|^{p}+\mu^p\big)\dx\,,
$$
where $c\equiv c(n,p,\ratio, \sigma)>0$. \end{lemma} The key to the
previous lemma is in fact the following one, whose proof features a
rather unorthodox choice of the test function $\varphi$ in \rif{wf0}
- see \rif{ro} below.
\begin{lemma}\label{lem:righthand}
Let $u \in HW^{1,p}(\Omega)$ be a weak solution to the equation
\trif{due}
 under the assumptions \trif{growth}-\trif{nondeg}, with $
2\leq  p < 4.$ Then
$$
\deltaX^{\frac{p}{2}} \vert Tu\vert^2\in L^1_{\loc}(\Omega) .
$$
Moreover, for all $\eta\in C^\infty_c(\Omega)$, we have
\begin{eqnarray}
\nonumber &&\int_\Omega \eta^2 \deltaX^{\frac{p}{2}} \vert Tu\vert^2 \dx \\
&&\qquad \qquad \le c (1+\vert\vert u\vert\vert_{L^\infty(\supp
\eta)}^2)\int_{\Omega}(\eta^2+\vert\X\eta\vert^2)
\deltaX^{\frac{p-2}{2}} \vert Tu\vert^2\dx,
\end{eqnarray}\label{HiX init3}
where $c\equiv c(n,p)>0$.
\end{lemma}
\begin{proof}
In the following we shall denote $\T_k(t):=\min\{t,k\}$, for $t \geq
0$ and $k \in \mathbb N$, slightly adjusting the definition already
given in \rif{TT}. Set \eqn{ro}
$$\varphi:=(\T_k(\vert Tu\vert))^2\eta^2u\,,$$
for $k>0$; we wish to take $\varphi$ as a test function in
\rif{wf0}. We first observe that the function $t \mapsto (\T_k(\vert
t\vert))^2$ is Lipschitz continuous and therefore, since $Tu \in HW^{1,2}(\Omega)$ then by the chain rule in the Heisenberg group -
see \cite{CDG} - it also follows that $(\T_k(\vert Tu\vert))^2\in
HW^{1,2}(\Omega)$. Then, since $u\in HW^{1,p}(\Omega)\cap
L^\infty_{\loc}(\Omega)$ a standard difference quotients argument, as for instance the one in Lemma \ref{eles},
finally gives that $\varphi\in HW^{1,2}_0(\Omega)$. Now recall
that in Lemma \ref{G:N}, we already showed
 that  $\Xu\in L_{\loc}^{p+2}(\Omega,\er^{2n})$. So by a standard approximation argument, we can easily show that
any function from $HW^{1,(p+2)/3}_0(\Omega)$ is a admissible in
\rif{wf0}. Thus  $\varphi$ as defined in \rif{ro} is an
admissible test function, since $(p+2)/3<2$. Recall here that we are
assuming $p < 4$. Therefore, using $\varphi$ in \rif{wf0}, we obtain
$$
\begin{aligned}
\int_\Omega\eta^2(\T_k(\vert Tu\vert))^2\langle a(\Xu),\XXX
u\rangle\dx &
 =-2\int_\Omega \eta u(\T_k(\vert Tu\vert))^2\langle a(\Xu),\X \eta\rangle\dx \\
&\qquad  -\int_\Omega \eta^2 u\langle a(\Xu),\X (\T_k(\vert
Tu\vert))^2\rangle\dx\;.
\end{aligned}
$$
In turn, using \rif{growth} and \rif{veramon} the previous equality
yields
\begin{equation}\label{HiX inti4}
\begin{aligned}
\int_\Omega\eta^2\deltaX^{\frac{p-2}{2}}&\vert\Xu\vert^2(\T_k(\vert Tu\vert))^2\dx\\
\leq &c\int_\Omega \eta\vert\X \eta\vert\vert u\vert\deltaX^{\frac{p-1}{2}}(\T_k(\vert Tu\vert))^2\dx\\
&+
 c\int_\Omega\eta^2 |u| \deltaX^{\frac{p-1}{2}}|\X (\T_k(\vert Tu\vert))^2|\dx\\
&+
 c\int_\Omega\eta^2 \mu^p (\T_k(\vert Tu\vert))^2\dx=:D+E+F\,,
\end{aligned}
\end{equation}with $c \equiv c(n,p,\ratio)$.
We use Young's inequality to estimate $D$ as follows:
\begin{equation*}
\begin{aligned}
 D &\le \frac{1}{4}\int_\Omega\eta^2\deltaX^{\frac{p-2}{2}}\vert\Xu\vert^2(\T_k(\vert Tu\vert))^2\dx\\
&\qquad +c \vert\vert u\vert\vert_{L^\infty(\supp
\eta)}^2\int_\Omega\vert\X\eta\vert^2\deltaX^{\frac{p-2}{2}}\vert
Tu\vert^2\dx\\ & \qquad + \int_\Omega\eta^2\mu^2
\deltaX^{\frac{p-2}{2}}(\T_k(\vert Tu\vert))^2\dx .
\end{aligned}
\end{equation*}
We estimate $E$ by Young's inequality and Lemma \ref{XT sigma lem}
with $\sigma=0$, that is
\begin{equation*}
\begin{aligned}
 E&\le \frac{1}{4} \int_\Omega\eta^2\deltaX^{\frac{p-2}{2}}\vert\Xu\vert^2(\T_k(\vert Tu\vert))^2\dx\\
& \qquad +c \vert\vert u\vert\vert_{L^\infty(\supp
\eta)}^2\int_\Omega \eta^2\deltaX^{\frac{p-2}{2}}\vert\X
Tu\vert^2\dx\\ & \qquad + \int_\Omega\eta^2\mu^2
\deltaX^{\frac{p-2}{2}}(\T_k(\vert Tu\vert))^2\dx \\
&\stackrel{\rif{XT sigma}}{\le}  \frac{1}{4} \int_\Omega\eta^2\deltaX^{\frac{p-2}{2}}\vert\Xu\vert^2(\T_k(\vert Tu\vert))^2\dx\\
&\qquad+c \vert\vert u\vert\vert_{L^\infty(\supp
\eta)}^2\int_\Omega\vert\X\eta\vert^2\deltaX^{\frac{p-2}{2}}\vert
Tu\vert^2\dx\\ & \qquad + \int_\Omega\eta^2\mu^2
\deltaX^{\frac{p-2}{2}}(\T_k(\vert Tu\vert))^2\dx .
\end{aligned}
\end{equation*}
Finally, since $\mu \leq 1$ we have
$$F \leq c\int_\Omega
 \eta^2\deltaX^{\frac{p-2}{2}} \vert Tu\vert^2\dx\;.$$
Plugging the above estimates for $D,E$ and $F$ into \eqref{HiX
inti4}, and eventually letting $k \nearrow \infty$, we obtain
\eqref{HiX init3}, using that $\mu \leq 1$. This completes the proof
of the lemma.
\end{proof}
\begin{proof}
The proof of \rif{hhcc00} follows combining Lemma
\ref{lem:righthand}, Lemma \ref{XX sigma lem} with $\sigma =2$,
Lemma \ref{G:N}, and finally Lemma \ref{G:Nin} again with
$\sigma=2$. Accordingly, the proof of \rif{estint1base} follows
combining all the a priori estimates of the used lemmata, taking
into account the fact that everywhere $\Omega', \Omega''$ and $\eta$
can be chosen arbitrarily. Moreover, the right hand side of \rif{HiX
init3} has to be estimated by means of Young's inequality, as
follows:
$$
\int_{\supp \eta} \deltaX^{\frac{p-2}{2}} \vert Tu\vert^2\dx\leq
 c \int_{\supp \eta} \big(|\Xu|^{p}+|Tu|^{p}+\mu^p\big)\dx\;.$$
\end{proof}
\section{Iteration and higher integrability}
The main result of this section is the following:
\begin{prop} \label{horiz}
 Let $u \in HW^{1,p}(\Omega)$ be a weak solution to the equation \trif{due}
 under the assumptions \trif{growth}-\trif{nondeg}, with $
2\leq  p < 4.$ Then it holds that \eqn{betterint22}
 $$ \mathfrak{X}u \in
L^{q}_{\textnormal{loc}}(\Omega,\er^{2n})\qquad \mbox{and }\qquad Tu
\in L^{q}_{\textnormal{loc}}(\Omega) \qquad \mbox{for every} \ \
q<\infty\;.$$ Moreover, for every $q<\infty$ there exists a constant
$c$, depending on $n,p,\ratio$, and $q$, but otherwise independent
of $\mu$, of the solution $u$, and of the vector field $a(\cdot)$,
such that the following reverse-H\"older type inequalities hold for
any CC-ball $B_R \subset \Omega$: \eqn{apap}
$$
\left(\intav_{B_{R/2}}|\XXX u|^{q}\, dx\right)^{1/q}\leq c
\left(\intav_{B_{R}}(\mu+|\XXX u|)^{p}\, dx\right)^{1/p}\;,
$$
and \eqn{apat}
$$
\left(\intav_{B_{R/2}}|T u|^{q}\, dx\right)^{1/q}\leq \frac{c}{R}
\left(\intav_{B_{R}}(\mu+|\XXX u|)^{p}\, dx\right)^{1/p}\;.
$$
\end{prop}
In order to prove the previous result we need a few preliminary
lemmata. Their iterated use will finally lead to the proof of
Proposition \ref{horiz}.

\begin{lemma} \label{HiX}
 Let $u \in HW^{1,p}(\Omega)$ be a weak solution to the equation \trif{due}
 under the assumptions \trif{growth}-\trif{nondeg}, with $
2\leq  p < 4.$ Assume that \eqn{asshix}
$$\Xu \in L_{\loc}^{p+\sigma}(\Omega,\er^{2n}), \qquad \vert\Xu\vert^{p-2+\sigma}\vert Tu\vert^2\in
L^1_{\loc}(\Omega),\quad \mbox{and} \qquad Tu \in
L_{\loc}^\frac{p+2+\sigma}{2}(\Omega),$$ for some $\sigma \geq 2$.
Then \eqn{hhcc}
$$
\Xu \in L_{\loc}^{p+2+\sigma}(\Omega, \er^{2n}).
$$
Moreover, for every couple of open subsets $\Omega' \Subset \Omega''
\Subset \Omega$ there exists a constant $c$ depending only on
$n,p,\ratio,\sigma$, $\dist(\Omega',
\partial \Omega'')$, and $\|u\|_{L^\infty(\Omega'')}$, but independent on
$\mu$, such that \eqn{estint1}
$$
\int_{\Omega'} |\Xu|^{p+2+\sigma}\dx \leq c \int_{\Omega''}
\big(|\Xu|^{p+\sigma}+|Tu|^{\frac{p+2+\sigma}{2}}+\mu^p\big)\dx\,.
$$
\end{lemma}
\begin{proof} By \rif{asshix} we can use Lemma \ref{XX sigma lem}; therefore combining
\rif{XX sigma est} with \rif{HGN}, by means of a standard covering
argument we deduce the validity of \rif{hhcc}. Once \rif{hhcc} holds
we use Young's inequality to estimate the last integral in the right
hand side of \rif{XX sigma est} as follows:
\begin{eqnarray}
\nonumber \int_\Om  \eta^2 \deltaX^{\frac{p-2+\sigma}{2}} |Tu|^2 \dx
& \leq &\epsilon \int_\Om  \eta^2
\deltaX^{\frac{p+2+\sigma}{2}}\dx  \\
&& \qquad \qquad +c(\epsilon) \int_\Om  \eta^2
|Tu|^{\frac{p+2+\sigma}{2}} \dx\;, \label{ytrivial}
\end{eqnarray}
where $\varepsilon \in (0,1)$; note that the intermediate integral
in \rif{ytrivial} is now finite. Connecting the previous inequality
to \rif{XX sigma est} and eventually to \rif{HGN}, and choosing $\epsilon$
small enough, but depending only on $n,p,\ratio,\sigma$ and
$\|u\|_{L^\infty(\supp \eta)}$, in order to re-absorb the
intermediate integral appearing in \rif{ytrivial} in the left-hand
side of \rif{HGN}, we gain, after a few elementary manipulations
\begin{eqnarray*}
&&\int_\Om \eta^2 \deltaX^\frac{p+2+\sigma}{2} \dx  \leq
c\int_{\supp \eta}
\deltaX^\frac{p+\sigma}{2} \dx \nonumber\\
&& \hspace{6cm} +c \int_{\supp \eta} |Tu|^{\frac{p+2+\sigma}{2}}
\dx\,.
\end{eqnarray*}
The constant $c$ in the last inequality depends only on the data
$n,p,\ratio,\sigma,$ and on the norms $\|\XXX \eta\|_{L^\infty},\|T
\eta\|_{L^\infty}, \|u\|_{L^\infty(\supp \eta)}$, but is otherwise
independent of the solution $u$, of the vector field $a(\cdot)$, and
of $\mu$. Note that we have used that $\mu \leq 1$. At this stage
the inequality in \rif{estint1} follows by the previous inequality
via a standard covering argument involving a suitable choice of the
cut-off function $\eta$; again we are using that $\mu \leq 1$.
\end{proof}
\begin{lemma} \label{HiT}
 Let $u \in HW^{1,p}(\Omega)$ be a weak solution to the equation \trif{due}
 under the assumptions \trif{growth}-\trif{nondeg}, with $
2\leq  p < 4.$ Assume that
$$\Xu \in L_{\loc}^{p+2+\sigma}(\Om, \R^{2n}) \qquad \mbox{and }\qquad Tu \in
L_{\loc}^\frac{p+2+\sigma}{2}(\Om)\;,$$ for some $\sigma \geq 0$,
then \eqn{hiTu}
$$
Tu \in L_{\loc}^\frac{p+3+\sigma}{2}(\Om).
$$
Moreover, for every couple of open subsets $\Omega' \Subset \Omega''
\Subset \Omega$ there exists a constant $c$ depending only on
$n,p,\ratio,\sigma$, $\dist(\Omega',
\partial \Omega'')$, and $\|u\|_{L^\infty(\Omega'')}$, but independent on
$\mu$, of the solution $u$, and of the vector field $a(\cdot)$, such
that \eqn{estint1t}
$$
\int_{\Omega'} |Tu|^{\frac{p+3+\sigma}{2}}\dx \leq c \int_{\Omega''}
\big(|\Xu|^{p+2+\sigma}+|Tu|^{\frac{p+2+\sigma}{2}}+\mu^p\big)\dx
$$
holds.
\end{lemma}
\begin{proof}
In the following we shall again denote $\T_k(t):=\min\{t,k\}$ for $t
\geq 0$ and $k \in \mathbb N$. Let $\eta\in C^\infty_c(\Omega)$ be
as usual a cut-off function such that $0\le \eta\le 1$. Using that
$T =  [X_i,Y_i] = X_iY_i-Y_iX_i,$ we start by integrating by parts
as follows:
\begin{align} \label{pre HiT}
\int_{\Omega} & \eta^2 |Tu|^2 \, \T_k(|Tu|^\frac{p-1+\sigma}{2}) \dx
=
\int_{\Omega} \eta^2 (X_1 Y_1 - Y_1 X_1)u \, Tu \,\T_k(|Tu|^\frac{p-1+\sigma}{2}) \dx \nonumber \\
&\leq 4 \int_{\Omega} \eta |\X \eta| |\Xu| |Tu|
\,\T_k(|Tu|^\frac{p-1+\sigma}{2}) \dx +
c \int_{\Omega} \eta^2 |\Xu| |\X Tu| \,\T_k(|Tu|^\frac{p-1+\sigma}{2}) \dx \nonumber\\
&=:P_4+P_5,
\end{align}
where $c=c(p,\sigma)>0$. Note that the previous integration by parts
is legal since \eqn{byparts}
$$Tu
\,\T_k(|Tu|^\frac{p-1+\sigma}{2}) \in HW^{1,2}_{\loc}(\Omega)\;.$$
This fact follows by chain rule in the Heisenberg group - see
\cite{CDG} - since by the very definition of $\T_k$ it follows that
the function $ t \mapsto
 t \T_k(|t|^\frac{p-1+\sigma}{2})$
 is globally Lipschitz continuous on $\er$, together with the fact that $Tu \in
 HW^{1,2}_{\loc}(\Omega)$ - see \rif{dese}.

Now, by Young's inequality, we have for the integral $P_4$
\begin{equation*}
P_4 \leq 4\int_{\Omega} |\XXX \eta||\Xu|^{p+2+\sigma} \dx+
4\int_{\Omega} |\XXX \eta||Tu|^\frac{p+2+\sigma}{2}\dx\,.
\end{equation*}
We now come to $P_5$; using repeatedly Young's inequality and
once inequality \rif{XT sigma} from Lemma \ref{XT sigma lem} we
have
\begin{eqnarray*}
P_5 &\leq & c \int_{\Omega}\eta^2
\deltaX^\frac{1}{2} |T u|^\frac{p-1+\sigma}{2} |\X Tu| \dx \\
&\leq  &c \int_{\Omega}\eta^2 \weight |T u|^\frac{\sigma}{2} |\X
Tu|^2 \dx\\ &&\qquad +c\int_{\Omega} \eta^2
\deltaX^\frac{4-p}{2} |T u|^\frac{2p-2+\sigma}{2} \dx \\
&\stackrel{\rif{XT sigma}}{\leq} & c\int_{\Omega}\vert \X\eta\vert^2
\weight |T u|^{\frac{\sigma+4}{2}} \dx\\ && \qquad + c\int_\Omega
\eta^2
\deltaX^\frac{4-p}{2} |T u|^\frac{2p-2+\sigma}{2} \dx\\
&\leq & c\int_{\Omega} (\eta^2+|\XXX
\eta|^2)\big(\mu^{p+2+\sigma}+|\Xu|^{p+2+\sigma} +
|Tu|^\frac{p+2+\sigma}{2}\big)\dx.
\end{eqnarray*}
Note how the crucial assumption $p<4$ hereby comes into the play
once again. Using the estimates found for $P_4,P_5$, inequality
\eqref{pre HiT} becomes \begin{eqnarray*} && \int_{\Omega}\eta^2
|Tu|^2 \, \T_k(|Tu|^\frac{p-1+\sigma}{2}) \dx \\
&& \qquad \qquad \leq c\int_{\Omega} (\eta^2+|\XXX \eta|+|\XXX
\eta|^2)\big(\mu^{p+2+\sigma}+|\Xu|^{p+2+\sigma} +
|Tu|^\frac{p+2+\sigma}{2}\big)\dx\,.
\end{eqnarray*} The constant $c$ in the last inequality depends only on $n,p,\sigma$.
Letting $k \nearrow\infty$ and using the fact that $\mu \leq 1$, we
have
\begin{equation*}
\int_{\Omega} \eta^2|Tu|^{\frac{p+3+\sigma}{2}}\dx \leq
c\int_{\Omega} (\eta^2+|\XXX \eta|+|\XXX
\eta|^2)\big(\mu^{p}+|\Xu|^{p+2+\sigma} +
|Tu|^\frac{p+2+\sigma}{2}\big)\dx\;.
\end{equation*}
Then \rif{hiTu} follows by a standard covering argument since the
choice of $\eta$ is arbitrary in the previous inequality. In the
same way, \rif{estint1t} follows via a standard covering argument
involving a suitable choice of $\eta$.
\end{proof}
\begin{proof}[Proof of Proposition \ref{horiz}]
The proof is divided in two steps: first we prove the qualitative
result in \rif{betterint22} with a first form of the main priori
estimates, that is \rif{apriori} below. Then, in a second step, we
show how to get the explicit form of the a priori estimates in
\rif{apap}-\rif{apat} from \rif{apriori} by means of a ``blow-up"
argument.

{\em Step 1: Iteration and higher integrability.} Here we prove
\rif{betterint22} and that, for every couple of open subsets $\Omega'
\Subset \Omega'' \Subset \Omega$, and $q<\infty$, there exists a
constant $c$ depending only on $n,p,\ratio,q$, $\dist(\Omega',
\partial \Omega'')$, and $\|u\|_{L^\infty(\Omega'')}$, but independent
of $\mu$, of the solution $u$, and of the vector field $a(\cdot)$,
such that \eqn{apriori}
$$
\int_{\Omega'} (|\XXX u|^q+|T u|^q)\, dx \leq c \int_{\Omega''}
\left(|\XXX u|^p +1\right)\, dx\;.
$$
For this, let us define the sequence \eqn{newsigma}
$$
\left\{
\begin{array}{ccc}
 \sigma_{k+1} & := &
\displaystyle \sigma_k + \frac{4}{p+3+\sigma_k}\\
\\ \sigma_0 & := &2.
\end{array}
\right.
$$
It is easy to see that $\{\sigma_k\}$ is a strictly increasing
sequence such that $\sigma_k \nearrow \infty$. We shall prove by
induction that
$$\Xu \in L_{\loc}^{p+2+\sigma_k}(\Om, \R^{2n}) \qquad \mbox{and }\qquad Tu \in
L_{\loc}^\frac{p+2+\sigma_k}{2}(\Om) \qquad {\bf (A)_k}\;, $$
holds every $k \in \mathbb N$, and moreover that, for every couple
of open subset $\Omega' \Subset \Omega'' \Subset \Omega$ and $k \in
\mathbb N$ there exists a constant $c$ depending only on
$n,p,\ratio,k$, $\dist(\Omega',
\partial \Omega'')$, and $\|u\|_{L^\infty(\Omega'')}$, but independent
of $\mu$, of the solution $u$, and of the vector field $a (\cdot)$,
such that
$$
\int_{\Omega'} (|\XXX u|^{p+2+\sigma_k}+|T
u|^{\frac{p+2+\sigma_k}{2}})\, dx \leq c \int_{\Omega''} \left(|\XXX
u|^p +|Tu|^p+1\right)\, dx \quad {\bf (B)_k}\;.
$$
We shall eventually show that this will suffice to prove
\rif{betterint22} and \rif{apriori}. Before going on let us point
out that when proving estimates like ${\bf (B)_k}$ we shall deal
with similar estimates where $\Omega',\Omega''$ vary in an arbitrary
way. Each time we shall implicitly pass to different open subsets,
since every time the open subsets involved in the inequalities will
be arbitrary.

Let us first prove the validity of ${\bf (A)_0}$ and ${\bf (B)_0}$.
The parts of the statements concerning $\Xu$ directly come from
Lemma \ref{HiX00}, therefore we concentrate on $Tu$. To this aim we
apply Lemma \ref{HiT} twice. First we choose $\sigma=0$, recalling
that $(p+2)/2\leq p$ in turn implies $Tu \in
L^{(p+2)/2}_{\loc}(\Omega)$; at this point we get that $Tu \in
L^{(p+3)/2}_{\loc}(\Omega)$ with a first corresponding estimate,
that is
$$
\int_{\Omega'} |Tu|^{\frac{p+3}{2}}\dx \leq c \int_{\Om''}
\left(|\Xu|^{p+2} +|Tu|^{\frac{p+2}{2}}+1\right)\dx\;.
$$
Then we are able to apply again Lemma \ref{HiT}, this time with
$\sigma =1$, getting that $Tu \in L^{(p+4)/2}_{\loc}(\Omega)$ and,
in view of \rif{estint1t}, also that
$$
\int_{\Omega'} |Tu|^{\frac{p+4}{2}}\dx \leq c \int_{\Om''}
\left(|\Xu|^{p+3} +|Tu|^{\frac{p+3}{2}}+1\right)\dx\;.
$$
Joining the last two estimates to \rif{estint1base}, passing
each time to different open subsets, which are not renamed, we easily get the also the part of ${\bf (B)_0}$
concerned with $Tu$.

Let us now assume the validity of ${\bf (A)_k}$ and ${\bf (B)_k}$
for some $k \geq 0$, and let us prove that of ${\bf (A)_{k+1}}$ and
${\bf (B)_{k+1}}$. By ${\bf (A)_k}$ we may apply Lemma \ref{HiT}
with the choice $\sigma \equiv \sigma_k$ in order to get that
\eqn{hitk}
$$Tu
\in L_{\loc}^\frac{p+3+\sigma_k}{2}(\Om)\;.$$ Observe that by the
very definition of $\sigma_k$ we have that \eqn{trisig}
$$
 \sigma_{k+1}<\sigma_k+ 1\;,$$ and therefore from \rif{hitk} we immediately get that
\eqn{hitk2}
 $$Tu \in
L_{\loc}^\frac{p+2+\sigma_{k+1}}{2}(\Om)\;.$$ We also observe that
using ${\bf (B)_{k}}$ and the estimate \rif{estint1t} for $\sigma
\equiv \sigma_k$, since in every occurrence the open subsets
$\Omega' \CC \Omega''$ are arbitrary, we easily gain \eqn{esttk1}
$$
\int_{\Omega'} |Tu|^{\frac{p+2+\sigma_{k+1}}{2}}\dx\leq
\int_{\Omega'} (|Tu|^{\frac{p+3+\sigma_k}{2}}+1)\dx \leq c
\int_{\Omega''} \big(|\Xu|^{p}+|Tu|^{p}+1\big)\dx\,,
$$ that in turn holds
for every couple of $\Omega' \CC \Omega''$ where $c$ depends as in
${\bf (B)_{k+1}}$. Here we used again \rif{trisig} and an elementary
estimation. We have indeed proved one part of ${\bf (B)_{k+1}}$ too.
Therefore it only remains to prove that $\Xu \in
L_{\loc}^{p+2+\sigma_{k+1}}(\Om, \R^{2n})$, that will complete the
proof of ${\bf (A)_{k+1}}$, and the corresponding remaining part of
${\bf (B)_{k+1}}$ with the estimation of $\Xu$. For this we wish to use Lemma \ref{HiX} with the choice $\sigma
\equiv \sigma_{k+1} $, therefore let us check its applicability;
estimate \rif{trisig}, assumption ${\bf (A)_k}$ and \rif{hitk2}
imply that we actually just have to check the second inclusion in
\rif{asshix}. To do this we apply Young's inequality as follows:
$$
|\Xu|^{p-2+\sigma_{k+1}}|Tu|^2 \leq
|\Xu|^{\frac{(p-2+\sigma_{k+1})(p+3+\sigma_k)}{p-1+\sigma_k}}+
|Tu|^{\frac{p+3+\sigma_k}{2}}\,.$$ By the definition in
\rif{newsigma} we have that
$$
\frac{(p-2+\sigma_{k+1})(p+3+\sigma_k)}{p-1+\sigma_k}=
p+2+\sigma_k\;,
$$
and hence the second inclusion in \rif{asshix} follows with $\sigma
\equiv \sigma_{k+1}$ by the first inclusion in ${\bf (A)_k}$ and
\rif{hitk}. Therefore Lemma \ref{HiX} and \rif{hhcc} with $\sigma\
\equiv \sigma_{k+1}$ finally imply that $ \Xu \in
L_{\loc}^{p+2+\sigma_{k+1}}(\Omega,\er^{2n}). $ Concerning the
remaining part of the proof of ${\bf (B)_{k+1}}$ observe that
\rif{trisig} allows for applying the elementary inequality $ |\XXX
u|^{p+\sigma_{k+1}}\leq|\XXX u|^{p+2+\sigma_{k}}+1 $; this, together
with \rif{esttk1} and \rif{estint1}, since the open subsets involved
everywhere are arbitrary, allows in turn to conclude that
$$
\int_{\Omega'} |\Xu|^{p+2+\sigma_{k+1}}\dx \leq c \int_{\Omega''}
\big(|\Xu|^{p+2+\sigma_k}+|Tu|^{\frac{p+2+\sigma_k}{2}}+1\big)\dx\,.
$$
At this point the full inequality in ${\bf (B)_{k+1}}$ follows by
the previous one together with \rif{esttk1} and ${\bf (B)_{k}}$,
after changing, accordingly, the open subsets
$\Omega', \Omega''$ involved.

In this way both ${\bf (A)_{k}}$ and ${\bf (B)_{k}}$ hold for every
$k \in \mathbb N$.

Now we prove the validity of \rif{betterint22} and
\rif{apriori}. The assertions in \rif{horiz} are immediate, while to
prove \rif{apriori} with a fixed $q$, take $k$ large enough such
that $(p+2+\sigma_k)/2 \geq q$, in order to estimate $ |\XXX u|^q+|T
u|^q \leq  |\XXX u|^{p+2+\sigma_k}+|T u|^{\frac{p+2+\sigma_k}{2}}+2
$, and then apply ${\bf (B)_{k}}$  in order to get
$$
\int_{\Omega'} (|\XXX u|^{q}+|T u|^{q})\, dx \leq c \int_{\Omega''}
\left(|\XXX u|^p +|Tu|^p+1\right)\, dx\;.
$$
Finally, changing again the subsets, the final form of \rif{apriori}
follows by Theorem \ref{domokos}.

{\em Step 2: Blow-up and local estimates.} Now, by means of scaling
arguments, we shall see how to get the precise form of the a priori
estimates in \rif{apap}-\rif{apat} from the rough one in
\rif{apriori}; of course we shall assume that $q>p$. First, let us
consider the case of a solution $v \in HW^{1,p}(B(0,1))$ to
\rif{due}, that is, when $\Omega \equiv B(0,1)\equiv B_1$. In the
following $\gamma$ will denote a number such that $\gamma \in
(0,1)$, and the constants in the subsequent estimates will
deteriorate when $\gamma \nearrow 1$. Applying Theorem \ref{CDG} we
find \eqn{inest2}
 $$
\|v\|_{L^{\infty}(B_{\gamma})} \leq c_1
\left(\|v\|_{L^{p}(B_1)}+\mu\right)\;,
 $$
where $c_1 \equiv c_1(n,p,\ratio,\gamma)$. Now let us define, for
every $z \in \er^{2n}$ \eqn{scaling}
$$
w:=\frac{v}{A}\qquad \mbox{and
}\qquad\tilde{a}(z):=\frac{a(Az)}{A^{p-1}}\,,
$$
where \eqn{hereAis}
$$
A:=c_1 \left(\|v\|_{L^{p}(B_1)}+\mu\right)\;.
$$
Obviously $A>0$ and moreover \eqn{muA}
$$
\mu/A \leq 1\;.
$$ The new scaled function $w$ weakly solves the
equation \eqn{scalaa}
$$ \divo \tilde{a}\!\left(\mathfrak{X}w\right) =0\,,
$$
and, as a consequence of \rif{inest2}, it is such that \eqn{ttyy}
$$\|w\|_{L^{\infty}(B_{\gamma})}\leq 1 \;.$$
Moreover an easy computation reveals that the new vector field
$\tilde{a}(z)$ defined in \rif{scaling} satisfies assumptions
\rif{growth}-\rif{ell} with $\mu$ replaced by $\mu/A$. Therefore,
keeping again \rif{muA} in mind, applying estimate \rif{apriori} to
$w$ with the choice $\Omega' =B_{\gamma^2}$ and $\Omega''
=B_{\gamma}$, yields \eqn{stimaw}
$$
\int_{B_{\gamma^2}} (|\XXX w|^q+|Tw|^q)\, dx \leq c_2 \int_{B_1}
(|\XXX w|^p +1)\, dx\;,
$$and the constant $c_2$
depends now only on $n,p,\ratio,q,\gamma$ by the inequality in
\rif{ttyy}. Scaling back to $v$, that is taking \rif{scaling} into
account, \rif{stimaw} gives
\begin{eqnarray}
&& \nonumber \int_{B_{\gamma^2}} (|\XXX v|^q+|Tv|^q)\, dx  \leq
c_2\left[c_1 \left(\|v\|_{L^{p}(B_1)}+\mu\right)\right]^{q-p}
\int_{B_1} |\XXX v|^p\, dx\\ && \hspace{5cm}+
|B_1|c_2\left[c_1\left(\|v\|_{L^{p}(B_1)}+\mu\right)\right]^{q}\,.\label{scalb}
\end{eqnarray}
Applying Young's inequality with conjugate exponents $q/p$ and
$q/(q-p)$ to estimate the first quantity in the right hand side of
\rif{scalb} easily gives \eqn{stimav2}
$$ \|\XXX v\|_{L^{q}(B_{\gamma^2})} +\|T v\|_{L^{q}(B_{\gamma^2})}  \leq c \left(\|\XXX
v\|_{L^{p}(B_1)}+\|v\|_{L^{p}(B_1)}+\mu\right) \;,
$$
where $c\equiv c(n,p,\ratio,q,\gamma)$. Now we observe that if
 $v$ solves \rif{due} then $v-\xi$ also solves \rif{due} whenever $\xi \in \er$.
Therefore we apply estimate \rif{stimav2} to $v-(v)_{B_1}$, and
using it together with Jerison's Poincar\'e's inequality - see
\cite{jerison, Lup} - that is $ \|v-(v)_{B_1}\|_{L^{p}(B_1)}\leq
c(n,p)\|\XXX v\|_{L^{p}(B_1)}$, we finally get \eqn{stimav3}
$$ \|\XXX v\|_{L^{q}(B_{\gamma^2})} +\|T v\|_{L^{q}(B_{\gamma^2})}  \leq c \||\XXX
v|+\mu\|_{L^{p}(B_1)}\;,
$$
where $c\equiv c(n,p,\ratio,q,\gamma)$; observe that the constant
$c$ blows-up whenever: $\gamma \nearrow 1$,  $q \nearrow \infty$,
$p\nearrow 4$. Choosing $\gamma=1/\sqrt{2}$ in \rif{stimav3}, we
immediately get that \eqn{stimav4}
$$
\left(\intav_{B_{1/2}}(|\XXX v|^q+|T v|^{q})\, dx\right)^{1/q}\leq
c\left(\intav_{B_{1}}(\mu+|\XXX v|)^{p}\, dx\right)^{1/p}\;,
$$
with $c \equiv c(n,p,\ratio,q)$, and this means that we have proved
\rif{apap}-\rif{apat} in the case $R=1$.

Now we can go back to the original solution $u$, taking a CC-ball
$B_R \equiv B(x_0,R) \subset \Omega$, and defining \eqn{riscala}
$$
v(x):=\frac{u(x_0\cdot \delta_{R}(x))}{R}\,,\qquad \mbox{for every}\
x \in B(0,1)\;,
$$
where the dilation operator $\delta_R$ has been defined in
\rif{dilation}. Now observe that for every $i=1,\ldots,2n$
\eqn{relazione}
$$
X_i v(x) = X_i u(x_0\cdot \delta_{R}(x))\qquad \mbox{and} \qquad  T
v(x) = RT u(x_0\cdot \delta_{R}(x))\;.$$ Using this fact, and again
the left invariance of the vector fields $\{X_i\}$, it is easy to
see that the function $v$ defined in \rif{relazione} solves the
equation \rif{due} in $B(0,1)$, and therefore \rif{stimav4} is
applicable. In fact, using \rif{stimav4} for $v$, re-scaling back to
$u$ in $B(x_0,R)$, and using \rif{relazione} we get
\rif{apap}-\rif{apat}. Observe that in such a re-scaling procedure
the appearance of the integral averages in \rif{apap}-\rif{apat} is
essentially due to the change-of-variable formula together with the
fact that $ \mbox{det}\ (x \mapsto x_0\delta_{R}(x))\approx R^Q
\approx |B(x_0,R)|.$ This is basically a consequence of
\rif{basicrel}.
\end{proof}
\section{Non-degenerate equations}
\begin{prop} \label{horiznon}
 Let $u \in HW^{1,p}(\Omega)$ be a weak solution to the equation \trif{due}
 under the assumptions \trif{growth}-\trif{nondeg}, with $
2\leq  p < 4.$ Then it holds that \eqn{betterintnon}
 $$ \mathfrak{X}u \in
L^{\infty}_{\textnormal{loc}}(\Omega,\er^{2n}) \qquad  \mbox{and}
\qquad Tu \in L^{\infty}_{\textnormal{loc}}(\Omega) \;.$$ Moreover there
exists a constant $c$, depending on $n,p$ and $\ratio$, but
otherwise independent of $\mu$, of the solution $u$, and of the
vector field $a(\cdot)$, such that \trif{apapinf}-\trif{apatinf}
hold for any CC-ball $B_R \subset \Omega$.
\end{prop}
\begin{proof}
The proof is again divided in two steps. First we treat a special
case; then we reduce to such a special case by a blow-up argument.

{\em Step 1: Universal estimates.} Here we assume that
 \eqn{uni1}
$$ \Omega \equiv B_1 \qquad  \qquad \mbox{and} \qquad \qquad \|\Xu \|_{L^p(B_1,
\er^{2n})}\leq 1\;,
$$
and we shall prove that there exist absolute constants $c_3,c_4
\equiv c_3,c_4(n,p,\ratio)$  such that \eqn{firstversion}
$$
\sup_{B_{1/2}} |\Xu|\leq c_3, \qquad \mbox{and}\qquad \sup_{B_{1/2}}
|Tu|\leq c_4\mu^{\frac{Q(2-p)}{4}}\;.
$$
With $\gamma =99/100$, a simple covering argument and
\rif{apap}-\rif{apat}, gives that \eqn{uni2}
$$
\int_{B_{\gamma}} \left(|\Xu|^{\frac{Q(p+2)}{Q-1}} +
|Tu|^{2Q}+|Tu|^{2}\right)\dx \leq c\;,
$$
where $c$ is a constant depending only on the quantities
$n,p,\ratio$. Note that we have used \rif{uni1} to get rid of
the dependence on the norms of $\Xu, Tu$ in the constant $c$. Now we
start from \rif{XX sigma estpre}, which we shall employ to implement
a suitable variant of Moser's iteration scheme. With $\eta\in
C^\infty_0(B_\gamma)$ being non-negative and such that $\eta \leq 1$
we immediately have that for any $\sigma \geq 2$ it does hold that
\begin{align}\label{XXuni0}
\int_{\Omega} \eta^2 &\weight
\sum_{s=1}^{2n} \deltaXi^\frac{\sigma}{2} |\X X_s u|^2 \dx \nonumber\\
&\leq c (\sigma +1) C_\eta\int_{\supp \eta} \sum_{s=1}^{2n}
(\mu^2+|X_su |^2)^{\frac{p+\sigma}{2}}\dx
%\sum_{i=1}^{2n} \deltaXi^\frac{\sigma}{2} \dx
\nonumber\\
&\qquad \qquad+ c (\sigma +1)^3 \int_\Om  \eta^2 |Tu|^2
\sum_{s=1}^{2n} (\mu^2+|X_su |^2)^{\frac{p-2+\sigma}{2}} \dx\;,
\end{align}
where we have set \eqn{ceta}
$$
C_\eta:=\|\X \eta\|^2_{L^\infty}+ \|T\eta\|_{L^\infty}+1\;.
$$
To estimate the last term appearing in \rif{XXuni0} we use
H\"older's inequality and then \rif{uni2}, thereby gaining
\begin{eqnarray*} && \int_{\Om} \eta^2 |Tu|^2 \sum_{s=1}^{2n}
(\mu^2+|X_su
|^2)^{\frac{p-2+\sigma}{2}} \dx \\
&&\leq c(n)\left(\int_{B_\gamma}  |Tu|^{2Q}dx \right)^{\frac{1}{Q}}
\left(\int_{\supp \eta} \sum_{s=1}^{2n} (\mu^2+|X_su
|^2)^{\frac{Q(p-2+\sigma)}{2(Q-1)}}\dx \right)^{\frac{Q-1}{Q}}\\
&& \leq c\left(\int_{\supp \eta} 1+ \sum_{s=1}^{2n} (\mu^2+|X_su
|^2)^{\frac{Q(p+\sigma)}{2(Q-1)}}\dx \right)^{\frac{Q-1}{Q}} \;,
\end{eqnarray*}
where, as we used \rif{uni2}, the constant $c$ in the last line
depends on $n,p,\ratio$. Moreover, again by H\"older's
inequality, it trivially follows that
\begin{eqnarray}
\nonumber && \int_{\supp \eta} \sum_{s=1}^{2n} (\mu^2+|X_su
|^2)^\frac{p+\sigma}{2}\dx  \leq  c(n)\left(\int_{\supp \eta}
\sum_{s=1}^{2n} (\mu^2+|X_su
|^2)^\frac{Q(p+\sigma)}{2(Q-1)}\dx\right)^{\frac{Q-1}{Q}}\\
&  & \hspace{3cm}\leq c(n)\left(\int_{\supp \eta} 1+\sum_{s=1}^{2n}
(\mu^2+|X_su
|^2)^\frac{Q(p+\sigma)}{2(Q-1)}\dx\right)^{\frac{Q-1}{Q}}
\;.\label{uni20}
\end{eqnarray}
The last two estimates together with \rif{XXuni0}, and again
H\"older's inequality, give
\begin{align} \label{XXuni1}
\int_{\Omega} \eta^2 &
\sum_{s=1}^{2n} (\mu^2+|X_su |^2)^\frac{p-2+\sigma}{2} |\X X_s u|^2 \dx \nonumber\\
&\qquad \leq c (\sigma +1)^3C_\eta \left(\int_{\supp \eta} 1+
\sum_{s=1}^{2n} (\mu^2+|X_su |^2)^{\frac{Q(p+\sigma)}{2(Q-1)}}\dx
\right)^{\frac{Q-1}{Q}} \;,
\end{align}
where $c \equiv c(n,p,\ratio)$ and $C_{\eta}$ is defined in
\rif{ceta}. Now we observe that
\begin{eqnarray*}
&&| \XXX\sum_{s=1}^{2n}  \eta (\mu^2+|X_su |^2)^\frac{p+\sigma}{4}
|^2\leq  c(n)C_\eta\sum_{s=1}^{2n} (\mu^2+|X_su
|^2)^\frac{p+\sigma}{2}\\ && \qquad \qquad \qquad +
c(n)(p+\sigma)^2\eta^2 \sum_{s=1}^{2n} (\mu^2+|X_su
|^2)^\frac{p-2+\sigma}{2} |\X X_s u|^2\,.
\end{eqnarray*}
Therefore, using again \rif{uni20}, the last estimate and
\rif{XXuni1} give
\begin{align}
\int_{B_\gamma} & | \XXX\sum_{s=1}^{2n} \eta (\mu^2+|X_su
|^2)^\frac{p+\sigma}{4} |^2 \dx\nonumber\\
&\qquad \leq c (p+\sigma)^5C_\eta \left(\int_{\supp \eta} 1+
\sum_{s=1}^{2n} (\mu^2+|X_su |^2)^{\frac{Q(p+\sigma)}{2(Q-1)}}\dx
\right)^{\frac{Q-1}{Q}} \;.
\end{align}
Applying Sobolev embedding theorem in the Heisenberg group, that is
Theorem \ref{subsobt} with $q=2$, in turn yields
\begin{align}\label{XXuni2}
&\left(\int_{B_\gamma}\eta^{\frac{2Q}{Q-2}} \sum_{s=1}^{2n}
(\mu^2+|X_su
|^2)^\frac{Q(p+\sigma)}{2(Q-2)}\dx\right)^{\frac{Q-2}{Q}} \nonumber\\
&\qquad \leq c (p+\sigma)^5C_\eta \left(\int_{\supp \eta} 1+
\sum_{s=1}^{2n} (\mu^2+|X_su |^2)^{\frac{Q(p+\sigma)}{2(Q-1)}}\dx
\right)^{\frac{Q-1}{Q}} \;,
\end{align}
where the constant $c$ depends only on $n,p,\ratio$. Observe that here we are using that $\supp \eta
\subset B_\gamma$. Now we choose the cut-off functions in the
framework of Moser's iteration technique. We take a family of
concentric interpolating balls $B_{3/4}\subset B_{\varrho_{k+1}}
\subset B_{\varrho_{k}}$ such that $B_{\varrho_0} = B_{7/8}\subset
B_{\gamma}$, $\varrho_{k+1}-\varrho_k \approx 2^{-k}$ and $\varrho_k
\searrow 3/4$. Accordingly we select $\eta_k \in
C^{\infty}_c(B_{\varrho_k})$ such that $\eta_k \equiv 1$ on
$B_{\varrho_{k+1}}$, and $C_\eta\leq c^k$; the existence of such
cut-off functions can be inferred as in \cite[Lemma 3.2]{CDG}.
Setting \eqn{tildechi}
$$\tilde{\chi} :=
\frac{Q-1}{Q-2}>1\,,$$ we recursively define the sequence
$\{\sigma_k\}$ as follows:
$$
\left\{
\begin{array}{ccc}
 \sigma_{k+1} & :=& \tilde{\chi} \sigma_k+ \frac{p}{Q-2}
\displaystyle \\
\\ \sigma_0 & := &2\,,
\end{array}
\right.
$$
so that \eqn{easy1}
$$
\frac{(p+\sigma_{k+1})Q}{Q-1} = \frac{(p+\sigma_{k})Q}{Q-2}
$$
holds for every $k\geq 0$. Observe that \eqn{trisigma}
$$
p+\sigma_k \approx \tilde{\chi}^{k}, \qquad \mbox{and}\qquad
|B_{\varrho_k}|\approx c(n)>0\,.
$$
Taking $\sigma\equiv \sigma_k$ and $\eta \equiv \eta_k$ in
\rif{XXuni2}, and observing that $\eta_k \equiv 1$ on
$B_{\varrho_{k+1}}$ and $\supp \eta_k \subset B_{\varrho_k}$, easily
gives
\begin{align}\label{XXuni3}
&\left(\int_{B_{\varrho_{k+1}}} 1+ \sum_{s=1}^{2n} (\mu^2+|X_su
|^2)^\frac{Q(p+\sigma_k)}{2(Q-2)}\dx\right)^{\frac{Q-2}{Q}} \nonumber\\
&\qquad \leq c^{k+1} \left(\int_{B_{\varrho_k}} 1+ \sum_{s=1}^{2n}
(\mu^2+|X_su |^2)^{\frac{Q(p+\sigma_k)}{2(Q-1)}}\dx
\right)^{\frac{Q-1}{Q}} \;,
\end{align}
where $c \equiv c(n,p,\ratio)\geq 1$ is a constant independent of
$k$, and we used \rif{trisigma}. Now, setting for every $k\geq 0$
$$
A_k:=\left(\int_{B_{\varrho_k}} 1+ \sum_{s=1}^{2n} (\mu^2+|X_su
|^2)^{\frac{Q(p+\sigma_k)}{2(Q-1)}}\dx
\right)^{\frac{Q-1}{Q(p+\sigma_k)}}\,,
$$
using \rif{easy1}-\rif{XXuni3}, an elementary manipulation gives
that
$$
A_{k+1}\leq c_0^{(k+1)\tilde{\chi}^{-k}}A_k\;,
$$
for a new constant $c_0$ depending only on $n,p,\ratio$. Keeping
\rif{tildechi} in mind, iterating the previous inequality easily
gives
$$
A_{k+1}\leq \exp\left[(\log c_0)\sum_{i=0}^{\infty}\frac{i+1}{\tilde{\chi}^{i}}\right] A_0\;.
$$
Letting $k\nearrow \infty$ in the previous inequality - note that the series in the last line converges by \rif{tildechi} - now gives
\eqn{prexx}
$$
\sup_{B_{3/4}} |\Xu|\leq c(n,p,\ratio) A_0\,,
$$
while taking \rif{uni2} and the fact that $\mu \leq 1$ into account
we obtain the first inequality appearing in \rif{firstversion}. As
for the second inequality in \rif{firstversion}, we observe that since
$\Xu$ is bounded we may apply Theorem \ref{verticaluno} with any $q$ satisfying \rif{chichi}. Noting that this implies
$2q/(q-p+2)\leq 2Q$, we may use \rif{uni2}; therefore taking $R=3/4$ and $\varrho =
1/2$ in \rif{apriorivert00} yields
\begin{equation} \|Tu\|_{L^{\infty}(B_{1/2})} \le
\tilde{c}c^{\frac{\chi}{\chi-1}} \mu^{\frac{(2-p)\chi}{2(\chi-1)}}
\label{apriorivert002},
\end{equation}
where we also used \rif{prexx}, $\tilde{c}\equiv
\tilde{c}(n,p,\ratio)$, and where $\chi$ appears in \rif{chichi0}.
All the constants in the above inequality only depend on
$n,p,\ratio$ and are actually independent of $q$. Therefore letting
$q\nearrow \infty$ in \rif{apriorivert002}, and keeping
\rif{chichi0} in mind, we obtain the
second inequality in \rif{firstversion} with the specified
dependence of the constant $c_4$.

 {\em Step 2: The general case.} First we observe
that we may reduce to the case $B_R \equiv B_1$ by performing the
blow-up scaling \rif{riscala}. Indeed once estimates
\rif{apapinf}-\rif{apatinf} hold for $v$ on $B_R \equiv B_1$, then
scaling back, and using \rif{relazione}, they also hold on general
balls $B_R$ as required in the statement. Therefore we just need to
prove the result for a solution $v$ in the ball $B_1$. In order to
reduce to the assumptions in \rif{uni1} we pass to the function $w$
defined in \rif{scaling} where this time we choose $A:= \left(\|\XXX
v\|_{L^{p}(B_1)}+\mu\right),$ so that both $\|\XXX
w\|_{L^p(B_1,\er^{2n})}\leq 1$ and \rif{muA} hold. As noted in the
proof of Proposition \ref{horiz}, Step 2, the function $w$ is a
solution of the equation \rif{scalaa}, while the new vector field
$\tilde{a}(z)$ defined in \rif{scaling} satisfies assumptions
\rif{growth}-\rif{ell} with $\mu$ replaced by $\mu/A\leq 1$.
Therefore, thanks to \rif{muA} we may apply the result of Step 1 to
$w$, thereby obtaining \eqn{firstversion2}
$$
\sup_{B_{1/2}} |\XXX w|\leq c_3, \qquad \mbox{and}\qquad
\sup_{B_{1/2}} |Tw|\leq
c_4\mu^{\frac{Q(2-p)}{4}}A^{\frac{Q(p-2)}{4}}\;.
$$
Going back to $v = w/A$, and keeping in mind the current definition
of $A$, we obtain the validity of \rif{apapinf}-\rif{apatinf} for
$v$ on $B_1$, and the proof is finally complete by the argument
outlined at the beginning of Step 2.\end{proof}
\begin{proof}[Proof of Theorems \ref{main}-\ref{main2}] The proof
of the a priori estimates of Theorem \ref{main2} is a direct
consequence of Proposition \ref{horiznon}. As far as the H\"older
continuity of the gradient is concerned, the focal point of the
regularity theory for quasilinear elliptic equations with $p$-growth
is the local Lipschitz regularity of solutions, as already explained
in \cite{Ca1, Ca2, MM}. From this point on the proof of the local
H\"older continuity of $D u$ proceeds as in \cite{MM}; see also
\cite{Ca0, Ca2} for detailed explanations.
\end{proof}
\section{The degenerate case}

\begin{proof}[Proof of Theorem \ref{main3}] Of course in the following
we shall restrict to the case $p>2$; indeed, as the reader will soon
recognize, in the case $p=2$ the role of $\mu$ is immaterial in
\rif{growth}-\rif{ell}, and the results of Theorems \ref{main} and
\ref{main2} still hold when $\mu=0$. When $p>2$ the case $\mu=0$ is
now a consequence of Proposition \ref{horiz} when combined with a
suitable approximation argument we are going to report in some
detail. Let us consider the regularized vector fields
$$ a_{k}(z):= a(z)+ \varepsilon_k^{p-2}z\,,\qquad \mbox{for every}\ \  z \in \er^{2n}\ \ \mbox{and}\ \  k \in \mathbb N\,,$$
where $\{\varepsilon_k\}_k$ is a sequence of positive numbers such
that $\varepsilon_k \searrow 0$ and $\varepsilon_k \leq 1$. By using
\rif{growth}-\rif{ell} it is easy to see that each vector field
$a_k(z)$ satisfies the following growth and ellipticity conditions:
\begin{equation}\label{growthk}
|Da_k(z)|(\varepsilon_k^2+|z|^2)^{\frac{1}{2}} + |a_k(z)| \leq
c(\varepsilon_k^2+|z|^2)^{\frac{p-1}{2}},
\end{equation}
and
\begin{equation}\label{ellk}
c^{-1} (\varepsilon_k^2+|z|^2)^{\frac{p-2}{2}}|\lambda|^2 \leq
\sum_{i,j=1}^{2n}D_{z_j}(a_k)_i(z)\lambda_i\lambda_j  ,
\end{equation}
for a constant $c>0$ depending only on $n,p,\ratio$ but independent
of $k \in \mathbb N$. Moreover, since $p\geq 2$, assumption
\rif{ellk} also implies, for a possibly different constant $c$ still
depending on $n,p,\ratio$, but otherwise independent of $k \in
\mathbb N$, that whenever $z,z_1,z_2 \in \er^{2n}$ the following
inequalities hold: \eqn{mon2d}
$$   c^{-1}
|z_2-z_1|^p \leq  \langle a_k(z_2)-a_k(z_1),z_2-z_1\rangle\;, \qquad
c^{-1}|z|^p-c\varepsilon_k^p \leq \langle a_k(x,z),z\rangle\;.$$
Compare with \rif{mon2} and \rif{veramon}. Now, let us consider a
CC-ball $B_R \subset \Omega$ and let us define $u_k \in
u+HW^{1,p}_0(B_R)$ as the unique solution to the Dirichlet problem
\rif{Dir0} with $a_k(\cdot)\equiv a(\cdot)$; therefore, for
the present application we have $v\equiv u_k $ in \rif{Dir0}.
Accordingly, by virtue of \rif{mon2d} we may apply Lemma
\ref{cotril} so that \rif{cotri} used for $v \equiv u_k$ gives
\eqn{bobo}
$$
\int_{B_{R}}|\XXX u_k|^p\, dx \leq c\int_{B_{R}}(\varepsilon_k+|\XXX
u|)^p\, dx\;,
$$
where $c\equiv c(n,p,\ratio)$ is independent of $k$. Next, using
\rif{mon2d}, the fact that both $u$ and $u_k$ are solutions, and
then applying the definition of $a_k(\cdot)$ together with Young's
and H\"older's inequalities, we have
\begin{eqnarray*}
 \int_{B_{R}} |\XXX u_k - \XXX u|^p\,
dx & \leq & c\int_{B_{R}} \langle a(\XXX u_k)-a(\XXX u), \XXX u_k -
\XXX u\rangle \, dx\\& =& c\int_{B_{R}} \langle a(\XXX u_k)-a_k(\XXX
u_k), \XXX u_k - \XXX u\rangle \, dx\\&\leq  &
 c
\int_{B_{R}}\varepsilon_k^{p-2}|\Xu_k||\XXX u_k -
\XXX u|\, dx \\
&\leq& \frac{1}{2} \int_{B_{R}} |\XXX u_k - \XXX u|^p\,
dx+c\varepsilon_k^{\frac{p(p-2)}{p-1}} \int_{B_{R}}|\XXX
u_k|^{\frac{p}{p-1}} \, dx\\&\leq& \frac{1}{2} \int_{B_{R}} |\XXX
u_k - \XXX u|^p\, dx+c\varepsilon_k^{\frac{p(p-2)}{p-1}}
\left(\int_{B_{R}}|\XXX u_k|^p\,
dx\right)^{\frac{1}{p-1}}\;.\end{eqnarray*} Re-absorbing in the
l.h.s. the first integral in the last line, eventually letting $k
\nearrow \infty$, and keeping \rif{bobo} in mind, we get
\eqn{stconv}
$$
\XXX u_k \to \XXX u \qquad \mbox{ strongly in} \qquad
L^p(B_R,\er^{2n})\;.$$
 Now, using estimates \rif{apapinf}
and \rif{apatinf} for $u_k$, and therefore considering the case $\mu
\equiv \epsilon_k
>0$, we get \eqn{apapk}
$$
\sup_{B_{R/2}} |\Xu_k|\leq c^*
\left(\intav_{B_{R}}(\varepsilon_k+|\XXX u_k|)^{p}\,
dx\right)^{1/p}\;,
$$
and \eqn{apatk}
$$ \left(\intav_{B_{R/2}}|T u_k|^{q}\,
dx\right)^{1/q}\leq \frac{c_*}{R}
\left(\intav_{B_{R}}(\epsilon_k+|\XXX u_k|)^{p}\, dx\right)^{1/p}\;,
$$ which hold uniformly with respect to $k$; in fact
the constants $c^*, c_*$ ultimately depend on $n,p,\ratio$, and also
$q$ as far as the latter is concerned, but are otherwise independent
of $k$. This follows directly from the statement of Proposition
\ref{horiz}. Letting $k \nearrow \infty$ in \rif{apapk}-\rif{apatk},
standard lower semicontinuity arguments to deal with the left hand
sides of \rif{apapk}-\rif{apatk}, and \rif{stconv} to deal with
right hand ones, finally give \rif{apapinfdeg}-\rif{apat00}. Since
the ball considered $B_R\subset \Omega $ is arbitrary, this finally
implies \rif{betterint22} via a standard covering argument and the
proof of Theorem \ref{main3} is complete.
\end{proof}
\begin{proof}[Proof of Corollaries \ref{cor1}-\ref{cor2}] Corollary
\ref{cor2} is immediate since from Theorem \ref{main3} we obtain
higher integrability for the Euclidean gradient: $Du  \in
L^q_{\loc}(\Omega,\er^{2n+1})$ for every $q < \infty$. As for
Corollary \ref{cor1}, it suffices to prove estimate \rif{Lipest}.
With $B_R \subset \Omega$ as in the statement, by
\rif{apapinf}-\rif{apapinfdeg} it immediately follows that
$$
M_{R/4}(|\Xu|)(x)\leq \sup_{B(x,R/4)} |\Xu| \leq
c\left(\intav_{B_{R}}(\mu+|\XXX u|)^{p}\, dz\right)^{1/p}\,,
$$
whenever $x \in B_{R/2}$, where $c$ depends only on $n,p,\ratio$.
The operator $M_{R/4}$ is the one defined in \rif{rem}. Therefore,
using Proposition \ref{HSp} we obtain \eqn{Lipestpre}
$$
|u(x)-u(y) |\leq  c\left(\intav_{B_{R}}(\mu+|\XXX u|)^{p}\,
dz\right)^{1/p}d_{cc}(x,y)
$$
as soon as $x,y \in B_{R/2}$ are such that $d_{cc}(x,y)\leq R/8$. At
this stage estimate \rif{Lipest} follows from the last one, applied to
suitable smaller balls, just
magnifying the constant in \rif{Lipestpre} of a finite factor, say
$16$.
\end{proof}
\section{Horizontal Calder\'on-Zygmund estimates}
In this section we are going to prove Theorem \ref{phcz}; the use of various types of restricted maximal operator will be essential here.
In the
following, when dealing with \rif{duecz} we shall always assume that
$ F \in L^{q}_{\loc}(\Omega, \er^{2n})$, for some $q
>p$. Now, let us fix an arbitrarily fixed open subset $\Omega' \Subset
\Omega$ ; for the rest of the section all balls the considered $B$
will be such that $B \Subset \Omega'$ unless otherwise specified,
and in the following all the regularity results we are going to
prove are in $\Omega'$. Since the choice of $\Omega'$ is arbitrary
the corresponding local regularity of $\XXX u$ in
$\Omega$ will also follow. With $\tilde{q} \equiv
\tilde{q}(n,p,\ratio) >p$ being the higher integrability exponent
identified in Theorem \ref{gehx}, let us define \eqn{deq0}
$$
q_0:=\frac{p+\tilde{q}}{2}
$$
which is such that $q_0 \in (p,\tilde{q})$ and can be therefore used
in \rif{gerh}. Moreover, for later use we observe that \eqn{lu}
$$
q > \tilde{q} \Longrightarrow \frac{1}{q-q_0}<
\frac{2}{\tilde{q}-p}\equiv c(n,p,\ratio)
$$
and the last dependence on the parameters follows from the one
specified in Theorem \ref{gehring}. Accordingly, with $R_0>0$ being
fixed, and eventually specified later, and with $\Omega' \Subset
\Omega$ chosen as described above, we let \eqn{ugly}
$$
[b]_{R_0}^* \equiv [b]_{R_0,\Omega'}^*:=\sup_{B_R \subseteq \Omega',
R \leq R_0}\left(\intav_{B_R}
|b(x)-(b)_{B_R}|^{\left(\frac{p}{p-1}\right)\left(\frac{q_0}{q_0-p}\right)}\,
dx\right)^{\frac{q_0-p}{q_0}}\,,
$$
where $(b)_{B_R}$ is the average in \rif{mediavmo}. Let us observe
that \eqn{vmocond2}
$$
\lim_{R\searrow 0} [b]_{R,\Omega'}^*=0\;,
$$ for every choice of the open subset $\Omega' \Subset \Omega,$
and this strengthens \rif{vmocond}. Indeed since $|b(x)|\leq L$ by
\rif{vmo} we have
$$
[b]_{R,\Omega'}^*\leq (2L)^{\frac{p}{p-1}-\frac{q_0-p}{q_0}}
\left([b]_{R,\Omega'}\right)^{\frac{q_0-p}{q_0}}
$$ and \rif{vmocond2} immediately follows by \rif{vmocond}.
\begin{lemma}\label{compvmo}
 Let $u \in HW^{1,p}(B_R)$ be a weak solution to \trif{duecz},
 and let $v \in u+HW^{1,p}_0(B_R)$ be a weak solution to the
 Dirichlet problem \trif{Dir0} under the assumptions
 \trif{growth}-\trif{ell} for $p\geq 2$,
where $B_{2R}\Subset \Omega'$ and $R \leq R_0$, for a certain
$R_0>0$.

\textnormal{(1)} For any $ p \geq 2$ it holds that \begin{eqnarray}
\nonumber \intav_{B_{R}} |\XXX u - \XXX v|^p\, dx
&\leq & c_5 [b]_{R_0}^*\intav_{B_{2R}}(\mu + |\XXX u|)^p\, dx \\
& &\qquad \qquad + c_5(1+[b]_{R_0}^*)\left(
\intav_{B_{2R}}|F|^{q_0}\, dx\right)^{p/q_0}\;,\label{fre1}
\end{eqnarray}
where the constant $c_5$ depends only on $n,p,\ratio$.

\textnormal{(2)} Assuming $p \in [2,4)$ we have that for any $p \leq s<
\infty$ there exists a constant $c_6\equiv c_6(n,p,\ratio)$ such
that \eqn{fre2}
$$
\left( \intav_{B_{R/2}} (\mu+| \XXX v|)^{s}\,
dx \right)^{\frac{1}{s}}\leq c_6 \left( \intav_{B_R} (\mu+| \XXX u|)^{p}\, dx
\right)^{\frac{1}{p}}\;,
$$
and $c_6 \nearrow \infty$ when $p\nearrow 4$.
\end{lemma}
\begin{proof} (1). Using that $u$ and $v$ are solutions to \rif{duecz} and \rif{Dir0} respectively,
testing \rif{duecz} and \rif{Dir0} by $u-v \in HW^{1,p}_0(B_R)$ and
summing up, with $ (b)_{B_{R}}$ as in \trif{mediavmo} we have
\begin{eqnarray} \nonumber \intav_{B_R} |\XXX u - \XXX v|^p\, dx
&\stackrel{\rif{mon2}}{\leq} & c \intav_{B_R} (b)_{B_R} \langle
a(\XXX u)- a(\XXX v), \XXX
u - \XXX v\rangle \, dx\\
\nonumber  &=& c\intav_{B_R}  \langle (b)_{B_R} a(\XXX u)-
b(x)a(\XXX u), \XXX u - \XXX v\rangle \, dx\\&& \qquad \qquad
+c\intav_{B_R}\langle |F|^{p-2}F, \XXX u - \XXX v\rangle\, dx =:
I+II\;,\label{si00}
\end{eqnarray}
where $c \equiv c(n,p,\ratio)$. In a standard way, via Young's
inequality we have in turn \eqn{iiii}
$$
II \leq \frac{1}{4}\intav_{B_R} |\XXX u - \XXX v|^p\, dx +
c\intav_{B_R}|F|^p\, dx\;,
$$
while, taking \rif{growth} into account and using H\"older's
inequality we have
\begin{eqnarray*}
I &\leq & \frac{1}{4}\intav_{B_R} |\XXX u - \XXX v|^p\, dx +
c\intav_{B_R} |b(x)-(b)_{B_R} |^{\frac{p}{p-1}}(\mu+|\XXX u|)^p\,
dx\\
&\leq & \frac{1}{4}\intav_{B_R} |\XXX u - \XXX v|^p\, dx
+c[b]_{R_0}^*\left(\intav_{B_R} (\mu+|\XXX u|)^{q_0}\,
dx\right)^{p/q_0}\\
&\stackrel{\rif{gerh}}{\leq} & \frac{1}{4}\intav_{B_R} |\XXX u - \XXX v|^p\, dx\\
&& \qquad +c[b]_{R_0}^*\intav_{B_{2R}} (\mu+|\XXX u|)^p\,
dx+c[b]_{R_0}^* \left(\intav_{B_{2R}}|F|^{q_0}\, dx\right)^{p/q_0}
\;.
\end{eqnarray*}
Estimate \rif{fre1} now follows combining the estimates found for
$I$ and $II$ to \rif{si00}.

(2). When $p
 \in [2,4)$ estimate \rif{fre2} just follows applying \rif{apapinf}-\rif{apapinfdeg} to the
function $v$, and then applying \rif{cotri}.
\end{proof}
In the following we shall concentrate on a ball $B_{R_0}$, such that
$B_{100R_0} \subset \Omega'$. The symbol $\M$ will denote the
restricted maximal operator relative to the ball $B_{100R_0} $ in
the sense of \rif{ma1}: $\M\equiv \M_{B_{100R_0}}$; accordingly we shall
denote by $\Mq$ the restricted maximal operator in the sense of
\rif{ma2}, again relative to $B_{100R_0} $, that is, $\Mq\equiv
\M_{q_0/p,B_{100R_0}}$. We recall that $q_0>p$ has been defined in
\rif{deq0}.
\begin{lemma}\label{intecz} Let $u \in HW^{1,p}(\Omega)$ be a weak solution to
the equation \trif{duecz} under assumptions \trif{growth}-\trif{ell}
with $2\leq p < 4$, and let $K\geq 1$ and $s >p$. There exist
numbers $\varepsilon \equiv \varepsilon (n,p,\ratio,K,s) \in (0,1)$
and $A \equiv A (n,p,\ratio)\geq 1$ such that if $[b]_{100R_0}^*\leq
\varepsilon $ then the following holds:

If $B$ is a CC-ball centered in $B_{R_0}$ and with radius less than
$2R_0$ satisfying \eqn{cont1}
$$
|E \cap 5B| > K^{-s/p} |B \cap B_{R_0}|
$$
then \eqn{cont2}
$$
5B \cap B_{R_0} \subset G\;,
$$
where
$$
E:=\{x \in B_{R_0} \ : \ \M(\mu^p+|\XXX u|^p)(x)> AK \lambda, \
\mbox{and}\ \Mq(|F|^p)(x)\leq  \varepsilon \lambda\}\;,
$$
and
$$
G:=\{x \in B_{R_0} \ : \ \M(\mu^p+|\XXX u|^p)(x)>  \lambda\}\;,
$$
while $\lambda >0$.
\end{lemma}
\begin{proof} We proceed by contradiction, therefore assuming that
\rif{cont2} fails, and showing that, choosing $\varepsilon$ and $A$
appropriately, but with the dependence on the constants as in the
statement of the lemma, also \rif{cont1} fails. Indeed, assume that
\rif{cont2} fails but \rif{cont1} does not; then there exists $z_1
\in 5B\cap B_{R_0}$ such that $\M(\mu^p+|\XXX u|^p)(z_1)\leq  \lambda$;
moreover $E \cap 5B$ is non-empty and therefore there exists $z_2
\in 5B\cap B_{R_0}$ such that $\Mq(|F|^p)(z_2)< \varepsilon\lambda$.
All in all we have that \eqn{sti1}
$$
\intav_{40B} (\mu^p+|\XXX u|^p)\, dx \leq \lambda, \qquad
\mbox{and}\qquad  \intav_{40B} |F|^{q_0}\, dx \leq
(\varepsilon\lambda)^{q_0/p}\;.
$$
Now define $v \in u + HW^{1,p}_0(20B)$ as the unique solution to the
Dirichlet problem \rif{Dir0} with $B_R \equiv 20B$. Therefore
applying \rif{fre1} in this context, and using \rif{sti1} with
$[b]_{100R_0}^*\leq \varepsilon $ too, an elementary manipulation
gives \eqn{sti2}
$$  \intav_{20B} |\XXX u - \XXX v|^p\, dx \leq c(n,p,\ratio)
\varepsilon \lambda\;. $$ Moreover estimates \rif{fre2} and
\rif{sti1} also give \eqn{sti3}
$$
\intav_{10B} (\mu^s+|\XXX v|^s) \, dx \leq [c(n,p,\ratio)]^{s/p}
\lambda^{s/p}\;.
$$
We now start giving a few estimates for the restricted maximal
operator relative to the ball $10B$, that in the following will be
denoted by $M^{**}$, therefore $M^{**}\equiv \M_{10B}$. First, let us
observe that a standard geometric argument using that $\M(\mu^p+|\XXX u|^p)(z_1)< \lambda$, exactly the same as the one
working in the Euclidean case, allows us to get the existence of an
absolute constant $c_*$, depending on the doubling constant $C_d$ in
\rif{doubling} and therefore ultimately on $n$, such that
\eqn{duemax}
$$
\M(\mu^p+|\XXX u|^p)(x)\leq \max\{M^{**}(\mu^p+|\XXX u|^p)(x),
c_*\lambda\}\;,
$$
whenever $x \in 5B \cap B_{R_0}$. Now, using \rif{weakes} with
$\gamma = s/p$, we have
\begin{eqnarray}
&&\hskip-15mm |\{ x \in 5B\ : \ M^{**}(\mu^p+|\XXX
u|^p)(x)>AK\lambda
\}| \nonumber\\
& \leq& |\{ x \in 5B\ : \
M^{**}(\mu^p+|\XXX v|^p)(x)>2^{-p}AK\lambda \}| \nonumber\\
&& \qquad \qquad +|\{ x \in 5B\ : \
M^{**}(|\XXX u - \XXX v|^{p})(x)>2^{-p}AK\lambda \}| \nonumber\\
&\stackrel{\rif{weakes}}{\leq} &
\frac{2^{s/p+p}c(n,p)}{(AK\lambda)^{s/p}} \int_{10B}(\mu^s+|\XXX
v|^{s}) \,dx + \frac{c(n,p)}{AK\lambda}
\int_{10B}|\XXX u-\XXX v|^{p} \,dx \nonumber \\
&\stackrel{\rif{sti2}-\rif{sti3}}{\leq} &
\frac{2^{s/p}[c(n,p,\ratio)]^{s/p}|B|}{(AK)^{s/p}} +
\frac{c(n,p,\ratio)\varepsilon |B|}{AK} \nonumber \\
&\leq & \frac{2^{s/p}[c_7(n,p,\ratio,s)]^{s/p}|B\cap B_{R_0}|}{(AK)^{s/p}} +
\frac{c_8(n,p,\ratio)\varepsilon |B\cap B_{R_0}|}{AK} \;.\label{rim}
\end{eqnarray}
In the last inequality we used the fact that $B$ is a ball centered
in $B_{R_0}$ whose radius does not exceed $2R_0$, and the doubling
condition \rif{doubling}. Now we fix $A \equiv A(n,p,\ratio)>1+ c_*$
large enough in order to have $(2c_7/A)^{s/p}\leq 2c_7/A\leq
1/4$; here $c_*\equiv c_*(n)$ is the constant appearing in
\rif{duemax}. Then we take $\varepsilon \equiv \varepsilon
(n,p,\ratio,K)$ in order to have $c_8\varepsilon K^{s/p-1}< 1/4$.
Such choices fix the quantities $A$ and $\varepsilon$ with the
dependence on the constants described in the statement of the lemma,
and together with \rif{rim} they give
$$
|\{ x \in 5B\cap B_{R_0}\ : \ M^{**}(\mu^p+|\XXX u|^p)(x)>AK\lambda
\}| < K^{-s/p}|B|\;.
$$
Now, since $K \geq 1$ and $A> c_*$, by \rif{duemax} we also obtain
$$
|\{ x \in 5B\cap B_{R_0}\ : \ M^{*}(\mu^p+|\XXX u|^p)(x)>AK\lambda
\}| < K^{-s/p}|B|\;,
$$
that finally contradicts \rif{cont1}, and the proof is complete.
\end{proof}
\begin{proof}[Proof of Theorem \ref{phcz}]  The proof
is actually split in two cases. The first is when $q\leq \tilde{q}$,
and $\tilde{q} \equiv \tilde{q}(n,p,\ratio) >p$ is the higher
integrability exponent identified in Theorem \ref{gehx}. In this
case the assertion follows directly from such a theorem. The other
case is when $q > \tilde{q}$, to which we specialize henceforth.
Therefore, with $\tilde{q}< q < \infty$ as in the statement, we fix
a number $s$ such that $s>q$. Note that as a consequence of the
choice of $s \equiv s(q)$, from now on all the constants depending
on $s$ will be actually depending on $q$, and as such they will be
denoted, and in particular we determine the constant $A$ when
eventually using Lemma \ref{intecz}. Then we take $K>1$ large enough
in order to have \eqn{Kchoice}
$$
2K^{\frac{q-s}{p}}= A^{-\frac{q}{p}}\;.
$$
Such a choice fixes $K \equiv K(n,p,\ratio,q)$ and this is the
number we are going to take when using Lemma \ref{intecz}. Therefore
this determies the choice of $\varepsilon \equiv \varepsilon (n,p,\ratio,q)>0$
for the use in Lemma \ref{intecz}. Finally we determine the
radius $R_0 \equiv (n,p,\ratio,s ,b(\cdot))>0$ in such a way that
$[b]_{100R_0}^*\leq \varepsilon $. This is possible by
\rif{vmocond2}. Now, let us set \eqn{mu1}
$$
\mu_1(t):=|\{x \in B_{R_0}\ : \ \M(\mu^p+|\Xu|^p)(x)
> t\}|\;,
$$
\eqn{mu2}
$$
\mu_2(t):=|\{x \in B_{R_0}\ : \ \Mq(|F|^p)(x)
> t\}|\;,
$$
and keep in mind that the  maximal operators $\M_{q_0/p}$ are
restricted to the ball $B_{100R_0}$. The proof will proceed by
iterating the function $\mu_1(\cdot)$ using information on
$\mu_2(\cdot)$, that is getting information on the measure of the
level sets of $|\XXX u|$, in terms of those of $|F|$. We choose the
``starting level" $\lambda_0$ as follows: \eqn{la0}
$$
\lambda_0 := 10 C_d^{10}c_WK^{s/p}\intav_{B_{100R_0}}
(\mu^p+|\Xu|^p)\,dx\;,
$$
where $C_d$ is the doubling constant appearing in \rif{doubling},
and $c_W\equiv c_{W}(n)$ is the constant appearing in \rif{weakes2}
for $\gamma=1$. Therefore using \rif{weakes2}, and that $AK>1$ we
find, for any $m \in \mathbb N$ \eqn{pre00}
$$\mu_1((AK)^m\lambda_0) \leq \mu_1(\lambda_0)\leq \frac{1}{2K^{s/p}}|B_{R_0}|\;.$$
Now we want to combine Lemma \ref{intecz} and Lemma \ref{metrico}.
More precisely, for every $m=0,1,2,\ldots$ we want to apply Lemma
\ref{metrico} with the choice $\delta = K^{-s/p}$ and
$$
E:=\{z \in B_{R_0} \ : \ \M(\mu^p+|\XXX u|^p)> (AK)^{m+1} \lambda_0,
\ \mbox{and}\ \Mq(|F|^p)< \varepsilon (AK)^m\lambda_0\}\;,
$$
$$
G:=\{z \in B_{R_0} \ : \ \M(\mu^p+|\XXX u|^p)>
(AK)^{m}\lambda_0\}\;.
$$
In fact using Lemma \ref{intecz} for $\lambda \equiv
(AK)^m\lambda_0$ in the context of Lemma \ref{metrico}, keeping
\rif{pre00} in mind, and recalling that $|G|= \mu_1(
(AK)^m\lambda_0)$ and that $|E|\geq  \mu_1((AK)^{m+1}\lambda_0)-\mu_2(
(AK)^m\varepsilon\lambda_0)$ we have
$$
\mu_1((AK)^{m+1}\lambda_0) \leq K^{-s/p}
\mu_1((AK)^{m}\lambda_0)+\mu_2((AK)^m\varepsilon \lambda_0)\;,
$$
for any $m=0,1,2,\ldots$. Induction on the previous inequality
easily gives
$$
\mu_1((AK)^{m+1}\lambda_0) \leq K^{-s(m+1)/p}
\mu_1(\lambda_0)+\sum_{i=0}^{m}K^{-s(m-i)/p}\mu_2(
(AK)^i\varepsilon\lambda_0)\;,
$$
and therefore, multiplying the previous inequalities by
$(AK)^{q(m+1)/p}$ and summing up on $m=0,1,\ldots,M \in \mathbb N$,
we have
\begin{eqnarray}
&& \sum_{m=0}^M(AK)^{q(m+1)/p}\mu_1((AK)^{m+1}\lambda_0)  \leq
\left(\sum_{m=0}^M [K^{-s/p}(AK)^{q/p}]^{m+1}\right)
\mu_1(\lambda_0)\nonumber \\ &&\hspace{4cm}
+\sum_{m=0}^M\sum_{i=0}^{m}(AK)^{q(m+1)/p}K^{-s(m-i)/p}\mu_2(
(AK)^i\varepsilon\lambda_0)\;.\label{ecco}
\end{eqnarray}
First, we notice that \rif{Kchoice} implies $$\sum_{m=0}^{\infty}
[K^{-s/p}(AK)^{q/p}]^{m+1}= 1 \;.$$ On the other hand, using
Fubini's theorem for series it easily follows that
\begin{eqnarray*}
&&\sum_{m=0}^{M}\sum_{i=0}^{m}
(AK)^{q(m+1)/p}K^{-s(m-i)/p}\mu_2((AK)^i\varepsilon\lambda_0)\\
& &\hspace{2cm} \leq 2(AK)^{q/p}\sum_{m=0}^{M}
(AK)^{qm/p}\mu_2((AK)^m\varepsilon\lambda_0)\;.
\end{eqnarray*}
Combining the last two inequalities with \rif{ecco}, and eventually
letting $M\nearrow \infty$, we obtain
\begin{eqnarray}
\nonumber && \sum_{m=1}^{\infty}(AK)^{qm/p}\mu_1((AK)^{m}\lambda_0)
\leq \mu_1(\lambda_0) +2(AK)^{q/p}\mu_2(\varepsilon \lambda_0)  \\
&& \hspace{5cm} +2(AK)^{q/p}\sum_{m=1}^{\infty}
(AK)^{qm/p}\mu_2((AK)^m\varepsilon\lambda_0)\;.\label{ecco2}
\end{eqnarray}
From now on keep in mind that $AK$ is a constant depending on
$n,p,\ratio,q$; without loss of generality we assume $AK \geq 2$.
Now, making a few elementary manipulations on \rif{ecco2} such as
$\mu_1(\cdot),\mu_2(\cdot)\leq |B_{R_0}|$, and using Fubini's
theorem, we estimate
\begin{eqnarray}
\nonumber \int_{B_{R_0}} (\mu+|\XXX u|)^q\, dx &\leq &c
\int_{B_{R_0}} [\M(\mu^p+|\XXX u|^p)]^{q/p}\, dx\\ \nonumber  &  = &
c\int_0^{\infty}
\lambda^{q/p-1} \mu_1(\lambda)d \lambda\\
\nonumber  &= & c\int_0^{\lambda_0} [\ldots] \, d \lambda +
c\int_{\lambda_0}^{\infty} [\ldots] \, d \lambda\\
\nonumber &\leq & c\lambda_0^{q/p}|B_{R_0}|+
c\sum_{m=0}^{\infty}\int_{(AK)^m\lambda_0}^{(AK)^{m+1}\lambda_0} [\ldots] \, d \lambda\\
\nonumber &\leq & c\lambda_0^{q/p}|B_{R_0}|+c
\lambda_0^{q/p}\sum_{m=0}^{\infty}(AK)^{qm/p}\mu_1((AK)^{m}\lambda_0)\\
&\stackrel{\rif{pre00}-\rif{ecco2}}{\leq } &
c\lambda_0^{q/p}|B_{R_0}|+
c\lambda_0^{q/p}\sum_{m=1}^{\infty}(AK)^{qm/p}\mu_2((AK)^{m}\varepsilon\lambda_0)\;,\label{ecco5}
\end{eqnarray}
with $c \equiv c(n,p,\ratio,q)$; moreover, \rif{la0}
yields\eqn{ecco3}
$$
\lambda_0^{q/p}|B_{R_0}|\leq c\left(\intav_{B_{100R_0}}
(\mu^p+|\Xu|^p)\,dx\right)^{q/p}|B_{R_0}|\;.
$$
In turn, again by means of Fubini's theorem and elementary
manipulations, we have
\begin{eqnarray}
\nonumber && \lambda_0^{q/p}\sum_{m=1}^{\infty}
(AK)^{qm/p}\mu_2((AK)^m\varepsilon\lambda_0)\leq
\frac{AK}{\varepsilon ^{q/p}(AK-1)}
\int_{0}^{\infty} \lambda^{q/p-1}\mu_2(\lambda)\, d \lambda\\
& &  \qquad \qquad\leq c \int_{B_{R_0}} [\Mq(|F|^p)]^{q/p}\,
dx\stackrel{\rif{weakes3}-\rif{lu}}{\leq} c\int_{B_{100R_0}} |F|^q\,
dx\;,\label{ecco4}
\end{eqnarray}
where, taking into account the peculiar dependence of $\varepsilon,
AK$, and also \rif{lu}, it turns out that the constant $c$ in the
last line depends only on $n,p,\ratio,q$. Connecting
\rif{ecco4}-\rif{ecco3} to \rif{ecco5}, we finally gain, after
further elementary manipulations \eqn{ecco8}
$$
\left(\intav_{B_{R_0}} |\XXX u|^q\, dx\right)^{1/q} \leq c
\left(\intav_{B_{100R_0}} (\mu^p+|\XXX u|^p)\, dx\right)^{1/p}
+c\left(\intav_{B_{100R_0}} |F|^q\, dx\right)^{1/q}\;.
$$
We have used again, and repeatedly, the doubling condition
\rif{doubling}; the constant $c$ depends on $n,p,\ratio,q$, but not
yet on $b(\cdot)$; the dependence on $q$ is such that $c$ blows-up
only when $q \nearrow  \infty$. Now notice that the only point to
use a ball with small radius $R_0$ in the above argumentation was to
fulfill the requirement $[b]_{100R_0}^*\leq \varepsilon $; therefore
estimate \rif{ecco8} continues to hold with $R_0$ replaced by any
other smaller radius, and therefore \eqn{ecco9}
$$
\left(\intav_{B_{R_1}} |\XXX u|^q\, dx\right)^{1/q} \leq c
\left(\intav_{B_{100R_1}} (\mu^p+|\XXX u|^p)\, dx\right)^{1/p}
+c\left(\intav_{B_{100R_1}} |F|^q\, dx\right)^{1/q}
$$
holds whenever $R_1\leq  R_0$ and $B_{100R_1} \Subset \Omega$.
 Summarizing, we have obtained
a first form of estimate \rif{apapd}, that is \rif{ecco9}, which is
valid for suitably small radii; moreover when estimating the left
hand side with the right-hand one we pass to an integral supported
on a ball with radius magnified of a factor $100$. In order to
derive the precise form \rif{apapd} we can proceed using a standard
covering argument at the end of which we shall get the desired
estimate, where the constant $c$ will be the one from \rif{ecco9},
magnified of a factor equal to $c(n,p,q)(R/R_0)^{Q(q-p)/p}$. Since
the radius $R_0$ has been chosen in order to verify
$[b]_{100R_0}^*\leq \varepsilon $ the final dependence of $c$ on
$b(\cdot )$ will follow. We hereby sketch the covering argument; we
first treat the most relevant case $R \geq R_0$. Consider a CC-ball
$B_{R} \Subset \Omega'$ with $R \geq R_0$, and cover $B_{R/2}$ with
a finite family of CC-balls $\{B_i\}$ with radius equal to
$R_0/1000$, centered in $B_{R/2}$, and such that the enlarged balls
have locally finite intersection in the following sense: every ball
$100B_i$ touches at most $c(n)$ of the other ones $100B_j$,
$i\not=j$. It clearly follows that $100B_i \Subset B_R$. The
existence of such a family follows considering the structure of the
CC-balls; see Section \ref{CCb}. We then apply \rif{ecco9} on every ball $B_i$ - this means we
are taking $R_1 = R_0/1000$ in \rif{ecco9} - and manipulate as
follows:
\begin{eqnarray}
&& \intav_{B_{R/2}} |\XXX u|^q\, dx \leq
c\left(\frac{R_0}{R}\right)^Q
\sum_i \intav_{B_{i}} |\XXX u|^q\, dx \nonumber\\
&& \qquad  \leq c \left(\frac{R_0}{R}\right)^Q \sum_i
\left(\intav_{100B_{i}} (\mu^p+|\XXX u|^p)\, dx\right)^{q/p}
\nonumber \\ && \hspace{4cm}+c\left(\frac{R_0}{R}\right)^Q
\sum_i\intav_{100B_{i}} |F|^q\, dx\nonumber\\
&& \qquad  \leq c \left(\frac{R_0}{R}\right)^QR_0^{-Qq/p}
\left(\int_{B_R} (\mu^p+|\XXX u|^p)\, dx\right)^{(q-p)/p}\sum_i
\int_{100B_{i}} (\mu^p+|\XXX u|^p)\, dx\nonumber \\ &&
\hspace{7cm}+c
\intav_{B_{R}} |F|^q\, dx\nonumber\\
&& \qquad  \leq c
\left(\frac{R}{R_0}\right)^{Q(q-p)/p}\left(\intav_{B_R} (\mu^p+|\XXX
u|^p)\, dx\right)^{q/p}+c \intav_{B_{R}} |F|^q\, dx\;, \label{ecco99}
\end{eqnarray}
where $c \equiv c(n,p,\ratio,q)$. Therefore estimate \rif{apapd}
follows in the case $R_0 \leq R$. The case $R<R_0$ can be treated in
a similar way, and it is actually almost contained in \rif{ecco9},
where $R_1\leq R_0$: we only need to pass from a ball $B_{R/2}$ to
$B_R$ instead of passing from $B_{R/100}$ to $B_R$ as in
\rif{ecco9}. This fact can be done via the same covering argument
used for the case $R_0\leq R$, by covering $B_{R/2}$ by small balls
with radius $R/1000$ and then perform the same computation as in
\rif{ecco99}; this  time since the radius of the balls $B_i$ is
comparable to that of $B_R$, when passing from estimate \rif{ecco9}
to \rif{apapd} the constant will magnify of a factor that depends
only on $n,p,\ratio,q$  but independent of $R_0$.
\end{proof}
\begin{remark}\label{alt} The argument at the end of the last proof leads to a statement which is
dual to the one in Theorem \ref{phcz}. Indeed it follows that for
every $q < \infty $ there exists a constant $c\equiv
c(n,p,\ratio,q)$ and a positive radius $R_0 \equiv
R_0(n,p,\ratio,q,b(\cdot))$ such that \rif{apapd} holds provided $R
\leq R_0$; this is actually the content of \rif{ecco9}. In this way
the constant $c$ is independent of $b(\cdot)$, while the dependence
on $b(\cdot)$ in the final estimate is shifted in $R_0$, that is
``the radius after which estimate \rif{apapd} starts to hold".
\end{remark}
\begin{remark}\label{alt2} The constant appearing in the estimate
\rif{apapd} blow-up when $p \nearrow 4$. As  far as the dependence
on $q$ is concerned, from the proof given we see that $c$ blows-up
when $q\nearrow \infty$, as it must be, while it remains stable when
$q\searrow p$. This last fact is basically a consequence of the use
of Theorem \ref{gehring} to prove \rif{apapd} when $q$ is ``close"
to $p$ - see the beginning of the section - and of inequality
\rif{lu} applied in \rif{ecco4}, when $q$ is ``larger" than $p$.
\end{remark}
\section{More equations} This section should be considered as an appendix to the previous one
in that we are describing here a few generalizations of the results
contained there. To begin with we observe that the result of
Theorem \ref{phcz} extends to the case of solutions to more general
equations of the type
\begin{equation}\label{duecz00} \divo
a\!\left(x,\mathfrak{X}u\right) =\divo (|F|^{p-2}F)\,,
\end{equation}
with the vector field $a : \Omega \times \er^{2n} \to \er^{2n}$ such
that \eqn{ass1}
$$
z \mapsto a(x,z) \qquad \textnormal{satisfies
\rif{growth}-\rif{ell},   for every }x \in \Omega\;,$$ and with
continuous dependence on the $x$-variable, that is \eqn{modulus}
$$
|a(x,z)-a(y,z)|\leq L\omega(d_{cc}(x,y))(\mu+|z|)^{p-1}\,,$$ is
satisfied for every $z \in \er^{2n}$ and $x,y \in \Omega$, where
$\omega: [0,\infty)\to [0,\infty)$ is a continuous, non-decreasing
function such that $\omega(0)=0$. The function $\omega(\cdot)$ is
usually called ``modulus of continuity". The proof of such an extension
is very close to the ones already given in the previous section and
we shall therefore confine ourselves to explaininig the main
differences, which occur in the following points.

When using Lemma \ref{compvmo} we shall consider as a comparison
function $v$ the unique solution of the Dirichlet problem
\eqn{Dir0dopo}
$$
\left\{
    \begin{array}{cc}
    \textnormal{div} \ a(x_0,\XXX v)=0& \qquad \mbox{in }B_R\\
        v= u&\qquad \mbox{on }\partial B_R\,,
\end{array}\right.
$$
where $x_0$ is the center of $B_R$. At this point the statement and
the proof of Lemma \ref{compvmo} are even simpler, as for instance
they do not need the use of Theorem \ref{gehring}; for the ease of
exposition we shall nevertheless refer to the already given proof
although it may be shortened at some points. Anyway we remark that
Theorem \ref{gehring} continues to hold for solutions to
\rif{duecz00} under the considered assumptions. Estimate \rif{fre1}
continues to hold in a different form, that is \rif{fre1dopo} below;
this is due to the fact that the comparison estimate \rif{si00} in
Lemma \ref{compvmo} has to be replaced by
\begin{eqnarray*} \nonumber
\intav_{B_R} |\XXX u - \XXX v|^p\, dx &\leq & c \intav_{B_R} \langle
a(x_0,\XXX u)- a(x_0,\XXX v), \XXX
u - \XXX v\rangle \, dx\\
\nonumber  &=& c\intav_{B_R}  \langle  a(x_0,\XXX u)- a(x,\XXX u),
\XXX u - \XXX v\rangle \, dx\\&& \qquad \qquad +c\intav_{B_R}\langle
|F|^{p-2}F, \XXX u - \XXX v\rangle\, dx =: I+II\;,
\end{eqnarray*}
which holds in view of \rif{Dir0dopo}. The estimation of $I$ will be
done this time using \rif{modulus}, the one for $II$ being exactly
as in \rif{iiii}. This finally yields the estimate
\begin{eqnarray}
\nonumber \intav_{B_{R}} |\XXX u - \XXX v|^p\, dx
&\leq & c_5 \omega^*(2R_0)\intav_{B_{2R}}(\mu + |\XXX u|)^p\, dx \\
& &\qquad \qquad + c_5[1+\omega^*(2R_0)]\left(
\intav_{B_{2R}}|F|^{q_0}\, dx\right)^{p/q_0}\;,\label{fre1dopo}
\end{eqnarray}
where $ \omega^*(\cdot):=[\omega(\cdot)]^{p/(p-1)}.$ Once the
comparison estimate is gained we may proceed as in the proof of
Lemma \ref{intecz} but using the assumption that $\omega^*(200R_0) <
\varepsilon$ instead of $[b]_{100R_0}^*\leq \varepsilon $. Then,
when using the comparison function $v$, it will be defined as the
unique solution to \rif{Dir0dopo} with $20B \equiv B_R$ and $x_0$ is
the center of $20B$, while the use of \rif{fre1dopo} will replace
the use of \rif{fre1}. This will give the proof of the new version
of Lemma \ref{intecz}.

Then, proceeding exactly as in the proof of Theorem \ref{phcz} we
arrive at the following:
\begin{theorem}\label{phczdopo} Let $u \in HW^{1,p}(\Omega)$ be a weak solution to the equation \trif{duecz00}
under the assumptions \trif{ass1}-\trif{modulus} with $2\leq p <4$.
Assume that $F \in L^{q}_{\loc}(\Omega,\er^{2n})$ for some $q>p$;
then $\XXX u\in L^{q}_{\loc}(\Omega,\er^{2n})$. Moreover there
 exists a constant
$c$, depending only on $n,p,\ratio, q$ and the function
$\omega(\cdot)$, such that the inequality \trif{apapd} holds for any
CC-ball $B_R \Subset \Omega$.
\end{theorem}
Again, the dependence on $\omega(\cdot)$ in the a priori estimates
of Theorem \ref{phczdopo} can be replaced as described in Remark
\ref{alt}.
\begin{remark}\label{bmo} Theorem \ref{phcz} admits
an obvious reformulation in the case the coefficient function
$b(\cdot)$ in \rif{duecz} is assumed to have a properly small BMO
norm instead of being locally in VMO. Referring to \rif{defivmo},
the function $b(\cdot)$ is said to have bounded mean oscillations
provided $[b]_{R,\Omega} < \infty$ for some $R>0$. Now it is easy
too see that in Theorem \ref{phcz} assumption \rif{vmo} can be
replaced in order to have the following statement: For every $q<
\infty$ there exists $\varepsilon
>0$ depending only on $n, p , \ratio$ and $q$ such that
$[b]_{R,\Omega} < \varepsilon$ for some $R>0$ implies $\XXX u \in
L^{q}_{\loc}(\Omega,\er^{2n})$. This comes directly from Lemma
\ref{intecz}, where $[b]_{100R_0}^*\leq \varepsilon $, which is
later implied by the VMO condition in the proof of Theorem
\ref{phcz}, is now immediately implied by the global smallness
assumption $[b]_{R,\Omega} < \varepsilon$.

\end{remark}

\end{document}